\documentclass[12pt]{amsart}

\def\Tiny{\font\Tinyfont = cmr10 at 5pt \relax  \Tinyfont}

\usepackage{amsfonts, amsmath, amsthm, amssymb, epsfig, graphics,
  psfrag, latexsym, color, mathtools, mathrsfs, enumitem, subfig, booktabs}
\usepackage[all,cmtip]{xy}
\definecolor{darkblue}{rgb}{0,0,0.4}
\usepackage[colorlinks=true, citecolor=darkblue, filecolor=darkblue, linkcolor=darkblue, urlcolor=darkblue]{hyperref}
\usepackage[all]{hypcap}

\graphicspath{{draws/}{}}

\numberwithin{equation}{section}
\newtheorem{thm}{Theorem}[section]
\newtheorem{theorem}[thm]{Theorem}
\newtheorem{lem}{Lemma}[section]               
\newtheorem{lemma}[lem]{Lemma}

\newtheorem{corollary}[lem]{Corollary}               
\newtheorem{prop}[lem]{Proposition}
\newtheorem{proposition}[lem]{Proposition}
\theoremstyle{definition}

\newtheorem{defn}[lem]{Definition} 
\newtheorem{definition}[lem]{Definition} 
  
\newtheorem{conjecture}[lem]{Conjecture}  
\theoremstyle{remark}     
\newtheorem{rem}{Remark}[section]
\newtheorem{remark}[rem]{Remark}
\newtheorem{exam}[rem]{Example}

\numberwithin{figure}{section}

\newcommand{\Section}[1]{\hyperref[sec:#1]{Section~\ref*{sec:#1}}}
\newcommand{\Subsection}[1]{\hyperref[subsec:#1]{Subsection~\ref*{subsec:#1}}}
\newcommand{\Lemma}[1]{\hyperref[lem:#1]{Lemma~\ref*{lem:#1}}}
\newcommand{\Theorem}[1]{\hyperref[thm:#1]{Theorem~\ref*{thm:#1}}}
\newcommand{\Definition}[1]{\hyperref[def:#1]{Definition~\ref*{def:#1}}}
\newcommand{\Remark}[1]{\hyperref[rem:#1]{Remark~\ref*{rem:#1}}}
\newcommand{\Figure}[1]{\hyperref[fig:#1]{Figure~\ref*{fig:#1}}}
\newcommand{\Conjecture}[1]{\hyperref[conj:#1]{Conjecture~\ref*{conj:#1}}}
\newcommand{\Corollary}[1]{\hyperref[cor:#1]{Corollary~\ref*{cor:#1}}}
\newcommand{\Proposition}[1]{\hyperref[prop:#1]{Proposition~\ref*{prop:#1}}}
\newcommand{\Question}[1]{\hyperref[ques:#1]{Question~\ref*{ques:#1}}}
\newcommand{\Example}[1]{\hyperref[exam:#1]{Example~\ref*{exam:#1}}}

\newcommand{\Equation}[2][{}]{Equation#1~(\ref{eq:#2})}

\newcommand{\Diagram}[2][{}]{Diagram#1~(\ref{eq:#2})}

\newcommand{\Case}[2][{}]{Case#1~(\ref{case:#2})}

\newcommand{\Condition}[2][{}]{Condition#1~(\ref{condition:#2})}
\newcommand{\ConditionCont}[1]{(\ref{condition:#1})}
\newcommand{\Part}[2][{}]{Part#1~(\ref{item:#2})}
\newcommand{\PartCont}[1]{(\ref{item:#1})}
\newcommand{\Item}[1]{(\ref{item:#1})}
\newcommand{\Object}[1]{(\ref{#1})}

\newcommand{\R}{\mathbb{R}}
\newcommand{\Z}{\mathbb{Z}}
\newcommand{\F}{\mathbb{F}}

\newcommand{\N}{\mathbb{N}}
\newcommand{\E}{\mathbb{E}}

\newcommand{\mf}{\mathfrak}

\newcommand{\wt}{\widetilde}
\newcommand{\ol}{\overline}

\newcommand{\del}{\partial}
\newcommand{\sbs}{\subset}
\newcommand{\sbseq}{\subseteq}
\newcommand{\sm}{\setminus}
\renewcommand{\emptyset}{\varnothing}
\newcommand{\interior}{\mathring}

\newcommand{\smas}{\wedge}

\newcommand{\si}{\sigma}

\newcommand{\ep}{\epsilon}

\newcommand{\from}{\colon}
\newcommand{\into}{\hookrightarrow}
\newcommand{\imto}{\looparrowright}
\newcommand{\onto}{\twoheadrightarrow}

\newcommand{\sequence}[3][1]{{#2}_{#1},\allowbreak\dots,\allowbreak{#2}_{#3}}
\newcommand{\tuple}[3][1]{(\sequence[#1]{#2}{#3})}

\newcommand{\set}[2]{\{#1\mid#2\}}
\newcommand{\restrict}[2]{{#1}|_{#2}}

\renewcommand{\th}{^{\text{th}}}

\renewcommand{\hat}{\widehat}

\DeclareMathOperator{\Ob}{Ob}
\DeclareMathOperator{\Id}{Id}
\DeclareMathOperator{\Sq}{Sq}
\DeclareMathOperator{\colim}{colim}

\DeclareMathOperator{\Hom}{Hom}

\DeclareMathOperator*{\fiberprod}{\times}
\DeclareMathOperator*{\fibercup}{\cup}

\newcommand{\Kh}{\mathit{Kh}}
\newcommand{\rKh}{\widetilde{\Kh}}
\newcommand{\KhCx}{\mathit{KC}}
\newcommand{\rKhCx}{\widetilde{\KhCx}}
\newcommand{\Cat}{\mathscr{C}}
\newcommand{\Funky}{\mathscr{F}}
\newcommand{\MorseFlowCat}{\mathscr{C_M}}
\newcommand{\Realize}[2][{}]{|#2|_{#1}}
\newcommand{\Codim}[1]{\langle#1\rangle}

\newcommand{\CubeFlowCat}{\mathscr{C}_C}
\newcommand{\KhFlowCat}{\mathscr{C}_K}
\newcommand{\rKhFlowCat}{\widetilde{\mathscr{C}_K}}
\newcommand{\Moduli}{\mathcal{M}}
\newcommand{\ParaModuli}{\widetilde{\mathcal{M}}}
\newcommand{\Forget}{\mathcal{F}}
\newcommand{\gen}[1]{#1}
\newcommand{\ob}[1]{\mathbf{#1}}
\newcommand{\gr}{\mathrm{gr}}
\newcommand{\intgr}{\gr_{q}}
\newcommand{\homgr}{\gr_{h}}
\newcommand{\expect}{\mathit{exp}}
\newcommand{\KhSpace}{\mathcal{X}_\mathit{Kh}}
\newcommand{\rKhSpace}{\widetilde{\mathcal{X}}_\mathit{Kh}}
\newcommand{\Euclid}[2]{\mathbb{E}^{#2}_{#1}}
\newcommand{\ESpace}[3][{}]{\mathbb{E}_{#1}[#2:#3]}

\newcommand{\Cell}[2][{}]{\mathcal{C}_{#1}(#2)}
\newcommand{\CellPrime}[2][{}]{\mathcal{C}'_{#1}(#2)}
\newcommand{\Tup}{}
\newcommand{\TupV}{\mathbf}
\newcommand{\Frame}{\varphi}
\newcommand{\Cube}{\mathcal{C}}
\newcommand{\imprec}{\leq_1}

\newcommand{\BO}{\mathit{BO}}
\newcommand{\EO}{\mathit{EO}}
\newcommand{\Grassmann}[1]{\mathit{Gr}(#1)}
\newcommand{\AssRes}[2]{D_{#1}(#2)}
\newcommand{\CubeCat}[1]{\underline{2}^{#1}}
\newcommand{\CubeSubCat}[1]{\underline{2}^{#1}\setminus\overline{1}}

\newcommand{\JCat}{\mathcal{J}}
\newcommand{\diff}{\delta}
\newcommand{\vect}{\overline}
\newcommand{\Permu}[1]{{P}_{#1}}
\newcommand{\matching}{-}
\newcommand{\mirror}[1]{m(#1)}
\newcommand{\SWdual}[1]{#1^\vee}

\newcommand{\inout}{*}
\newcommand{\OutIn}{right}
\newcommand{\InOut}{left}

\newcommand{\co}{\colon}
\newcommand{\bdy}{\partial}
\newcommand{\RR}{\R}
\newcommand{\DD}{\mathbb{D}}

\DeclareMathOperator{\ind}{ind}
\DeclareMathOperator{\rank}{rank}
\newcommand{\pt}{\mathrm{pt}}
\newcommand{\CInc}{\mathcal{I}}

\newcommand{\ZZ}{\mathbb{Z}}
\DeclareMathOperator{\cokernel}{coker}

\newcommand{\respectively}{resp.\ }

\begin{document}
\title[A Khovanov stable homotopy type]{A Khovanov stable homotopy type}

\author{Robert Lipshitz}
\thanks{RL was supported by an NSF grant number DMS-0905796 and a Sloan Research Fellowship.}
\email{\href{mailto:lipshitz@math.columbia.edu}{lipshitz@math.columbia.edu}}

\author{Sucharit Sarkar}
\thanks{SS was supported by a Clay Mathematics Institute Postdoctoral Fellowship}
\email{\href{mailto:sucharit@math.columbia.edu}{sucharit@math.columbia.edu}}

\subjclass[2010]{\href{http://www.ams.org/mathscinet/search/mscdoc.html?code=57M25,55P42}{57M25,
    55P42}}

\address{Department of Mathematics, Columbia University, New York, NY 10027}

\date{\today}

\begin{abstract}
  Given a link diagram $L$ we construct spectra $\KhSpace^j(L)$ so
  that the Khovanov homology $\Kh^{i,j}(L)$ is isomorphic to the
  (reduced) singular cohomology
  $\widetilde{H}^{i}(\KhSpace^j(L))$. The construction of
  $\KhSpace^j(L)$ is combinatorial and explicit. We prove that the
  stable homotopy type of $\KhSpace^j(L)$ depends only on the isotopy class
  of the corresponding link.
\end{abstract}

\maketitle

\tableofcontents

\section{Introduction}
In~\cite{Kho-kh-categorification}, Khovanov introduced an elaboration
of the Jones polynomial, now generally called Khovanov homology. The
Khovanov homology of a link $L$ takes the form of a bigraded abelian
group $\Kh^{i,j}(L)$, the homology of a bigraded chain complex which
we denote $\KhCx^{i,j}(L)$. Khovanov homology relates to the Jones
polynomial by taking the graded Euler characteristic:
\[
\chi(\Kh^{i,j}(L))=\sum_{i,j}(-1)^i q^j\rank\Kh^{i,j}(L)=(q+q^{-1})V(L).
\]

In this paper we give a space-level version of Khovanov homology. That
is, we construct a family of (suspension) spectra $\KhSpace^j(L)$ so
that the Khovanov homology of $L$ is the (reduced) singular cohomology
of these spaces. To be precise:
\begin{thm}\label{thm:kh-space}
  Let $L$ be an oriented link diagram. Let
  $\KhSpace(L)=\bigvee_j\KhSpace^j(L)$ 
  be the Khovanov spectrum (\Definition{Kh-space}) and let $\Kh(L)$ be the
  Khovanov homology (\Definition{knot-diag-to-kh-chain-complex}). Then
  the following hold.
  \begin{enumerate}[label=($\Theta$-\arabic*), ref=$\Theta$-\arabic*]
  \item\label{item:thm-refines} The reduced cohomology of
    $\KhSpace^j(L)$ is the Khovanov homology $\Kh^{*,j}(L)$:
    \[
    \wt{H}^i(\KhSpace^j(L))=\Kh^{i,j}(L).
    \]
  \item\label{item:thm-invariant} The stable homotopy type of $\KhSpace^j(L)$
    is an invariant of the isotopy class of the link corresponding to
    $L$ (and is independent of all the other choices in its construction).
  \end{enumerate}
\end{thm}
In particular, the Euler characteristic of our space $\KhSpace^j(L)$
is the coefficient of $q^i$ in the Jones polynomial
$(q+q^{-1})V(L)$. (There is also a reduced version of both theories,
which eliminates the $q+q^{-1}$, but, as with Khovanov homology itself, the
relationship between the reduced and un-reduced homotopy types is not
entirely trivial.)

Our work is inspired by a question posed by Cohen-Jones-Segal fifteen
years ago~\cite{CJS-gauge-floerhomotopy}: is Floer homology the ordinary
(singular) homology of some naturally associated space (or rather,
spectrum)?  Since then, such a space has been constructed in several
interesting
cases~\cite{Man-gauge-swspectrum,Man-gauge-gluing,KM-gauge-swspectrum,Kragh:transfer-spectra,Cohen10:Floer-htpy-cotangent},
but the question remains open in general.

In fact, Cohen-Jones-Segal went further: they proposed a construction
of such a spectrum using the higher-dimensional Floer moduli
spaces. These moduli spaces are encoded in a structure they call a
\emph{framed flow category}; and to any framed flow category they
showed how to associate a CW complex. A framed flow category is a
category $\Cat$, together with a $\ZZ$-valued grading function $\gr\co
\Ob(\Cat)\to\ZZ$ such that for $x,y\in \Ob(\Cat)$, $\Hom(x,y)$ is a
compact $(\gr(x)-\gr(y)-1)$-dimensional manifold with corners,
satisfying certain compatibility conditions with respect to the
composition maps and corner structure (see
Definitions~\ref{def:flow-cat} and~\ref{def:framed-flow-cat}). The
realization $\Realize{\Cat}$ of a framed flow category $\Cat$ is not
the usual geometric realization of a (topological) category, but
rather a generalization of the Dold-Thom construction. For example, in the
special case that $\Cat=\{x_1,\dots,x_k,y_1,\dots,y_\ell\}$ with
$\gr(x_i)=n+1$, $\gr(y_j)=n$, the realization $\Realize{\Cat}$ is obtained from a
wedge sum of $\ell$ $n$-spheres, one for each $y_j$, by attaching $k$
$(n+1)$-disks, one for each $x_i$. The attaching map
$S^n=\bdy\DD^{n+1}_{x_i}\to \bigvee_{j=1}^\ell S^n_{y_j}\to S^n_{y_j}$ is
determined by an integer, the degree; this degree is the signed number
of points in the framed $0$-manifold $\Hom(x_i,y_j)$. One way to
construct the attaching map is to embed $\amalg_j\Hom(x_i,y_j)$ into $\bdy
\DD^{n+1}$, use the framing to identify a neighborhood of
$\Hom(x_i,y_j)$ with $\Hom(x_i,y_j)\times \DD^n$, project
$\Hom(x_i,y_j)\times \DD^n\to S^n_{y_j}$, and collapse the complement
of these neighborhoods to the basepoint. The general case of the
Cohen-Jones-Segal construction follows along these lines; see
\Definition{flow-gives-space}.

So, to construct a Khovanov stable homotopy type, it suffices to
construct a flow category for Khovanov homology. Our Khovanov flow
category has one object for each generator of the standard Khovanov
complex (which is reviewed in \Section{res-config}), and the grading
is the homological grading on Khovanov homology. Let $\diff$ denote
the differential on the Khovanov complex. If $\gr(x)-\gr(y)=1$ and $x$
occurs in $\delta(y)$ with coefficient $n_{xy}$, then we take
$\Hom(x,y)$ to be a disjoint union of $|n_{xy}|$ points, all framed
positively or negatively according to the sign of $n_{xy}$. The spaces
$\Hom(x,y)$ with $\gr(x)-\gr(y)>1$ are constructed inductively. The
boundary of $\Hom(x,y)$ is already determined by lower-dimensional
$\Hom$ spaces, and it turns out that $\bdy\Hom(x,y)$ is a disjoint
union of spheres. We choose $\Hom(x,y)$ to be a disjoint union of
disks. (Keeping track of the corner structure, these disks are, in
fact, permutohedra; see \Example{permutohedron} and
\Lemma{cube-cat-struct}.) If $\gr(x)-\gr(y)>2$ then there is a unique
way to construct such disks starting with their boundaries; if
$\gr(x)-\gr(y)=2$ then the construction depends on a choice, which we
call the \emph{ladybug matching} (see \Section{1D}). 

To prove inductively that it is possible to choose the $\Hom$ spaces
to be disks, that is, their expected boundaries (unions of
products of lower dimensional $\Hom$ spaces) are spheres, and to
specify the framings of these disks, we use an auxiliary flow
category, called the \emph{cube flow category}
(\Definition{cube-flow-category}). The cube flow category has one
object for each vertex of the cube $\{0,1\}^N$, and its $\Hom$ spaces
are permutohedra. During our inductive construction of the $\Hom$
spaces for the Khovanov flow category, we impose an additional
requirement that these spaces should admit covering maps to the
corresponding $\Hom$ spaces of the cube flow category. For dimensions
up to $2$, we check explicitly that it is possible to construct $\Hom$ spaces with such covering maps
(Sections~\ref{sec:0D}--\ref{sec:ind-3-bdy}). For the $n$-dimensional
$\Hom$ spaces with $n>2$ (\Section{nd-moduli-spaces}), by the
inductive construction, their expected boundaries already admit
covering maps to the boundary of the $n$-dimensional permutohedron,
which, being $S^{n-1}$, is simply-connected; therefore, their expected
boundaries are disjoint unions of spheres and the covering map is
trivial. Then there is a unique way to define the $n$-dimensional
$\Hom$ spaces to be disks, and they admit covering maps to the
$n$-dimensional permutohedron. The use of the cube flow category in
this construction may be viewed as a space-level analogue of the fact
that the Khovanov complex lies over the cube $\{0,1\}^N$.

In particular, our construction gives a CW complex whose cells are in
one-to-one correspondence with generators of the standard Khovanov
chain complex.  Furthermore, it also follows that the Khovanov flow
category decomposes as a disjoint union of subcategories, one for each
quantum grading, and so its realization decomposes as a wedge sum of
spaces, one for each quantum grading. Like the construction of
Khovanov homology, our construction of the Khovanov flow category is
entirely combinatorial; although we use some language from Morse
theory, no Morse theory (not to mention Floer theory) is required.

The cohomology of a spectrum does not carry a cup product, but it does
carry stable cohomology operations. In the sequel~\cite{RS-steenrod}
we give an explicit computation of the operation $\Sq^2$ on Khovanov
homology coming from our stable type. This operation is nontrivial
for many simple knots, such as the torus knot $T_{3,4}$.  This implies
in particular that our Khovanov stable type is not simply a wedge
sum of Moore spaces.
Indeed, Seed has extended the results from~\cite{RS-steenrod} to find pairs of
links with isomorphic Khovanov homology but distinct Khovanov homotopy
types~\cite{See-kh-squares}. So, $\KhSpace(L)$ is a strictly
stronger invariant than $\Kh(L)$.

Rasmussen constructed a slice genus bound, called the $s$-invariant,
using Khovanov homology \cite{Ras-kh-slice}. Using the Khovanov
homotopy type, we produce a family of generalizations of the
$s$-invariant \cite{RS-s-invariant}; each of them is a slice genus
bound, and we show that at least one of them is a stronger bound.

A different kind of Khovanov stable homotopy type has been constructed by
Everitt-Turner~\cite{ET-kh-spectrum}. Their construction gives spaces
whose homotopy groups (rather than homology groups) are Khovanov
homology. It is shown in~\cite{ELST-trivial} that the spaces
constructed in~\cite{ET-kh-spectrum} are products of Eilenberg-MacLane
spaces; in particular, they are determined by the Khovanov
homology. Since the first version of this paper, Hu-Kriz-Kriz have
given another construction of a Khovanov stable homotopy type (in our
sense)~\cite{HKK-Kh-htpy}. Although their construction uses quite
different techniques, it seems natural to conjecture that the spaces
constructed in~\cite{HKK-Kh-htpy} are stably homotopy equivalent to the
spaces constructed here. (In particular, their construction also uses
the ladybug matching from \Section{1D} as a key choice.) 

This paper is organized as follows. The first two sections are
background: \Section{res-config} introduces notation related to the
Khovanov chain complex, that will be used when constructing the
Khovanov flow category, and \Section{cjs-puppe} discusses flow
categories themselves and how one produces a CW complex from a flow
category. \Section{cube-flow} introduces a simple flow
category that will be a kind of local model for the Khovanov flow
category. The Khovanov flow category itself is constructed in
\Section{Kh-flow}. Invariance of the Khovanov homotopy type is
proved in \Section{invariance}. Like Khovanov homology, the Khovanov
homotopy types trivially satisfy an
unoriented skein triangle; this is explained in \Section{skein}. The
spaces for reduced Khovanov homology are introduced in
\Section{reduced}. We describe with some examples in
\Section{examples}, including the computation of $\KhSpace(L)$ for any
alternating link $L$, in terms of the Khovanov homology $\Kh(L)$. We
conclude in \Section{speculations} with some speculations.

\textit{Acknowledgments.} The authors are indebted to Ralph Cohen,
Christopher Douglas and Ciprian Manolescu for many clarifying
conversations on the Cohen-Jones-Segal construction. We thank Mikhail
Khovanov and Cotton Seed for helpful comments on a previous version. Finally, we thank the referee for many helpful suggestions.

\section{Resolution configurations and Khovanov homology}\label{sec:res-config}
The original definition of Khovanov homology
\cite{Kho-kh-categorification} is in terms of Kauffman states, and
transitions between consecutive Kauffman states (with respect to the
obvious partial order). Our Khovanov homotopy type will be defined in
terms of more complicated families of Kauffman states. This section is
devoted to developing terminology to discuss these sequences of
Kauffman states. In \Definition{res-config}, we introduce
\emph{resolution configurations}; essentially, a resolution
configuration is a Kauffman state together with a marking of the
crossings at which one took the $0$-resolution. (Resolution
configurations have appeared informally elsewhere in the
literature---for instance, in~\cite{OSzR-kh-oddkhovanov}.)
We then introduce
several operations on resolution configurations, a notion of a labeled
resolution configuration, and a partial order on labeled resolution
configurations. Most of these concepts are illustrated in
Figures~\ref{fig:res-config-from-knot}
and~\ref{fig:res-config-all}. We conclude the section by restating the
definition of Khovanov homology in this language.

\begin{definition}\label{def:res-config}
  A \emph{resolution configuration} $D$ is a pair $(Z(D),A(D))$, where
  $Z(D)$ is a set of pairwise-disjoint embedded circles in $S^2$, and
  $A(D)$ is a totally ordered collection of disjoint arcs embedded in
  $S^2$, with $A(D)\cap Z(D)=\bdy A(D)$.

  We call the number of arcs in $A(D)$ the \emph{index} of the
  resolution configuration $D$, and denote it by $\ind(D)$.

  We sometimes abuse notation and write $Z(D)$ to mean $\bigcup_{Z\in
    Z(D)}Z$ and $A(D)$ to mean $\bigcup_{A\in A(D)} A$.
\end{definition}

Occasionally, we will describe the total order on
$A(D)$ by numbering the arcs: a lower numbered arc precedes a
higher numbered one.

\begin{definition}\label{def:knot-diag-to-res-config}
  Given a link diagram $L$ with $n$ crossings, an ordering of the
  crossings in $L$, and a vector $v\in \{0,1\}^n$ there is an
  associated resolution configuration $\AssRes{L}{v}$ gotten by taking
  the resolution of $L$ corresponding to $v$ (i.e., taking the
  $0$-resolution at the $i\th$ crossing if $v_i=0$, and the
  $1$-resolution otherwise) and then placing arcs corresponding to
  each of the crossings labeled by $0$'s in $v$ (i.e., at the $i\th$
  crossing if $v_i=0$); see \Figure{res-config-from-knot}.
\end{definition}

Therefore, $n-\ind(\AssRes{L}{v})$ equals $|v|=\sum v_i$, the
(Manhattan) norm of $v$.

\captionsetup[subfloat]{font=normalsize,labelformat=simple,labelsep=colon,width={0.45\textwidth}}
\begin{figure}
  \centering
  \subfloat[The $0$-resolution and the
  $1$-resolution.]{
    \psfrag{0}{$0$}
    \psfrag{1}{$1$}
    \includegraphics[width=0.45\textwidth]{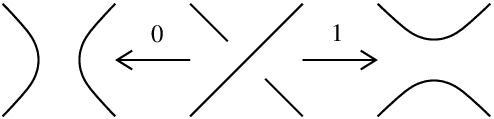}}\\
  \captionsetup[subfloat]{width={0.35\textwidth}}
  \subfloat[A knot diagram $K$ for the trefoil (the three crossings
  are numbered).]{ \tiny
    \psfrag{a}{$1$} \psfrag{b}{$2$} \psfrag{c}{$3$} 
    \hspace{.05\textwidth}\includegraphics[width=0.25\textwidth]{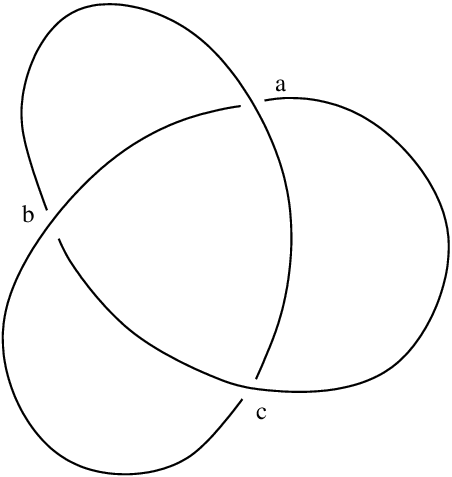}}
  \hspace{0.1\textwidth} \subfloat[The resolution
  configuration $\AssRes{K}{(0,1,0)}$.]{
    \hspace{.05\textwidth}\includegraphics[width=0.25\textwidth]{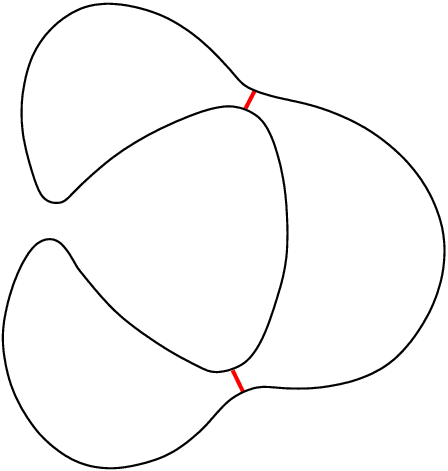}}
  \caption[Resolution configurations from a knot diagram.]{\textbf{Resolution configurations from a knot diagram.}}
  \label{fig:res-config-from-knot}
\end{figure}

\captionsetup[subfloat]{width={0.3\textwidth}}
\begin{figure}
  \centering
  \subfloat[The starting resolution configuration $D$. The circles
  and arcs are numbered.]{\label{fig:res-config-all-a}\tiny      
    \psfrag{z1}{$Z_1$}
    \psfrag{z2}{$Z_2$}
    \psfrag{z3}{$Z_3$}
    \psfrag{z4}{$Z_4$}
    \psfrag{z5}{$Z_5$}
    \psfrag{a1}{$A_1$}
    \psfrag{a2}{$A_2$}
    \psfrag{a3}{$A_3$}
    \psfrag{a4}{$A_4$}
    \includegraphics[width=0.3\textwidth]{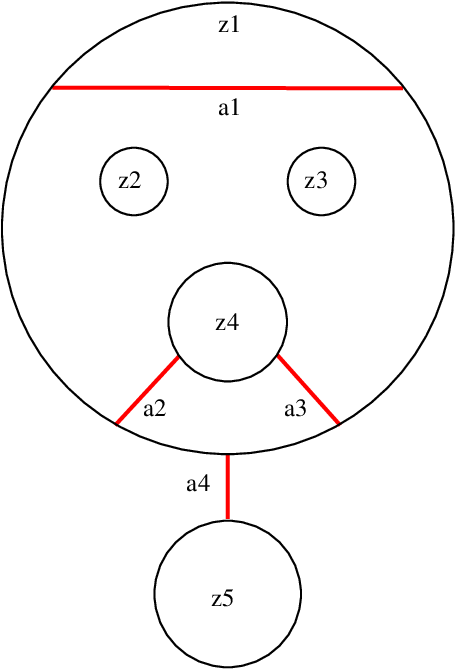}}
  \hspace{0.03\textwidth} 
  \subfloat[The core $c(D)$.]{\label{fig:res-config-all-b}
    \includegraphics[width=0.3\textwidth]{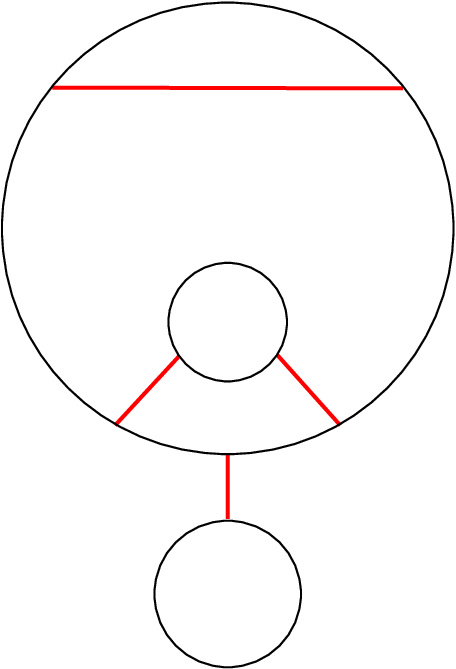}}
  \hspace{0.03\textwidth}
  \subfloat[The surgery $E=s_{\{A_2,A_3\}}(D)$.]{\label{fig:res-config-all-c}
    \includegraphics[width=0.3\textwidth]{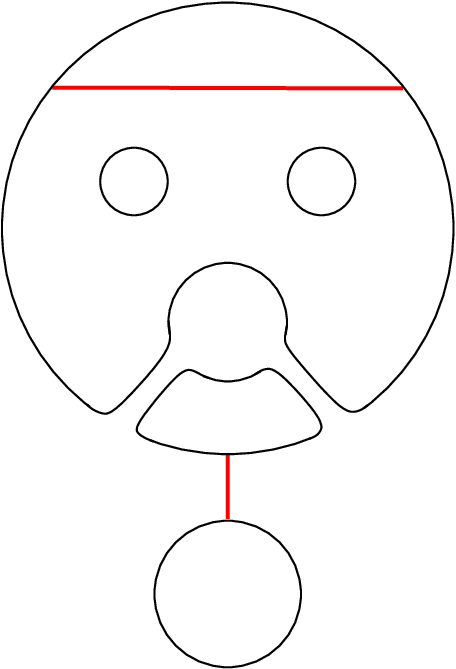}}\\
  \subfloat[The resolution configuration $D\sm E$.]{\label{fig:res-config-all-d}
    \includegraphics[width=0.3\textwidth]{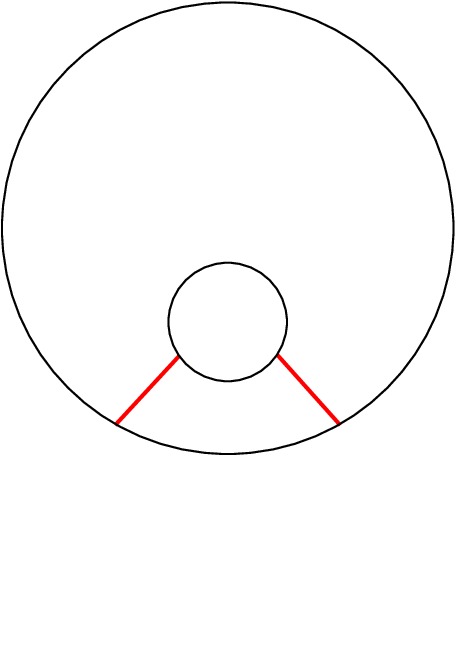}}
  \hspace{0.03\textwidth} 
  \subfloat[The resolution configuration $D\cap E=D\sm(D\sm E)$.]{\label{fig:res-config-all-e}
    \includegraphics[width=0.3\textwidth]{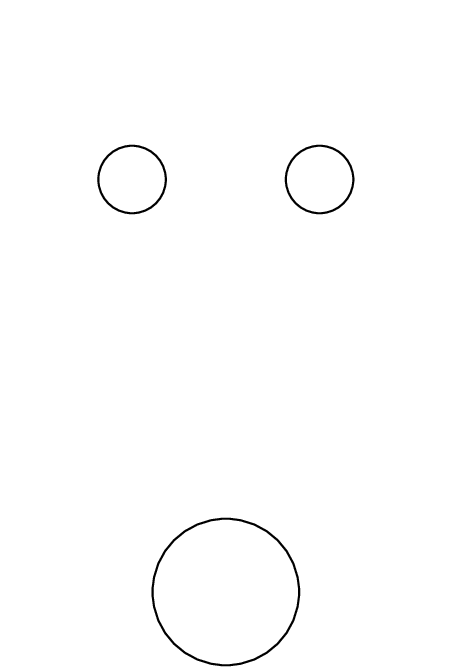}}
  \hspace{0.03\textwidth}
  \subfloat[The dual resolution configuration $D^*$.]{\label{fig:res-config-all-f}
    \includegraphics[width=0.3\textwidth]{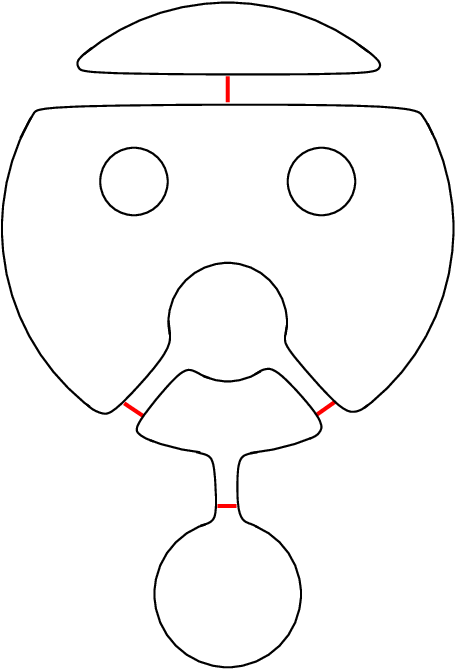}}
  \caption[An illustration of some resolution configuration notations.]{\textbf{An illustration of some resolution configuration notations.}}
  \label{fig:res-config-all}
\end{figure}

\begin{definition}
  Given resolution configurations $D$ and $E$ there is a new
  resolution configuration $D\setminus E$ defined by
  \[
  Z(D\setminus E)=Z(D)\setminus Z(E) \qquad A(D\setminus
  E)=\set{A\in A(D)}{\forall Z\in Z(E)\co\del A\cap Z=\emptyset}.
  \]
  Let $D\cap E=D\sm(D\sm E)$. See \Figure{res-config-all}.
\end{definition}

Note that $Z(D\cap E)=Z(E\cap D)$ and $A(D\cap E)=A(E\cap D)$;
however, the total orders on $A(D\cap E)$ and $A(E\cap D)$ could be
different.

\begin{definition}\label{def:core-res-config}
  The \emph{core $c(D)$} of a resolution configuration $D$ is the
  resolution configuration obtained from $D$ by deleting all the
  circles in $Z(D)$ that are disjoint from all the arcs in $A(D)$. See
  \Figure{res-config-all-b}.

  A resolution configuration $D$ is called \emph{basic} if $D=c(D)$,
  i.e., if every circle in $Z(D)$ intersects an arc in $A(D)$. 
\end{definition}

\begin{definition}\label{def:surger-res-config}
  Given a resolution configuration $D$ and a subset $A'\subseteq A(D)$
  there is a new resolution configuration $s_{A'}(D)$, the
  \emph{surgery of $D$ along $A'$}, obtained as follows. The circles
  $Z(s_{A'}(D))$ of $s_{A'}(D)$ are obtained by performing embedded
  surgery along the arcs in $A'$; in other words, $Z(s_{A'}(D))$ is
  obtained by deleting a neighborhood of $\bdy A'$ from $Z(D)$ and
  then connecting the endpoints of the result using parallel
  translates of $A'$. The arcs of $s_{A'}(D)$ are the arcs of $D$ not
  in $A'$, i.e., $A(s_{A'}(D))=A(D)\setminus A'$. See
  \Figure{res-config-all-c}.

  Let $s(D)=s_{A(D)}(D)$ denote the maximal surgery on $D$.
\end{definition}

\begin{lem}\label{lem:res-config-surgery}
  If a resolution configuration $E$ is obtained from a resolution
  configuration $D$ by a surgery, then $D\sm E$ is a basic resolution
  configuration, $E\sm D=s(D\sm E)$ and $D\cap E=E\cap D$.
\end{lem}

\begin{proof}
This is immediate from the definitions.
\end{proof}

\begin{definition}\label{def:dual-res-config}
  Associated to a resolution configuration $D$ is a \emph{dual
    resolution configuration $D^*$}, defined as follows. The circles
  $Z(D^*)$ are obtained from $Z(D)$ by performing embedded surgery
  according to the arcs $A(D)$; i.e., $Z(D^*)=Z(s(D))$. The arcs
  $A(D^*)$ are duals to the arcs $A(D)$. That is, write
  $A(D)=A_1\cup\dots\cup A_n$. For each $i$, choose an arc $A_i^*$ in
  a neighborhood of $A_i$, with boundary on $Z(D^*)$, and intersecting
  $A_i$ once. Then $A(D^*)=A_n^*\cup\dots\cup A_1^*$. See
  \Figure{res-config-all-f}.
\end{definition}

There is a forgetful map from resolution configurations to graphs, by
collapsing the circles:
\begin{definition}\label{def:graph-from-res-config}
  A resolution configuration $D$ specifies a graph $G(D)$ with one
  vertex for each element of $Z(D)$ and an edge for each element of
  $A(D)$. Given an element $Z\in Z(D)$ (\respectively $A\in A(D)$) we
  will write $G(Z)$ (\respectively $G(A)$) for the corresponding vertex
  (\respectively edge) of $G(D)$.

  A \emph{leaf} of a resolution configuration $D$ is a circle $Z\in
  Z(D)$ so that $G(Z)$ is a leaf of $G(D)$. A \emph{co-leaf} of $D$ is
  an arc $A\in A(D)$ so that one endpoint of the dual arc $A^*\in
  A(D^*)$ is a leaf of the dual configuration $D^*$. For example, in
  \Figure{res-config-all-a}, the circle $Z_5$ is a leaf of $D$, while
  the arc $A_1$ is a co-leaf of $D$.
\end{definition}

\begin{definition}
  A \emph{labeled resolution configuration} is a pair $(D,\gen{x})$ of
  a resolution configuration $D$ and a labeling $\gen{x}$ of each
  element of $Z(D)$ by either $x_+$ or $x_-$.
\end{definition}

\begin{definition}\label{def:prec}
  There is a partial order $\prec$ on labeled resolution
  configurations defined as follows. We declare that $(E,\gen{y})\prec
  (D,\gen{x})$ if:
  \begin{enumerate}
  \item The labelings $\gen{x}$ and $\gen{y}$ induce the same labeling
    on $D\cap E=E\cap D$. 
  \item $D$ is obtained from $E$ by surgering along a single arc of
    $A(E)$. In particular, either:
    \begin{enumerate}
    \item\label{case:prec-split}
      $Z(E\sm D)$ contains exactly one circle, say $Z_i$, and $Z(D\sm
      E)$ contains exactly two circles, say $Z_j$ and $Z_k$, or
    \item\label{case:prec-merge} 
      $Z(E\sm D)$ contains exactly two circles, say $Z_i$ and $Z_j$,
      and $Z(D\sm E)$ contains exactly one circle, say $Z_k$.
    \end{enumerate}
  \item In \Case{prec-split}, either
    $\gen{y}(Z_i)=\gen{x}(Z_j)=\gen{x}(Z_k)=x_-$ or $\gen{y}(Z_i)=x_+$
    and $\{\gen{x}(Z_j),\gen{x}(Z_k)\}=\{x_+,x_-\}$.

    In \Case{prec-merge}, either
    $\gen{y}(Z_i)=\gen{y}(Z_j)=\gen{x}(Z_k)=x_+$ or
    $\{\gen{y}(Z_i),\gen{y}(Z_j)\}=\{x_-,x_+\}$ and $\gen{x}(Z_k)=x_-$.
  \end{enumerate}
  Now, $\prec$ is defined to be the transitive closure of this
  relation.
\end{definition}

Morally, the three conditions of \Definition{prec} say that
$(D,\gen{x})$ is obtained from $(E,\gen{y})$ by applying one of the
edge maps in the Khovanov cube;
cf.~\Definition{knot-diag-to-kh-chain-complex}.

\begin{definition}
  A \emph{decorated resolution configuration} is a triple
  $(D,\gen{x},\gen{y})$ where $D$ is a resolution configuration and
  $\gen{x}$ (\respectively $\gen{y}$) is a labeling of each component of
  $Z(s(D))$ (\respectively $Z(D)$) by an element of $\{x_+,x_-\}$, such
  that: $(D,\gen{y})\preceq (s(D),\gen{x})$.
  
  Associated to a decorated resolution configuration
  $(D,\gen{x},\gen{y})$ is the poset $P(D,\gen{x},\gen{y})$ consisting
  of all labeled resolution configurations $(E,\gen{z})$ with
  $(D,\gen{y})\preceq (E,\gen{z})\preceq (s(D),\gen{x})$.
\end{definition}

\begin{defn}\label{def:decorated-res-config-dual}
  The \emph{dual} of a decorated resolution configuration
  $(D,\gen{x},\gen{y})$ is the decorated resolution configuration
  $(D^*,\gen{y}^*,\gen{x}^*)$, where $D^*$ is the dual of $D$, the
  labelings $\gen{x}$ and $\gen{x}^*$ are dual, i.e.\ they disagree on
  every circle in $Z(D^*)=Z(s(D))$, and the labelings $\gen{y}$ and
  $\gen{y}^*$ are dual, i.e.\ they disagree on every circle in
  $Z(D)=Z(s(D^*))$.
\end{defn}

\begin{lem}\label{lem:res-config-dual}
  For any decorated resolution configuration $(D,\gen{x},\gen{y})$,
  the poset $P(D^*,\gen{y}^*,\gen{x}^*)$ is the reverse of the poset
  $P(D,\gen{x},\gen{y})$.
\end{lem}

\begin{proof}
  When $\ind(D)=1$, this is immediate from \Definition{prec}. The
  general case reduces to the index $1$ case since $\prec$ is defined
  as a transitive closure.
\end{proof}

The following is a technical lemma that will be useful to us for a
later case analysis.

\begin{lem}\label{lem:technical-res-config-leaf}
  Let $(D,\gen{x},\gen{y})$ be a decorated resolution diagram. Let
  $Z_1\in Z(D)$ be a leaf of $D$, and let $A_1\in A(D)$ be the arc
  one of whose endpoints lies on $Z_1$. Let $Z_2\in Z(D)$ be the
  circle containing the other endpoint of $A_1$, let $Z^*_1\in Z(s(D))$ be the
  circle which contains both the endpoints of $A^*_1$, and let
  $Z_2^*\in Z(s_{A(D)\sm \{A_1\}}(D))\sm\{Z_1\}$ be the unique circle
  that contains an endpoint of $A_1$. 

  Let $D'$ be the resolution configuration obtained from $D$ by
  deleting $Z_1$ and $A_1$. Let $(D',\gen{x}',\gen{y}')$ be the
  following decorated resolution configuration: The labeling
  $\gen{y}'$ on $Z(D')$ is induced from the labeling $\gen{y}$ on
  $Z(D)$. The labeling $\gen{x}'$ on $Z(s(D'))\sm\{Z^*_2\}$ is induced
  from the labeling $\gen{x}$ on $Z(s(D))\sm\{Z^*_1\}$. If
  $\gen{y}(Z_1)=x_+$, set $\gen{x}'(Z^*_2)=\gen{x}(Z^*_1)$; otherwise,
  set $\gen{x}'(Z^*_2)=x_+$.

  Then $P(D,\gen{x},\gen{y})=P(D',\gen{x}',\gen{y}')\times\{0,1\}$,
  where $\{0,1\}$ is the two-element poset with $0\prec 1$.
\end{lem}
\begin{proof}
  We will produce a map from $P(D,\gen{x},\gen{y})$ to
  $P(D',\gen{x}',\gen{y}')\times\{0,1\}$ that will induce an
  isomorphism of posets. Let $(E,\gen{z})\in P(D,\gen{x},\gen{y})$
  with $E=s_A(D)$ for some $A\sbseq A(D)$. We will produce
  $((E',\gen{z}'),i)\in P(D',\gen{x}',\gen{y}')\times\{0,1\}$.

  There are two cases.
  \begin{enumerate}
  \item $A_1\notin A$. Set $i=0$ and $E'=s_A(D')$. The labeling
    $\gen{z}'$ on $Z(s_A(D'))$ is induced from the labeling $\gen{z}$
    on $Z(s_A(D))=Z(s_A(D'))\amalg\{Z_1\}$.
  \item $A_1\in A$. Set $i=1$ and $E'=s_{A\sm\{A_1\}}(D')$. Let
    $\{Z_A\}=Z(s_A(D))\sm Z(s_{A\sm\{A_1\}}(D'))$ and let
    $\{Z'_A\}=Z(s_{A\sm\{A_1\}}(D'))\sm Z(s_A(D))$. The labeling
    $\gen{z}'$ on $Z(s_{A\sm\{A_1\}}(D'))\cap Z(s_A(D))$ is induced
    from $\gen{z}$. It only remains to define $\gen{z}'(Z'_A)$. If
    $\gen{y}(Z_1)=x_+$, set $\gen{z}'(Z'_A)=\gen{z}(Z_A)$; otherwise,
    set $\gen{z}'(Z'_A)=x_+$.
\end{enumerate}

It is a straightforward, albeit elaborate, check that this map
$P(D,\gen{x},\gen{y})\to P(D',\gen{x}',\gen{y}')\times\{0,1\}$ induces
an isomorphism of posets.
\end{proof}

We conclude this section by giving the definition of the Khovanov
chain complex from \cite{Kho-kh-categorification} in the language of
resolution configurations and the partial order $\prec$:
\begin{definition}\label{def:knot-diag-to-kh-chain-complex}
  Given an oriented link diagram $L$ with $n$ crossings and an
  ordering of the crossings in $L$, the \emph{Khovanov chain complex}
  is defined as follows.

  The chain group $\KhCx(L)$ is the $\Z$-module freely
  generated by labeled resolution configurations of the form
  $(\AssRes{L}{u},\gen{x})$ for $u\in\{0,1\}^n$. The chain group $\KhCx(L)$
  carries two gradings, a \emph{homological grading} $\homgr$ and a
  \emph{quantum grading} $\intgr$, defined as follows:
  \begin{align*}
  \homgr((\AssRes{L}{u},\gen{x}))&=-n_-+|u|,\\
  \intgr((\AssRes{L}{u},\gen{x}))&=n_+-2n_-+|u|+\#\set{Z\in
    Z(\AssRes{L}{u})}{\gen{x}(Z)=x_+}\\
  &\qquad\qquad{}-\#\set{Z\in Z(\AssRes{L}{u})}{\gen{x}(Z)=x_-}.
  \end{align*}
  Here $n_+$ denotes the number of positive crossings in $L$; and
  $n_-=n-n_+$ denotes the number of negative crossings.

  The differential preserves the quantum grading, increases the
  homological grading by $1$, and is defined as
  \[
  \diff (\AssRes{L}{v},\gen{y}) = \!\!\!\!\!
  \sum_{\substack{(\AssRes{L}{u},\gen{x})\\|u|=|v|+1\\(\AssRes{L}{v},\gen{y})\prec
      (\AssRes{L}{u},\gen{x})}}\!\!\!\!\!
  (-1)^{s_0(\Cube_{u,v})}(\AssRes{L}{u},\gen{x}),
  \]
  where $s_0(\Cube_{u,v})\in\F_2$ is defined as follows: if
  $u=(\ep_1,\dots,\ep_{i-1},1,\ep_{i+1},\dots,\ep_n)$ and
  $v=(\ep_1,\dots,\ep_{i-1},0,\ep_{i+1},\dots,\ep_n)$, then
  $s_0(\Cube_{u,v})=(\ep_1+\dots+\ep_{i-1})$; see also
  \Definition{sign-assignment}.
\end{definition}

\section{The Cohen-Jones-Segal construction}\label{sec:cjs-puppe}

In Morse theory, under suitable hypotheses the moduli
spaces of flows admit compactifications as manifolds with corners,
where the codimension-$i$ boundary consists of $i$-times broken flows;
in \Section{Kh-flow} we will construct a collection of spaces
refining the Khovanov complex, satisfying the same
property. In \cite{CJS-gauge-floerhomotopy}, Cohen-Jones-Segal
proposed a convenient language for encoding this information, which
they called a flow category. They went on to explain how one can build
a spectrum from a flow category; in the Morse theory case, their construction
recovers the suspension spectrum of the base manifold. 

This section is devoted to reviewing the Cohen-Jones-Segal
construction, and the basic
definitions underlying it. In \Section{mfld-w-corners} we
review basic notions about manifolds with corners, and in
\Section{flow-cats} we recall the definition of a (framed)
flow category, and how these arise from Morse
theory. \Section{flow-to-space} gives a reformulation of the
Cohen-Jones-Segal procedure for building a space from a flow
category.

The flow category for Khovanov homology is built as a cover of a
simpler flow category; this notion of a cover is introduced in
\Section{cover-sub-quot-flow}. The proof of invariance of the Khovanov
stable homotopy type will make extensive use of subcomplexes and
quotient complexes, basic results on which are also given in
\Section{cover-sub-quot-flow}.

\subsection{\texorpdfstring{$\Codim{n}$}{<n>}-manifolds}\label{sec:mfld-w-corners}

\captionsetup[subfloat]{width=0.25\textwidth}
\begin{figure}
  \centering
  \subfloat[A manifold with corners which is not a manifold with faces.]{\hspace{.025\textwidth}\includegraphics[width=.2\textwidth]{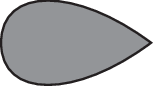}}
  \hspace{.05\textwidth}
  \subfloat[A manifold with faces which is not a $\Codim{2}$-manifold.]{\hspace{.025\textwidth}\includegraphics[width=.2\textwidth]{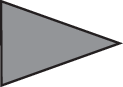}}
  \hspace{.05\textwidth}
  \subfloat[A $\Codim{2}$-manifold. $\partial_1$ is dotted and $\partial_2$ is solid.]{\hspace{.025\textwidth}\includegraphics[width=.2\textwidth]{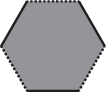}}
  \caption[Manifolds with corners, manifolds with faces, and $\Codim{n}$-manifolds.]{\textbf{Manifolds with corners, manifolds with faces, and $\Codim{n}$-manifolds.}}
  \label{fig:mfld-w-corners}
\end{figure}

We start with some preliminaries on $\Codim{n}$-manifolds, following
\cite{Lau-top-cobordismcorners} and \cite{Jan-top-o(n)manifolds}.
\begin{defn}\label{def:codim-n-mfld}
  Let $\RR_+=[0,\infty)$.

  A \emph{$k$-dimensional manifold with corners} is a (possibly empty)
  topological space $X$ along with a maximal atlas, where an atlas is
  a collection of charts $(U,\phi)$, where $U$ is an open subset of
  $X$ and $\phi$ is a homeomorphism from $U$ to an open subset of
  $(\RR_+)^k$, such that
  \begin{itemize}
  \item $X=\bigcup U$ and
  \item given charts $(U,\phi)$ and $(V,\psi)$, the map $\phi\circ
    \psi^{-1}$ is a $C^\infty$ diffeomorphism from $\psi(U\cap V)$ to
    $\phi(U\cap V)$.
  \end{itemize}

  For $x\in X$, let $c(x)$ denote the number of coordinates in
  $\phi(x)$ which are $0$, for any chart $(U,\phi)$ with $U\ni x$. The
  \emph{codimension-$i$ boundary} of $X$ is the subspace $\{x\in X\mid
  c(x)= i\}$; the usual boundary $\del X$ is the closure of the
  codimension-$1$ boundary. A \emph{connected face} is the closure of
  a connected component of the codimension-$1$ boundary of $X$. A
  \emph{face} of $X$ is a (possibly empty) union of disjoint connected
  faces of $X$.

  A \emph{$k$-dimensional manifold with faces} is a $k$-dimensional
  manifold with corners $X$ such that every $x\in X$ belongs to
  exactly $c(x)$ connected faces of $X$. It is easy to see that any
  face of a manifold with faces is also a manifold with faces. A
  \emph{$k$-dimensional $\Codim{n}$-manifold} is a $k$-dimensional
  manifold with faces $X$ along with an ordered $n$-tuple
  $(\del_1X,\dots,\del_nX)$ of faces of $X$ such that
  \begin{itemize}
  \item $\bigcup_i\del_iX=\del X$ and
  \item for all distinct $i,j$, $\del_iX\cap\del_jX$ is a face of both
    $\del_iX$ and $\del_jX$.
  \end{itemize}

  See \Figure{mfld-w-corners}.
\end{defn}

\begin{exam}\label{exam:permutohedron}
  %Any polytope is a manifold with faces. In particular, the {SS: Not true}
  The \emph{permutohedron $\Permu{n+1}$} is the following $n$-dimensional
  $\Codim{n}$-manifold. As a polytope, it is the convex hull in
  $\R^{n+1}$ of the $(n+1)!$ points $(\si(1),\dots,\si(n+1))$, where
  $\si$ varies over the permutations of $\{1,\dots,n+1\}$. (See
  \cite[Example 0.10]{Zie-top-polytopes} for more details.) The faces
  of $\Permu{n+1}$ are in bijection with subsets $\emptyset\neq
  S\sbs\{1,\dots,n+1\}$: the face corresponding to such a subset $S$
  is the intersection of $\Permu{n+1}$ and plane $\sum_{i\in
    S}x_i=\frac{k(k+1)}{2}$ in $\R^{n+1}$, where $k$ is the
  cardinality of $S$. Define $\del_i\Permu{n+1}$ to be the union of
  faces corresponding to subsets of cardinality $i$.
\end{exam}

\begin{defn}
  A smooth map $f$ from an $\Codim{n}$-manifold $X$ to another
  $\Codim{n}$-manifold $Y$ is said to an \emph{$n$-map} if $f^{-1}(\del_i Y)
  = \del_i X$ for all $i$.
\end{defn}

\begin{defn}\label{def:n-diagrams}
  Let $\CubeCat{n}$ be the category which has one object for each
  element of $\{0,1\}^n$, and a unique morphism from $b$ to $a$ if
  $b\leq a$ in the obvious partial order on $\{0,1\}^n$.  For $a\in
  \{0,1\}^n$ let $|a|$ denote the sum of the entries in $a$; we call
  this the \emph{weight} of $a$. Let $\vect{0}$ denote $(0,0,\dots,0)$
  and let $\vect{1}$ denote $(1,1,\dots,1)$.

  An \emph{$n$-diagram} is a functor from $\CubeCat{n}$ to the
  category of topological spaces.
  
  An $\Codim{n}$-manifold $X$ can be treated as $n$-diagram in the
  following way.  For $a=\tuple{a}{n}\in\{0,1\}^n$, define
  \[
  X(a)=
  \begin{cases}
    X&\text{ if }a=\vect{1}\\
    \bigcap\limits_{i\in\{i\mid a_i=0\}}\del_i X&\text{ otherwise.}
  \end{cases}
  \]
  For $b\leq a$ in $\{0,1\}^n$, the map $X(b)\to X(a)$ is the inclusion
  map.  It is easy to see that $X(a)$ is an $\Codim{|a|}$-manifold; in
  fact, the $|a|$-diagram corresponding to $X(a)$ is the functor
  restricted to the full subcategory of $\CubeCat{n}$ corresponding to
  the vertices less than or equal to $a$. 

  By an abuse of terminology, henceforth whenever we say
  $\Codim{n}$-manifold, we will either mean a topological space or an
  $n$-diagram, depending on the context.
  
  In this language, products are easy to define: Given
  $\Codim{n_i}$-manifolds $X_i$ for $i\in\{1,2\}$, the \emph{product}
  $X_1\times X_2$ is an $\Codim{n_1+n_2}$-manifold defined via the
  isomorphism $\CubeCat{n_1+n_2}=\CubeCat{n_1}\times\CubeCat{n_2}$.
\end{defn}

Since the category $\CubeCat{n}$ has a unique maximal vertex
$\vect{1}$, and for $Y$ an $\Codim{n}$-manifold the maps $Y(a)\to
Y(\vect{1})$ are injective, the manifold $Y=Y(\vect{1})$ is the direct limit (colimit) of $Y$, thought of as an $n$-diagram.
This observation allows us to talk about the
boundary of an $\Codim{n}$-manifold $Y$ without explicitly mentioning
$Y(\vect{1})$:

\begin{defn}\label{def:n-boundary}
  Let $\CubeSubCat{n}$ be the full subcategory of $\CubeCat{n}$
  generated by all the elements except $\vect{1}$. A
  \emph{truncated $n$-diagram} is a functor from $\CubeSubCat{n}$ to
  the category of topological spaces. A \emph{$k$-dimensional
    $\Codim{n}$-boundary} is a truncated $n$-diagram such that the
  restriction of the functor to each maximal dimensional face of
  $\CubeSubCat{n}$ is a $k$-dimensional $\Codim{n-1}$-manifold.

  Given an $\Codim{n}$-boundary $Y$, let $\colim Y$ denote the pushout
  of $Y$. Note that for any vertex $a\in\{0,1\}^n$, except $\vect{1}$,
  there is a continuous map
  \[
  \phi_a\co Y(a)\to \colim Y.
  \]
\end{defn}

Therefore, the restriction of a $k$-dimensional $\Codim{n}$-manifold
$X$ to $\CubeSubCat{n}$ produces a $(k-1)$-dimensional
$\Codim{n}$-boundary $Y$, with $\colim Y=\del X$.

\begin{lem}\label{lem:Ya-to-colim-inj}
  Let $Y$ be an $\Codim{n}$-boundary, and $a$ any vertex of
  $\CubeSubCat{n}$. Then $\phi_a\co Y(a)\to \colim Y$ is injective.
\end{lem}
\begin{proof}
  This is immediate from the definitions.
\end{proof}

By \Lemma{Ya-to-colim-inj}, it makes sense to say that a point
$p\in \colim Y$ lies in $Y(a)$ if $p$ is in the image of $\phi_a$.

\begin{lem}\label{lem:colimit-intersection}
  Given vertices $a,b$ of $\CubeSubCat{n}$, let $c$ be the greatest
  lower bound of $a$ and $b$. (That is, $c\leq a$ and $c\leq b$, and
  $c$ is maximal with respect to this property; there is a unique such
  $c$.) Then $Y(a)\cap Y(b)=Y(c)$.
\end{lem}
\begin{proof}
  This is immediate from the definition of the colimit (and the fact
  that, for the cube, any pair of vertices has a unique greatest lower
  bound).
\end{proof}

\begin{lem}\label{lem:depth-in-bdy}
  Let $Y$ be an $\Codim{n}$-boundary. The following conditions on a point
  $p\in \colim Y$ are equivalent:
  \begin{enumerate}
  \item\label{item:depth-1} There are exactly $k$ weight $(n-1)$
    vertices $a_1,\dots,a_k\in\{0,1\}^n$ so that $p\in Y(a_i)$,
    $i=1,\dots, k$.
  \item\label{item:depth-2} There is a unique weight $(n-k)$ vertex $b$
    so that $p\in Y(b)$.
  \item\label{item:depth-3} For each weight $(n-1)$ vertex $a$ so that
    $p\in Y(a)$, $p$ lies in the codimension-$(k-1)$ boundary of
    $Y(a)$.
  \end{enumerate}
  If any of these three equivalent conditions hold for $p$ then we say
  that $p$ has \emph{depth k}.
\end{lem}

\begin{proof}
  \Item{depth-1}$\implies$\Item{depth-2}: This follows
  from \Lemma{colimit-intersection} by induction; the vertex
  $b$ is the greatest lower bound of $a_1,\dots,a_k$.

  \Item{depth-2}$\implies$\Item{depth-1}: Again, this is immediate
  from \Lemma{colimit-intersection}: if $a_i$ is a weight $(n-1)$
  vertex so that $b\leq a_i$ then $p\in Y(a_i)$. There are $k$ such
  vertices. If $p\in Y(a)$ for any other weight $(n-1)$ vertex $a$
  then $p$ would lie in $Y(c)$ for some weight $(n-k-1)$ vertex $c$
  (by the previous step), and hence in $Y(b')$ for another weight
  $(n-k)$ vertex $b'$.

  \Item{depth-3}$\implies$\Item{depth-2}: Since the face $\{c\mid
  c\leq a\}$ is an $\Codim{n-1}$-manifold and $p$ is in the
  codimension-$(k-1)$ boundary of $Y(a)$, there must be a weight
  $(n-k)$ vertex $b$ so that $p\in Y(b)$. Moreover, $b$ is unique,
  since if $b'$ is another such vertex and $c$ is the greatest lower
  bound of $b$ and $b'$ then $p\in Y(c)$ (by
  \Lemma{colimit-intersection}), and the weight of $c$ is less than
  $(n-k)$, so $p$ lies in the codimension $<(k-1)$ boundary of $Y(a)$.

  \Item{depth-2}$\implies$\Item{depth-3}: By
  \Lemma{colimit-intersection}, $b\leq a$. Since the face $\{c\mid
  c\leq a\}$ of $\CubeSubCat{n}$ is an $\Codim{n-1}$-manifold, and
  $p\in Y(b)$, it follows that $p$ lies in the codimension $\geq (k-1)$
  boundary of $Y(a)$. If $p$ lies in the codimension $\geq k$ boundary
  of $Y(a)$, then $p\in Y(c)$ for some vertex $c$ of weight less than
  $(n-k)$ (by the previous step), and hence $p\in Y(b')$ for another
  weight $(n-k)$ vertex $b'$.
\end{proof}

\begin{prop}\label{prop:n-bdy-cover}
  Let $Y$ and $X$ be $n$-boundaries so that $Y(a)$ and $X(a)$ are
  compact for each object $a$ of $\CubeSubCat{n}$. Let $F\co Y\to X$
  be a natural transformation such that the restriction of $F$ to each
  face of $Y$ is an $(n-1)$-map. If for every
  $a\in\Ob_{\CubeSubCat{n}}$ the map $F(a)\co Y(a)\to X(a)$ is a
  covering map then
  \[
  G=\colim(F)\co \colim(Y)\to\colim(X)
  \]
  is a covering map.
\end{prop}
\begin{proof}
  Fix a point $p\in\colim(X)$ and a point $q\in
  G^{-1}(p)\subset\colim(Y)$. Let $k$ be the depth of $p$.

  We start by showing that $q$ has depth $k$. If $a\in\{0,1\}^n$ is a
  weight $(n-1)$ vertex so that $q\in Y(a)$ then $p=G(q)\in X(a)$.
  By \Part{depth-3} of \Lemma{depth-in-bdy}, $p$ lies in the
  codimension-$(k-1)$ boundary of $X(a)$. Since the restriction of $F$
  to the face containing $a$ is an $(n-1)$-map, $q$ lies in the
  codimension-$(k-1)$ boundary of $Y(a)$, so the depth of $q$ is $k$.

  Now, let $a_1,\dots,a_k$ be the weight $(n-1)$ vertices so that
  $q\in Y(a_i)$ (and $x\in X(a_i)$) for $i=1,\dots,k$. Let $a_{i,j}$
  be the unique weight $(n-2)$ vertex so that $a_{i,j}\leq a_i$ and
  $a_{i,j}\leq a_j$.  Choose a small open neighborhood $U_i$ of $p$ in
  $X(a_i)$ so that $U_i\cap X(a_{i,j})=U_j\cap X(a_{i,j})$ for each
  pair $i,j$. Then $U=U_1\cup\dots\cup U_k$ is an open neighborhood of
  $p$ in $\colim X$. Let $V_i=F^{-1}(U_i)\subset Y(a_i)$. Let $V$ be
  the connected component of $V_1\cup\dots\cup V_k$ containing $q$;
  then $V$ is an open neighborhood of $q$ in $\colim Y$, and $G|_V\co
  V\to U$ is a homeomorphism.
\end{proof}

\begin{defn}\label{def:neat-embedding}
Given an $(n+1)$-tuple $\TupV{d}=\tuple[0]{\Tup{d}}{n}\in\N^{n+1}$, let
  \[
  \Euclid{n}{\TupV{d}}=\R^{\Tup{d}_0}\times\R_+\times\R^{\Tup{d}_1}\times\R_+\times\dots\times\R_+\times\R^{\Tup{d}_n}.
  \]
  Treat
  $\Euclid{n}{\TupV{d}}$ as an $\Codim{n}$-manifold by defining
  \[
  \del_i(\Euclid{n}{\TupV{d}})=\R^{\Tup{d}_0}\times\dots\times\R^{\Tup{d}_{i-1}}\times\{0\}\times\R^{\Tup{d}_i}\times\dots\times\R^{\Tup{d}_n}.
  \]
  For
  $\TupV{d}\leq \TupV{d}'$ in $\N^{n+1}$ (with respect to the obvious
  partial order), view $\Euclid{n}{\TupV{d}}$ as the subspace
  \[
  \R^{\Tup{d}_0}\times\{0\}^{\Tup{d}'_0-\Tup{d}_0}\times\R_+\times\dots\times\R_+\times\R^{\Tup{d}_n}\times\{0\}^{\Tup{d}'_n-\Tup{d}_n}\subseteq \Euclid{n}{\TupV{d}'}.
  \]

  A \emph{neat immersion $\iota$} of an $\Codim{n}$-manifold is a
  smooth immersion $\iota\from X\imto \Euclid{n}{\TupV{d}}$ for some $\TupV{d}$,
  such that:
  \begin{enumerate}
  \item $\iota$ is an $n$-map, i.e.\ $\iota^{-1}(\del_i\Euclid{n}{\TupV{d}})=\del_iX$ for all $i$, and
  \item the intersection of $X(a)$ and $\Euclid{n}{\TupV{d}}(b)$ is
    perpendicular, for all $b<a$ in $\{0,1\}^n$.
  \end{enumerate}
  A \emph{neat embedding} is a neat immersion that is also an
  embedding. (See \Figure{neat-embed-triangle} for an example.)

  Given $\TupV{d}'\in\N^{n+1}$ and a neat immersion $\iota\from X
  \imto\Euclid{n}{\TupV{d}}$, let $\iota[\TupV{d}']\from X\imto
  \Euclid{n}{\TupV{d}+\TupV{d}'}$ be the induced neat immersion.
\end{defn}

\begin{figure}
  \centering
  \includegraphics[width=0.3\textwidth]{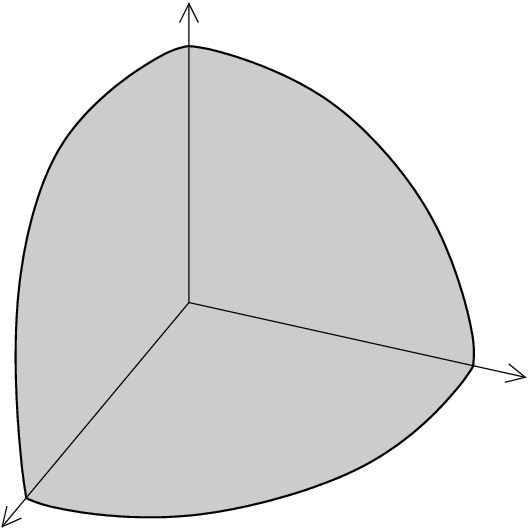}
  \caption[A neat embedding of the triangle.]{\textbf{A neat embedding of the $2$-dimensional
      $\Codim{3}$-manifold `triangle' in $(\R_+)^3$.}}\label{fig:neat-embed-triangle}
\end{figure}

The existence of neat embeddings is guaranteed by the following
version of the Whitney embedding theorem for manifolds with
corners. 
\begin{lem}\label{lem:extend-neat-embedding}
  Fix a compact $\Codim{n}$-manifold $X$. Let
  $\iota\from\del X\into\Euclid{n}{\TupV{d}}$ be an embedding such
  that $\iota^{-1}(\del_i\Euclid{n}{\TupV{d}})=\del_iX$ for all $i$,
  and the intersection of $X(a)$ and $\Euclid{n}{\TupV{d}}(b)$ is
  perpendicular, for all $b<a$ in $\{0,1\}^n\sm\{\vect{1}\}$. Then
  $\iota[\TupV{d}'-\TupV{d}]$ can be extended to a neat embedding of
  $X$ in $\Euclid{n}{\TupV{d}'}$ for some $\TupV{d}'\geq \TupV{d}$.
\end{lem}
\begin{proof}
  The proof is standard, and is similar to the proof
  of~\cite[Proposition 2.1.7]{Lau-top-cobordismcorners}, using the
  fact that faces of manifolds with corners admit collar
  neighborhoods~\cite[Proposition 2.1.6]{Lau-top-cobordismcorners}.
\end{proof}

\begin{defn}
  Let $X$ be a $k$-dimensional $\Codim{n}$-manifold. Given a neat
  immersion $\iota\from X\imto \Euclid{n}{\TupV{d}}$, the \emph{normal
    bundle $\nu_{\iota}$} is defined as follows. For $a\in\{0,1\}^n$
  and $x\in X(a)$, the bundle $\nu_{\iota}$ at $x$ is the normal
  bundle of the immersion $\restrict{\iota}{X(a)}\from X(a)\imto
  \Euclid{n}{\TupV{d}}(a)$.  

  Equivalently, after setting $r=|\TupV{d}|+n-k$, we can view the normal
  bundle as a map $\nu_{\iota}\from X\to
  \Grassmann{r,r+|a|}\to\Grassmann{r}= \BO(r)$, as
  follows. For $a\in\{0,1\}^n$ and $x\in X(a)$, $\nu_{\iota}(x)$ is
  the orthogonal complement of $T_x(X(a))$ inside
  $T_x(\Euclid{n}{\TupV{d}}(a))=\R^{r+|a|}$. The map is
  well-defined because for all $b\leq a$ in $\{0,1\}^n$, the following
  diagram commutes.
  \[
  \xymatrix{
    X(a)\ar[r]&\Grassmann{r,r+|a|}\\
    X(b)\ar[r]\ar@{^(->}[u]&\Grassmann{r,r+|b|}.\ar@{^(->}[u]}
  \]
  (Here, the vertical arrow on the right is induced from the inclusion
  $\Euclid{n}{\TupV{d}}(b)\into \Euclid{n}{\TupV{d}}(a)$.)
\end{defn}

\subsection{Flow categories}\label{sec:flow-cats}
In this subsection, we recall the basic definitions related to flow
categories and framed flow categories, and review briefly how they
arise in Morse theory. In the next subsection we will show how to
associate a CW complex to a framed flow category.
\begin{definition}\label{def:flow-cat}
  A \emph{flow category} is a pair $(\Cat,\gr)$ where $\Cat$ is a
  category with finitely many objects $\Ob=\Ob(\Cat)$ and
  $\gr\co \Ob\rightarrow\Z$ is a function, called the grading, satisfying the
  following additional conditions:
  \begin{enumerate}[series=flow-category,
    label=(M-\arabic*),ref=M-\arabic*]
  \item\label{condition:objects-graded}
    $\Hom(x,x)=\{\Id\}$ for all $x\in\Ob$, and for distinct
    $x,y\in\Ob$, $\Hom(x,y)$ is a compact
    $(\gr(x)-\gr(y)-1)$-dimensional $\Codim{\gr(x)-\gr(y)-1}$-manifold
    (with the understanding that negative dimensional manifolds are
    empty).
  \item\label{condition:maps-to-bdy} For distinct $x,y,z\in\Ob$ with
    $\gr(z)-\gr(y)=m$, the composition map
    \[
    \circ\co \Hom(z,y)\times\Hom(x,z)\to \Hom(x,y)
    \]
    is an embedding into $\del_m\Hom(x,y)$. Furthermore,
    \[
    \circ^{-1}(\del_i\Hom(x,y))=
    \begin{cases}
      \del_i\Hom(z,y)\times\Hom(x,z)\text{ for }i<m\\
      \Hom(z,y)\times\del_{i-m}\Hom(x,z)\text{ for }i>m
    \end{cases}
    \]
  \item\label{condition:boundary-is} For distinct $x,y\in\Ob$, $\circ$
    induces a diffeomorphism
    \begin{equation}\label{eq:del-i-part-of-bdy}
      \bdy_i\Hom(x,y)\cong\coprod_{\substack{z\\\gr(z)=\gr(y)+i}}\Hom(z,y)\times\Hom(x,z).
    \end{equation}
  \end{enumerate}

  Therefore, if $\mf{D}$ is the diagram whose vertices are the spaces
  $\Hom(z_m,y)\times \Hom(z_{m-1},z_m)\times\dots\times \Hom(x,z_1)$
  for $m\geq 1$ and distinct $z_1,\dots,z_m\in\Ob\sm\{x,y\}$, and
  whose arrows correspond to composing a single adjacent pair of
  $\Hom$'s, then
  \[
  \bdy\Hom(x,y)\cong\bigcup_i\bdy_i\Hom(x,y)\cong\colim\mf{D}.
  \]

  Given objects $x,y$ in a flow category $\Cat$, define the
  \emph{moduli space from $x$ to $y$} to be
  \[
  \Moduli(x,y)=
  \begin{cases}
    \emptyset\text{ if }x=y\\
    \Hom(x,y)\text{ otherwise.}
  \end{cases}
  \]

  Given a flow category $\Cat$ and an integer $n$, let $\Cat[n]$ be
  the flow category obtained from $\Cat$ by increasing the grading of
  each object by $n$.
\end{definition}

Key motivation for \Definition{flow-cat} comes from the following:
\begin{definition}\label{def:flow-cat-of-Morse}
  Given a Morse function $f\from M\to\RR$ with finitely many critical
  points and a gradient-like flow for $f$, one can construct the
  \emph{Morse flow category $\MorseFlowCat(f)$} as follows.

  The objects of $\MorseFlowCat(f)$ are the critical points of $f$;
  the grading of an object is the index of the critical point. For
  critical points $x,y$ of relative index $n+1$, let
  $\ParaModuli(x,y)$ be the \emph{parametrized moduli space} from $x$
  to $y$, which is the $(n+1)$-dimensional subspace of $M$ consisting
  of all the points that flow up to $x$ and flow down to $y$. Let
  $\interior{\Moduli}(x,y)$ be the quotient of $\ParaModuli(x,y)$ by
  the natural $\R$-action. Define $\Moduli(x,y)$, the space of
  morphisms from $x$ to $y$, to be the \emph{unparametrized moduli
    space} from $x$ to $y$, which is obtained by compactifying
  $\interior{\Moduli}(x,y)$ by adding all the broken flowlines from
  $x$ to $y$, cf.\ \cite[Lemma 2.6]{AB-gauge-morsebott}.
\end{definition}

\begin{definition}
  Let $\Cat$ be a flow category. Given an integer $i$, let $\Ob(i)$
  denote the set of all objects of $\Cat$ in grading $i$, topologized
  as a discrete space. Given integers $i$ and $j$, let
  \[
  \Moduli(i,j)=\coprod_{x\in\Ob(i),y\in\Ob(j)}\Moduli(x,y).
  \]
\end{definition}

There are obvious maps from $\Moduli(i,j)$ to $\Ob(i)$ and $\Ob(j)$.
For all $m,n,i$ with $1\leq i<m-n$, \Equation{del-i-part-of-bdy}
induces a diffeomorphism between the fiber product
$\Moduli(n+i,n)\displaystyle\fiberprod_{\Ob(n+i)}\Moduli(m,n+i)$ and
$\del_i\Moduli(m,n)$.

We turn next to (framed) neat embeddings of flow categories; such a framed embedding is part of the input to the Cohen-Jones-Segal construction. Because we will consider covering spaces later on, we will actually work with a slightly more general notion: neat immersions (as we did with $\Codim{n}$-manifolds); see also \Section{covers}. Still, the Cohen-Jones-Segal construction requires an embedding, so we will show in \Lemma{perturb-connect-framings} that any framed neat immersion can be perturbed to a framed neat embedding in an essentially unique way.

\begin{definition}\label{def:neat-immersion}
  For each integer $i$, fix a natural number $\Tup{d}_i$ and let
  $\TupV{d}$ denote this sequence. For $a<b$,
  let $\ESpace[\TupV{d}]{a}{b}=
  \Euclid{b-a-1}{\tuple[a]{\Tup{d}}{b-1}}$.

  A \emph{neat immersion (\respectively neat embedding) $\iota$} of a flow
  category $\Cat$ relative $\TupV{d}$ is a collection of neat
  immersions (\respectively neat embeddings) $\iota_{x,y}\from\allowbreak
  \Moduli(x,y)\allowbreak \imto \ESpace[\TupV{d}]{\gr(y)}{\gr(x)}$, defined for
  all objects $x,y$, such that the following conditions hold.
  \begin{enumerate}[series=neat-embed,label=(N-\arabic*),ref=N-\arabic*]
  \item For all integers $i,j$, $\iota$ induces a neat immersion
    (\respectively neat embedding) $\iota_{i,j}$ of
    $\Moduli(i,j)$. 
  \item For all objects $x,y,z$, and for all points $(p,q)\in\Moduli(x,z)\times\Moduli(z,y)$,
    \[
    \iota_{x,y}(q\circ p)=(\iota_{z,y}(q),0,\iota_{x,z}(p)).
    \]
  \end{enumerate}
  (See \Figure{emb-flow-cat} for an example.)

  Given a neat immersion $\iota$ of a flow category $\Cat$ relative
  $\TupV{d}$, and some other $\TupV{d}'$, let $\iota[\TupV{d}']$
  denote the neat immersion relative $\TupV{d}+\TupV{d}'$ induced by
  the inclusions
  $\ESpace[\TupV{d}]{a}{b}\into\ESpace[\TupV{d}+\TupV{d}']{a}{b}$,
  which in turn are induced by the canonical inclusions
  $\RR^{d_i}\cong \RR^{d_i}\times \{0\}\into \RR^{d_i+d'_i}$.
\end{definition}

\begin{lem}\label{lem:flow-has-embeddings}
  Any flow category admits a neat embedding relative to some $\TupV{d}$.
\end{lem}

\begin{proof}
  This follows from \Lemma{extend-neat-embedding}, by induction.
\end{proof}

\begin{lem}\label{lem:connect-neat-immersions}
  Let $\DD^k$ be a $k$-dimensional disk, and let $S^{k-1}$ denote its
  boundary.  If $\iota$ is a smooth $S^{k-1}$-parameter family of neat
  immersions of a flow category $\Cat$ relative to some fixed
  $\TupV{d}$, then for $\TupV{d}'$ sufficiently large,
  $\iota[\TupV{d}']$ can be extended to a smooth
  $\DD^k$-parameter family of neat immersions $\iota'$ of
  $\Cat$. Furthermore, if $\iota$ is a neat
  embedding for some point in $S^{k-1}$, then we can ensure that
  $\iota'$ is a neat embedding at all points in the interior of
  $\DD^k$.
\end{lem}

\begin{proof}
  We follow the proof of \cite[Proposition
  2.2.3]{Lau-top-cobordismcorners}. Define $\TupV{d}'$ by setting
  $\Tup{d}'_i=\Tup{d}_i+1$. Assume $\DD^k$ is the unit disk in
  $\R^k$. We use `spherical coordinates' $(r,\theta)$ to describe
  points on $\DD^k$, where $r$ is the distance from the origin and
  $\theta$ is a point on $S^{k-1}$.

  Fix objects $x,y$ with $\gr(x)=m$ and
  $\gr(y)=n$.  Recall that
  \begin{align*}
    \ESpace[\TupV{d}']{n}{m}&=\R^{d_n}\times\R^{d_n}\times\R\times\R_+\times\dots\times\R_+\times\R^{d_{m-1}}\times\R^{d_{m-1}}\times\R\\
    \intertext{Hence, write $\iota_{x,y}(\theta)[\TupV{d}']$
      in coordinate form as}
    \iota_{x,y}(\theta)[\TupV{d}']&=(\iota_n(\theta),0,0,\ol{\iota}_{n+1}(\theta),\dots,\ol{\iota}_{m-1}(\theta),\iota_{m-1}(\theta),0,0).
  \end{align*}

  Let $f\from[0,1]\to\R_+$ be defined as
  \[
  f(r)=
  \begin{cases}
    e^{-1/(r-r^2)}\text{ if }r\neq 0,1\\
    0\text{ otherwise.}
  \end{cases}
  \] 
  Choose some point $\theta_0\in S^{k-1}$; if possible, choose a
  $\theta_0$ such that $\iota(\theta_0)$ is a neat embedding.  Define
  the isotopy to be
  \begin{align*}
    \iota'_{x,y}(r,\theta)&=((1-r)\iota_n(\theta_0)+r\iota_n(\theta),f(r)\iota_n(\theta_0),f(r)\ol{\iota}_{n+1}(\theta_0),(1-r)\ol{\iota}_{n+1}(\theta_0)+r\ol{\iota}_{n+1}(\theta),\dots\\
    &\qquad\dots,(1-r)\ol{\iota}_{m-1}(\theta_0)+r\ol{\iota}_{m-1}(\theta),(1-r)\iota_{m-1}(\theta_0)+r\iota_{m-1}(\theta),f(r)\iota_{m-1}(\theta_0),0).\qedhere
  \end{align*}
\end{proof}

\begin{defn}\label{def:coherent-framing}
  Let $\iota$ be a neat immersion of a flow category $\Cat$ relative
  $\TupV{d}$. A \emph{coherent framing} $\Frame$ for $\iota$
  is a framing for $\nu_{\iota_{x,y}}$ for all objects $x,y$, such
  that the product framing of $\nu_{\iota_{z,y}}\times
  \nu_{\iota_{x,z}}$ equals the pullback framing of
  $\circ^*\nu_{\iota_{x,y}}$ for all $x,y,z$.  (Again, see \Figure{emb-flow-cat} for an example.)
\end{defn}

\begin{lem}\label{lem:framing-extension}
  Let $\DD^k$ be a $k$-dimensional disk and let $\DD^{k-1}$ be a
  hemisphere on its boundary. Assume $\iota$ is a smooth $\DD^k$-parameter
  family of neat immersions of a flow category $\Cat$ relative to some
  fixed $\TupV{d}$, and assume $\Frame$ is a smooth $\DD^{k-1}$-parameter
  family of coherent framings for $\restrict{\iota}{\DD^{k-1}}$. Then
  $\Frame$ can be extended to a smooth $\DD^k$-parameter family of coherent
  framings for $\iota$.
\end{lem}

\begin{proof}
  For any $x,y$, $\iota_{x,y}$ is a $\DD^k$-parameter family of neat
  immersions of $\Moduli(x,y)$ in
  $\ESpace[\TupV{d}]{\gr(y)}{\gr(x)}$. The normal bundle gives a map
  $\Moduli(x,y)\times \DD^k\to \BO(r)$, where we have set
  $r=\Tup{d}_{\gr(y)}+\dots+\Tup{d}_{\gr(x)-1}$. The framing
  $\Frame$ of $\restrict{\iota}{\DD^{k-1}}$ produces a map
  $\Moduli(x,y)\times\DD^{k-1}\to\EO(r)$ such
  that the following commutes.
  \[
  \xymatrix{
    \Moduli(x,y)\times\DD^{k-1}\ar[r]\ar@{^(->}[d]&\ar[d]\EO(r)\\
    \Moduli(x,y)\times\DD^k \ar[r]&\BO(r) 
  }
  \]
  We want to produce a map $\Moduli(x,y)\times \DD^k \to \EO(r)$ which
  will commute with the previous diagram, and will induce a coherent
  framing. We do this by an induction on $\gr(x)-\gr(y)$.

  The base case, when $\gr(x)-\gr(y)=0$, is vacuous. Now assume that
  we have framed $\nu_{\iota_{x',y'}(t)}$ for all $t\in\DD^k$ and all
  $x',y'$ with $\gr(x')-\gr(y')<\gr(x)-\gr(y)$. Since every point in
  the boundary of $\Moduli(x,y)$ lies in the product of lower
  dimensional moduli spaces, the normal bundle $\nu_{\iota_{x,y}(t)}$,
  restricted to $\del \Moduli(x,y)$, is already framed.
  Thus, we have the following commuting diagram; we want to produce
  the dashed arrow so that the diagram still commutes.
  \[
  \xymatrix{
    (\Moduli(x,y)\times\DD^{k-1})\cup(\del \Moduli(x,y)\times\DD^k)\ar[r]\ar@{^(->}[d]&\ar[d]\EO(r)\\
    \Moduli(x,y)\times\DD^k\ar[r]\ar@{-->}[ur]&\BO(r) 
  }
  \]
  By the collar neighborhood theorem, there is a homeomorphism of
  pairs 
  \[
  (\Moduli(x,y)\times\DD^k,(\Moduli(x,y)\times\DD^{k-1})\cup(\del
  \Moduli(x,y)\times\DD^k))\cong (\Moduli(x,y)\times\DD^k,\Moduli(x,y)\times\DD^{k-1}).
  \]
  So, since $\EO(r)\to \BO(r)$ is a fibration, the dashed arrow exists.
\end{proof}

\begin{definition}\label{def:framed-flow-cat}
  A \emph{framed flow category} is a neatly embedded flow category
  $\Cat$, along with a coherent framing for some
  neat immersion of $\Cat$ (relative to some $\TupV{d}$).

  To a framed flow category, one can associate a chain complex
  $C^*(\Cat)$ as follows.  The $n\th$ chain group $C^n$ is the
  $\Z$-module freely generated by $\Ob(n)$. The differential $\diff$
  is of degree one. For $x,y\in\Ob$ with $\gr(x)=\gr(y)+1$, the
  coefficient of $\diff y$ evaluated at $x$ is the number of points in
  $\Moduli(x,y)$, counted with sign (recall, $\Moduli(x,y)$ is a
  compact framed $0$-dimensional $\Codim{0}$-manifold).  We say a
  framed flow category \emph{refines} its associated chain complex.
\end{definition}

\begin{defn}\label{def:coherent-framing-to-framed}
  Let $\Frame$ be a coherent framing for some neat immersion $\iota$
  of a flow category $\Cat$. Using \Lemma{flow-has-embeddings}, choose
  some neat embedding $\iota'$ of $\Cat$. Using
  \Lemma{connect-neat-immersions}, connect $\iota[\TupV{d}]$ to
  $\iota'[\TupV{d}']$ for some $\TupV{d},\TupV{d}'$. The coherent
  framing $\Frame$ induces a coherent framing for $\iota[\TupV{d}]$,
  also denoted by $\Frame$. Therefore, using
  \Lemma{framing-extension}, produce a coherent framing $\Frame'$ for
  $\iota'[\TupV{d}']$. We say that the framed flow category
  $(\Cat,\iota'[\TupV{d}'],\Frame')$ is \emph{a perturbation} of
  $(\Cat,\iota,\Frame)$.
\end{defn}

\begin{lem}\label{lem:perturb-connect-framings}
  If $(\Cat,\iota_0,\Frame_0)$ and $(\Cat,\iota_1,\Frame_1)$ are two
  perturbations of $(\Cat,\iota,\Frame)$, then for some
  $\TupV{d}_0,\TupV{d}_1$, there exists a smooth $1$-parameter family of
  framings of $\Cat$ that connects $(\iota_0[\TupV{d}_0],\Frame_0)$ to
  $(\iota_1[\TupV{d}_1],\Frame_1)$.
\end{lem}

\begin{proof}
  For $i\in\{0,1\}$, since $(\Cat,\iota_i,\Frame_i)$ is a perturbation
  of $(\Cat,\iota,\Frame)$, for some $\TupV{d}'_i$ there exists a
  $1$-parameter family $\iota_i(t)$ of neat immersions connecting
  $\iota[\TupV{d}'_i]$ to $\iota_i$, along with a $1$-parameter
  family of coherent framings $\Frame_i(t)$ for $\iota_i(t)$.

  Using \Lemma{connect-neat-immersions}, choose a $1$-parameter
  family of neat embeddings $\iota(t)$ connecting
  $\iota_0[\TupV{d}''_0]$ to $\iota_1[\TupV{d}''_1]$ for some
  $\TupV{d}''_i$. Consider a triangle $\Delta$, whose three vertices
  are $v,v_0,v_1$, with opposite edges $e,e_0,e_1$,
  respectively. The three $1$-parameter families of neat immersions
  $\iota(t)$, $\iota_0(t)[\TupV{d}''_0]$ and
  $\iota_1(t)[\TupV{d}''_1]$ together produce a $\del\Delta$-parameter
  family of neat immersions of $\Cat$: the restriction to $e$ is
  $\iota(t)$, and the restriction to $e_i$ is $\iota_i[\TupV{d}''_i]$.

  For $\TupV{d}$ sufficiently large, using
  \Lemma{connect-neat-immersions} again, choose a $\Delta$-parameter
  family $\ol{\iota}$ of neat immersions of $\Cat$, whose restriction
  to $e$ is $\iota(t)[\TupV{d}]$, and whose restriction to $e_i$ is
  $\iota_i(t)[\TupV{d}''_i+\TupV{d}]$.  Set
  $\TupV{d}_i=\TupV{d}''_i+\TupV{d}$. The framings $\phi_0(t)$ and
  $\phi_1(t)$ induce framings of $\Cat$ along the edges $e_0$ and
  $e_1$ of $\Delta$. By \Lemma{framing-extension} we can extend these
  framings to a $\Delta$-parameter family $\ol{\Frame}$ of coherent
  framings for $\ol{\iota}$. Then
  $(\restrict{\ol{\iota}}{e},\restrict{\ol{\Frame}}{e})$ is a
  $1$-parameter family of framings connecting $\iota_0[\TupV{d}_0]$ to
  $\iota_1[\TupV{d}_1]$.
\end{proof}

\subsection{Framed flow categories to CW
  complexes}\label{sec:flow-to-space}
We are interested in framed flow categories because one can build a CW
complex $|\Cat|$ from a framed flow category $\Cat$ in such a way that
if $\Cat$ refines a chain complex $C^*$ then $C^*$ is the cellular
cochain complex of $|\Cat|$. The construction of $|\Cat|$, which was
first given in~\cite{CJS-gauge-floerhomotopy} (partly inspired
by~\cite{Franks-top-Morse}), is not the same as the usual geometric
realization of a (topological) category. In this subsection we give a
slight reformulation of the construction
from~\cite{CJS-gauge-floerhomotopy}; our formulation is less elegant
but perhaps more concrete.

\begin{definition}\label{def:flow-gives-space}
  Let $(\Cat,\iota,\Frame)$ be a framed flow category, where $\Cat$ is
  a flow category, $\iota$ is a neat embedding relative some
  $\TupV{d}$, and $\Frame$ is a framing of the normal
  bundles to $\iota$. 
  Assume that all objects in $\Cat$ have grading in $[B,A]$ for some
  fixed $A,B\in\Z$.

   We associate a CW complex $\Realize[{\iota,\Frame}]{\Cat}$ to
  $(\Cat,\iota,\Frame)$ as follows, cf.\
  \Figure{space-constructor}. The CW complex $\Realize{\Cat}$ will
  have one $0$-cell, and one cell $\Cell{x}$ for each object $x$ of
  $\Cat$. If $x$ has grading $m$ then the dimension of $\Cell{x}$ will
  be $\Tup{d}_B+\dots+\Tup{d}_{A-1}-B+m$. For convenience, let
  $C=\Tup{d}_B+\dots+\Tup{d}_{A-1}-B$.

  Since $\iota$ is a neat embedding of $\Cat$, for all integers $i,j$,
  $\iota_{i,j}$ embeds $\Moduli(i,j)$ in $\ESpace[\TupV{d}]{j}{i}$. Choose
  $\ep>0$ sufficiently small so that for all $i,j$, $\iota_{i,j}$
  extends to an embedding of
  $\Moduli(i,j)\times[-\ep,\ep]^{\Tup{d}_j+\dots+\Tup{d}_{i-1}}$ via the
  framings of the normal bundles. Choose $R$ sufficiently large so
  that for all $i,j$,
  $\iota_{i,j}(\Moduli(i,j)\times[-\ep,\ep]^{\Tup{d}_j+\dots+\Tup{d}_{i-1}})$ lies
  in
  $[-R,R]^{\Tup{d}_j}\times[0,R]\times\dots\times[0,R]\times[-R,R]^{\Tup{d}_{i-1}}$. 

  Suppose that we have defined the $(C+m-1)$-skeleton
  $\Realize{\Cat}^{(C+m-1)}$ of $\Realize{\Cat}$, i.e., the part of
  $\Realize{\Cat}$ built from objects $y$ with grading $\gr(y)<m$. Fix
  an object $x$ with $\gr(x)=m$.  While describing $\Cell{x}$, we will
  define some related spaces $\Cell[i]{x}$.  Take a copy $\Cell[1]{x}$
  of $\R_+\times\R^{\Tup{d}_B}\times\dots\times
  \R_+\times\R^{\Tup{d}_{A-1}}$.\footnote{It is important that we take a
    copy. For distinct objects $x,y$, we will take different
    copies $\Cell[1]{x}$ and $\Cell[1]{y}$, and we will not, \emph{a
      priori}, identify the corresponding points in $\Cell[1]{x}$ and
    $\Cell[1]{y}$.} Define $\Cell{x}$ to be the following subset
  \begin{equation}\label{eq:C-of-x}
    \begin{split}
      \Cell{x}&= [0,R]\times[-R,R]^{\Tup{d}_B}\times\dots\times
      [0,R]\times[-R,R]^{\Tup{d}_{m-1}}\times
      \{0\}\times[-\ep,\ep]^{\Tup{d}_m}\\
      &\qquad\qquad{}\times\dots\times\{0\}\times [-\ep,\ep]^{\Tup{d}_{A-1}}
      \sbs\Cell[1]{x}.
    \end{split}
  \end{equation}

  It remains to define the attaching map $\bdy \Cell{x}\to
  \Realize{\Cat}^{(C+m-1)}$. 
  Recall that the neat embedding $\iota_{x,y}$, along with the framing
  $\Frame$ of $\nu_{\iota_{x,y}}$, identifies $\Moduli(x,y)\times
  [-\ep,\ep]^{\Tup{d}_n}\times\{0\}\times\dots\times\{0\}\times
  [-\ep,\ep]^{\Tup{d}_{m-1}}$ with a subset $\Cell[y,1]{x}$ of
  $[-R,R]^{\Tup{d}_n}\times[0,R]\times\dots\times[0,R]\times
  [-R,R]^{\Tup{d}_{m-1}}$.  
  Let
  \begin{multline*}
  \Cell[y]{x}=[0,R]\times[-R,R]^{\Tup{d}_B}\times\dots\times[0,R]\times[-R,R]^{\Tup{d}_{n-1}}
  \times \{0\}\times \Cell[y,1]{x}\\
  \times \{0\}\times[-\ep,\ep]^{\Tup{d}_m}\times\dots\times\{0\}\times
  [-\ep,\ep]^{\Tup{d}_{A-1}}
  \subset \bdy\Cell{x}.
  \end{multline*}

  Now, define the attaching map
  $\del\Cell{x}\to\Realize{\Cat}^{C+m-1}$ as follows: on
  $\Cell[y]{x}\cong \Moduli(x,y)\times\Cell{y}$, define it to be the
  projection map to $\Cell{y}$; and map $\del\Cell{x}\sm\bigcup_y\Cell[y]{x}$
  to the basepoint.
\end{definition}

\Figure{space-constructor} illustrates
Definition~\ref{def:flow-gives-space} for a chain complex with six
generators spread over three gradings. All the cells are subsets of
$\E=\R_+\times\R\times\R_+\times\R$. We do not draw $\E$, but we
attempt to draw its boundary $\del E=(\{0\}\times
\R\times\R_+\times\R)\cup(\R_+\times\R\times\{0\}\times\R)$. The space $\del\E$
is naturally a $3$-dimensional $\Codim{2}$-manifold. To draw it, we
flatten out the codimension-$1$ corner
$\{0\}\times\R\times\{0\}\times\R$, giving a homeomorphism (but not a
diffeomorphism) $\del\E\cong \R^3$. In the figures, this (flattened) corner is a
horizontal plane, the first $\R_+$-factor is drawn
below the plane, and the second $\R_+$-factor is drawn above the plane.

\captionsetup[subfloat]{width={0.45\textwidth}}
\begin{figure}
  \centering
    \subfloat[The chain complex associated to the given flow
    category.]{ \tiny \psfrag{0}{$0$} \psfrag{2}{$2$} \psfrag{1}{$1$}
      \psfrag{x1}{$x$} \psfrag{y1}{$z_1$} \psfrag{y2}{$z_2$}
      \psfrag{y3}{$z_3$} \psfrag{y4}{$z_4$} \psfrag{z1}{$y$}
      \psfrag{e1}{$\R^{d_0}$} \psfrag{e2}{$\R^{d_1}$}
      \psfrag{f}{$\R_+$} \psfrag{p}{$+$} \psfrag{m}{$-$}
      \includegraphics[width=0.45\textwidth]{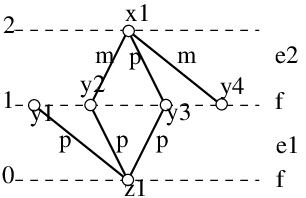}}
    \hspace{0.05\textwidth} \subfloat[A framed neat embedding of the flow
    category.]{\label{fig:emb-flow-cat} \tiny
      \psfrag{e1}{$\R^{d_0}$} \psfrag{e2}{$\R^{d_1}$}
      \psfrag{f}{$\R_+$} \psfrag{xz}{$\Moduli(x,y)$}
      \psfrag{xy2}{$\Moduli(x,z_2)$} \psfrag{xy3}{$\Moduli(x,z_3)$}
      \psfrag{xy4}{$\Moduli(x,z_4)$} \psfrag{y1z}{$\Moduli(z_1,y)$}
      \psfrag{y2z}{$\Moduli(z_2,y)$} \psfrag{y3z}{$\Moduli(z_3,y)$}
      \psfrag{xy2z}{$\Moduli(z_2,y)\times\Moduli(x,z_2)$}
      \psfrag{xy3z}{$\Moduli(z_3,y)\times\Moduli(x,z_3)$}
      \includegraphics[width=0.45\textwidth]{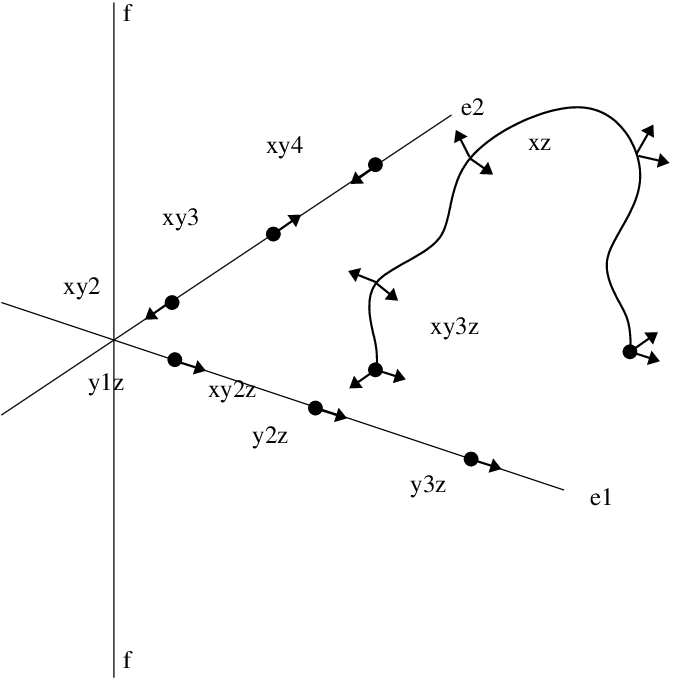}}\\
    \subfloat[$\Cell{y}$ is the shaded square with the boundary
    quotiented to the basepoint.]{ \tiny \psfrag{e1}{$\R^{d_0}$}
      \psfrag{e2}{$\R^{d_1}$} \psfrag{f}{$\R_+$}
      \includegraphics[width=0.45\textwidth]{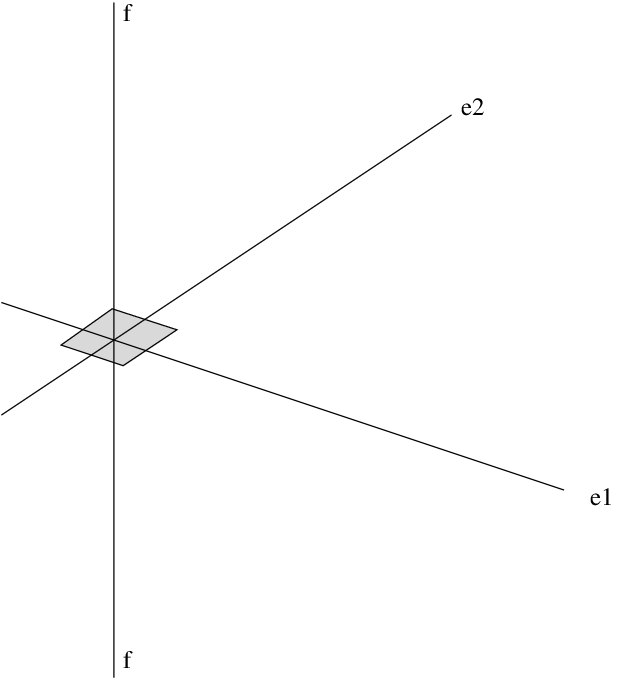}}
    \hspace{0.05\textwidth}\subfloat[$\Cell{z_1}$ is the deeply shaded
    slab with the lightly shaded square on the boundary identified
    with $\Cell{y}$ and the rest of the boundary quotiented to the
    basepoint.]{\label{fig:Cell-z-1} \tiny
      \psfrag{e1}{$\R^{d_0}$} \psfrag{e2}{$\R^{d_1}$}
      \psfrag{f}{$\R_+$}
      \includegraphics[width=0.45\textwidth]{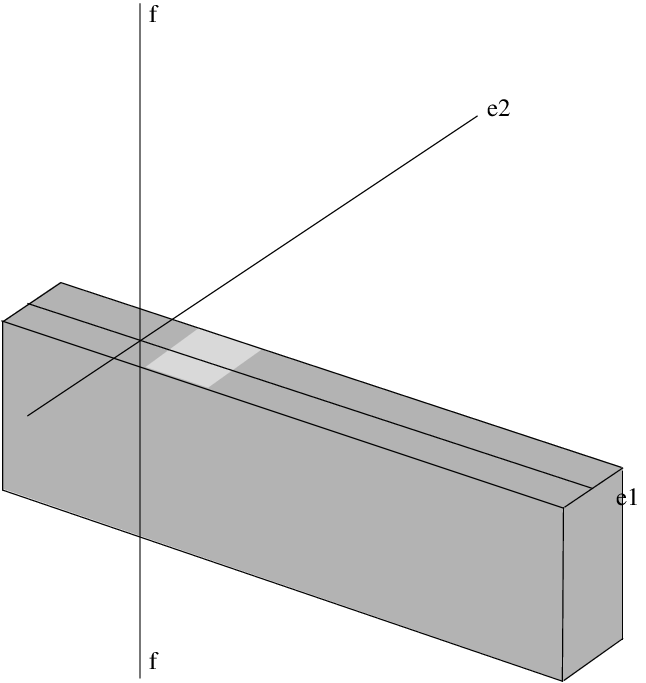}}
  \caption[From a framed flow category to a CW complex.]{\textbf{From a framed flow category to a CW complex.} (Part I)}\label{fig:space-constructor}
\end{figure}
\begin{figure}
  \centering
    \ContinuedFloat \subfloat[$\Cell{z_2}$.]{ \tiny
      \psfrag{e1}{$\R^{d_0}$} \psfrag{e2}{$\R^{d_1}$}
      \psfrag{f}{$\R_+$}
      \includegraphics[width=0.45\textwidth]{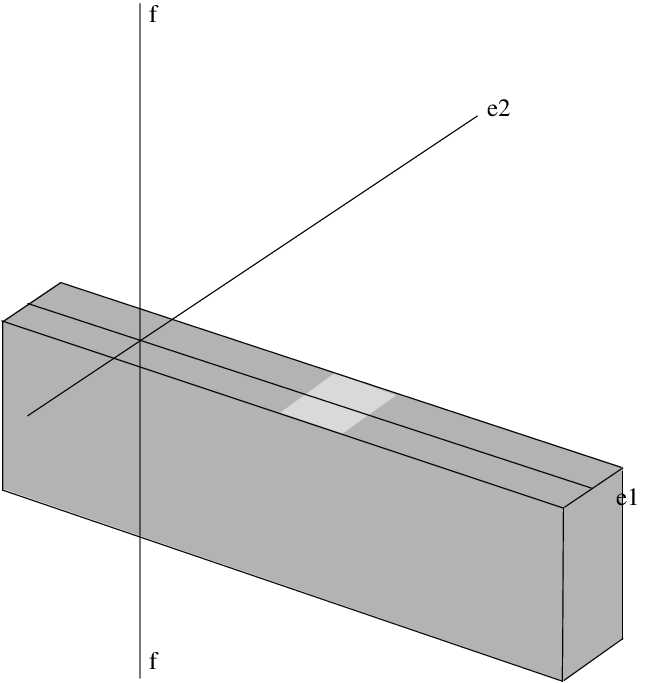}}
    \hspace{0.05\textwidth} \subfloat[$\Cell{z_3}$.]{ \tiny
      \psfrag{e1}{$\R^{d_0}$} \psfrag{e2}{$\R^{d_1}$}
      \psfrag{f}{$\R_+$}
      \includegraphics[width=0.45\textwidth]{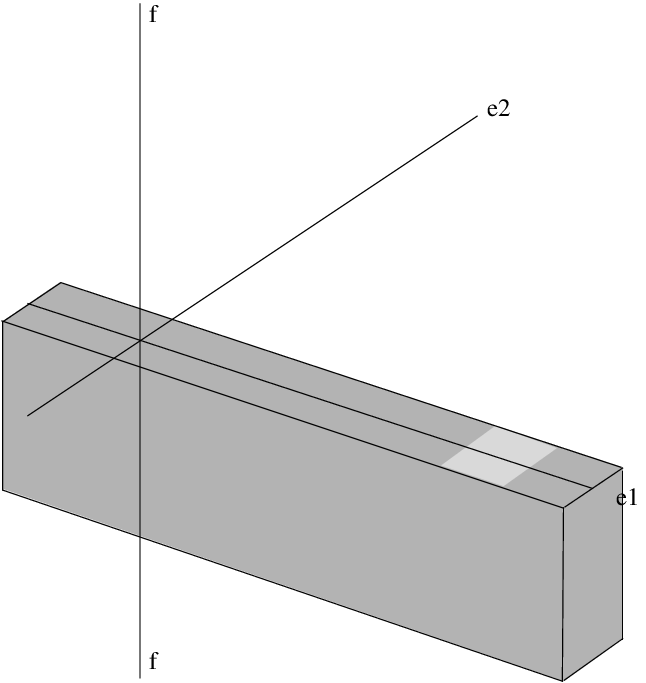}}\\
    \subfloat[$\Cell{z_4}$ (here the entire boundary is quotiented to
    the basepoint).]{ \tiny \psfrag{e1}{$\R^{d_0}$}
      \psfrag{e2}{$\R^{d_1}$} \psfrag{f}{$\R_+$}
      \includegraphics[width=0.45\textwidth]{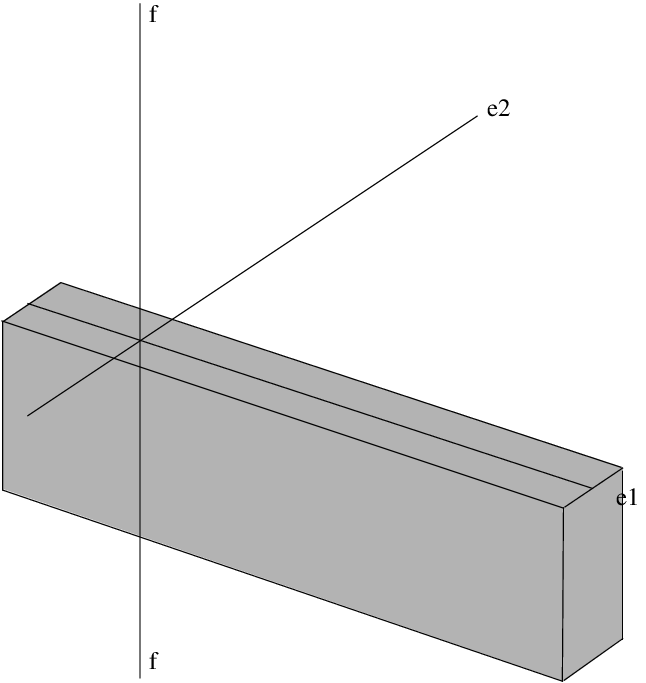}}
    \hspace{0.05\textwidth} \subfloat[\label{fig:Cell-x-bdy}A portion of $\del\Cell{x}$ is
    shown. The complement of the shaded region is quotiented to the
    basepoint; the three shaded slabs are identified with
    $\Cell{z_i}$, $i=2,3,4$; the lightly shaded region is quotiented
    to $\Cell{y}$.]{ \tiny \psfrag{e1}{$\R^{d_0}$}
      \psfrag{e2}{$\R^{d_1}$} \psfrag{f}{$\R_+$}
      \includegraphics[width=0.45\textwidth]{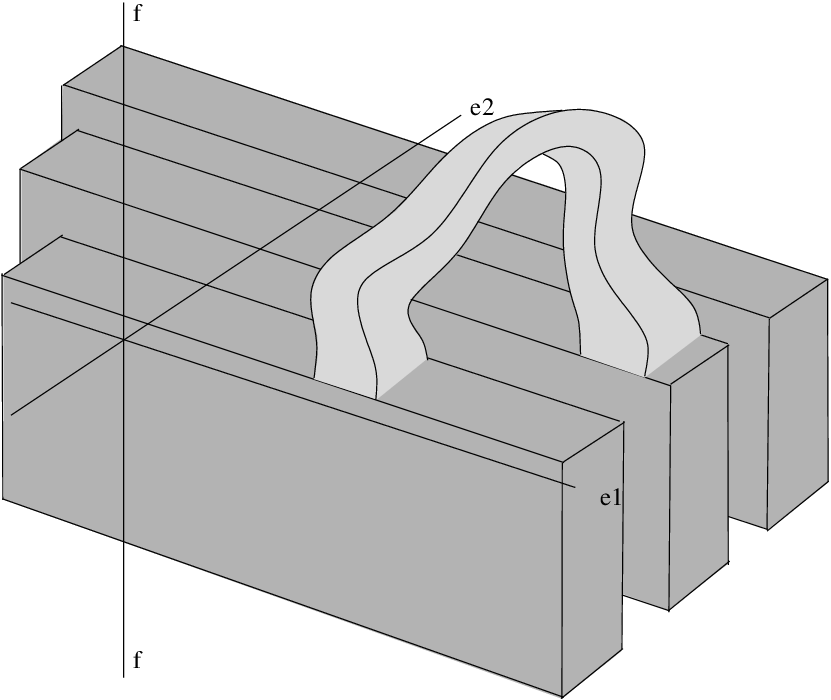}}
  \caption[]{\textbf{From a framed flow category to a CW complex.} (Part II)}
\end{figure}

\begin{lemma}\label{lem:realization-well-defined-refines-chain-complex}
  \Definition{flow-gives-space} produces a well-defined CW complex
  $\Realize\Cat$.  Furthermore, the chain complex associated to $\Cat$
  is isomorphic to $\wt{C}^*(\Realize[\iota,\Frame]{\Cat})[-C]$, where
  $\wt{C}^*$ is the reduced cellular cochain complex and $[\,]$ is the
  degree shift operator.  (For the signs in the cellular cochain
  complex, we orient the cells $\Cell{x}$ from \Equation{C-of-x} by
  the product orientations.)
\end{lemma}

\begin{proof}
  The first part follows from the following elementary check. For
  arbitrary $x,y,z$ with $\gr(x)\geq\gr(z)\geq\gr(y)$, the dashed
  arrow in the following diagram exists such that the diagram
  commutes:
    \[
    \xymatrix{
      \Cell[y]{x}\cap \Cell[z]{x} \ar@{^{(}->}[r]\ar@{^{(}->}[d] \ar@{-->}[ddrr]  & \Cell[z]{x} \ar[r] & \Cell{z}\\
     \Cell[y]{x} \ar[d]& & \del\Cell{z}\ar@{^{(}->}[u]\\
      \Cell{y} & & \Cell[y]{z} \ar[ll] \ar@{^{(}->}[u]
    }
    \]

    For the second part, fix $x,y$ with $\gr(y)=\gr(x)-1$. Then the
    map $\pi\co \bdy \Cell{x}\to \Cell{y}/\bdy \Cell{y}$ has degree
    $\#\pi^{-1}(p)$ for any $p$ in the interior of $\Cell{y}$.
    However, since $\Cell[y]{x}$ is canonically homeomorphic to
    $\Moduli(x,y)\times\Cell{y}$, the number of points in
    $\pi^{-1}(p)$ equals $(-1)^{\Tup{d}_B+\dots+\Tup{d}_{m-2}+m-B}$
    times the number of points in $\Moduli(x,y)$.  The sign
    $(-1)^{\Tup{d}_B+\dots+\Tup{d}_{m-2}+m-B}$ comes from the
    outward-normal first convention for the boundary orientation---the
    opposite of the orientation of $0\in\bdy[0,R]$---and commuting the
    factor $[0,R]$ in $\Cell{x}$ to the first position (which requires
    $\Tup{d}_B+\dots+\Tup{d}_{m-2}+m-B-1$ exchanges).
\end{proof}

\begin{lem}\label{lem:CW-indep-ep-R-framing}
  For a framed flow category $(\Cat,\iota,\Frame)$, the isomorphism type
  of the CW complex $\Realize[{\iota,\Frame}]{\Cat}$ from
  \Definition{flow-gives-space} is independent of the choice of real
  numbers $\ep,R$. In particular, different choices of $\ep, R$ give
  homeomorphic spaces.

  Furthermore, for $t\in[0,1]$, if $(\iota(t),\Frame(t))$ is a
  smooth $1$-parameter family of framings of a flow category $\Cat$, then the
  CW complexes $\Realize[{\iota(0),\Frame(0)}]{\Cat}$ and
  $\Realize[{\iota(1),\Frame(1)}]{\Cat}$ are isomorphic.
\end{lem}

\begin{proof}
  This is clear from the definitions. To wit, changing $\epsilon$ or
  $R$ or deforming the framing or framed embedding has the effect of
  changing the boxes $\Cell{x}\subset\Cell[1]{x}=\RR^N$ and sub-boxes
  $\Cell[y]{x}$ by ambient isotopies of $\RR^N$; and these isotopies
  are consistent across different $x$'s in an obvious sense. Following
  these ambient isotopies gives a homeomorphism between the CW
  complexes $\Realize{\Cat}$ constructed with the two collections of
  data, and this homeomorphism takes cells homeomorphically to cells.
\end{proof}

\begin{lem}\label{lem:CW-indep-A-B}
  Fix a framed flow category $(\Cat,\iota,\Frame)$, where $\iota$ is a
  neat embedding relative $\TupV{d}$. Let
  $\Realize[{\iota,\Frame,B,A}]{\Cat}$ be the CW complex defined in
  \Definition{flow-gives-space}, where the subscripts also indicate
  our choice of integers $A,B$. Let
  $C_{\TupV{d}}(B,A)=d_B+\dots+d_{A-1}-B$ denote the grading shift.

  For any $\TupV{d'}\geq \TupV{d}$, let
  $\iota'=\iota[\TupV{d}'-\TupV{d}]$ be the induced neat embedding relative
  $\TupV{d}'$; let $\Frame'$ denote the induced coherent framing for
  $\iota'$. For any $A'\geq A$ and $B'\leq B$, there is a homotopy equivalence
  \[
  \Realize[{\iota',\Frame', B',A'}]{\Cat}\sim
  \Sigma^{C_{\TupV{d}'}(B',A')-C_{\TupV{d}}(B,A)}\Realize[{\iota,\Frame,B,A}]{\Cat}.
  \]
\end{lem}

\begin{proof}
  For notational convenience, if we denote some object in
  $\Realize[{\iota,\Frame,B,A}]{\Cat}$ by some symbol $\ast$, then we
  will denote the corresponding object in
  $\Realize[{\iota',\Frame',B',A'}]{\Cat}$ by the symbol
  $\ast'$. It is enough to consider the following three cases.
  \begin{enumerate}
  \item $A'=A+1;B'=B;\TupV{d}'=\TupV{d};d_A=0$. 
    In this case, from
    \Equation{C-of-x} it is clear that $\Cell{x}'=\Cell{x}$ for
    all objects $x$ in $\Cat$; furthermore, this identification respects
    the attaching maps. Therefore, we have a homeomorphism
    \[
    \Realize[{\iota',\Frame', B',A'}]{\Cat}\cong
    \Realize[{\iota,\Frame,B,A}]{\Cat}=
    \Sigma^{C_{\TupV{d}'}(B',A')-C_{\TupV{d}}(B,A)}\Realize[{\iota,\Frame,B,A}]{\Cat}.
    \]
  \item\label{case:change-B-suspend}
    $A'=A;B'=B-1;\TupV{d}'=\TupV{d};d_{B-1}=0$. In this case, it
    follows from \Equation{C-of-x} that for all objects $x$,
    $\Cell{x}'=[0,R]\times\Cell{x}$; and for all objects $x,y$, with
    $\Cell[y]{x}'=[0,R]\times\Cell[y]{x}$. Furthermore, in the
    attaching maps, the whole of $\{0,R\}\times\Cell{x}$ is
    quotiented to the basepoint. Therefore, we have a homeomorphism
    \[
    \Realize[{\iota',\Frame', B',A'}]{\Cat}\cong
    S^1\wedge\Realize[{\iota,\Frame,B,A}]{\Cat}=
    \Sigma^{C_{\TupV{d}'}(B',A')-C_{\TupV{d}}(B,A)}\Realize[{\iota,\Frame,B,A}]{\Cat}.
    \]
  \item $A'=A;B'=B;\Tup{d}'_m=\Tup{d}_m+1$ for some $m\in[B,A-1]$ and
    $d'_i=d_i$ otherwise. In this case,
    $C_{\TupV{d}'}(B',A')-C_{\TupV{d}}(B,A)=1$, and we want to produce
    a homotopy equivalence between $\Realize[{\iota',\Frame',
      B',A'}]{\Cat}$ and $\Sigma\Realize[{\iota,\Frame,B,A}]{\Cat}$.

    For all objects $x$ with $\gr(x)\leq m$, $\Cell{x}'$ can be
    identified with $[-\ep,\ep]\times\Cell{x}$, and for all objects
    $x$ with $\gr(x)>m$, $\Cell{x}'$ can be identified with
    $[-R,R]\times\Cell{x}$. Furthermore, for all objects $x,y$ with
    $\gr(y)\leq m$, $\Cell[y]{x}'$ can be identified with
    $[-\ep,\ep]\times\Cell[y]{x}$, and for all objects $x,y$ with
    $\gr(y)>m$, $\Cell[y]{x}'$ can be identified with
    $[-R,R]\times\Cell[y]{x}$.

    Isotoping the cells $\Cell{x}\subset\RR^N$ and the subsets
    $\Cell[y]{x}\subset\Cell{x}$ (in a consistent way) has the effect
    of homotoping the attaching maps in $\Realize{\Cat}$ (in a
    consistent way), and hence does not change the homotopy type of
    $\Realize{\Cat}$. So, for $t\in[\ep,R]$, consider the cells
    $\Cell{x}^t$ and subsets $\Cell[y]{x}^t$ defined as follows:
    For all objects $x$ with $\gr(x)\leq m$, set
    $\Cell{x}^t=\Cell{x}'$, and for all objects $x$
    with $\gr(x)>m$, set $\Cell{x}^t=
    [-t,t]\times\Cell{x}$. Similarly, for all objects $x,y$ with
    $\gr(y)\leq m$, set
    $\Cell[y]{x}^t=\Cell[y]{x}'$, and for all
    objects $x,y$ with $\gr(y)>m$, set
    $\Cell[y]{x}^t=[-t,t]\times\Cell[y]{x}$. This produces a
    homotopy equivalence between $\Realize{\Cat}^{\ep}$ and
    $\Realize{\Cat}^R=\Realize[{\iota',\Frame',B',A'}]{\Cat}$.

    Therefore, in $\Realize{\Cat}^{\ep}$, $\Cell{x}^{\ep}$ is identified with
    $[-\ep,\ep]\times\Cell{x}$, and $\Cell[y]{x}^{\ep}$ is identified with
    $[-\ep,\ep]\times\Cell[y]{x}$ for all $x,y$. As in
    \Case{change-B-suspend}, we get the homotopy equivalence
    \[
    \Realize[{\iota',\Frame',B',A'}]{\Cat}=\Realize{\Cat}^R\simeq\Realize{\Cat}^{\ep}\cong
    S^1\wedge\Realize[{\iota,\Frame,B,A}]{\Cat}.\qedhere
    \]
  \end{enumerate}
\end{proof}

\begin{prop}
  This construction from \Definition{flow-gives-space} agrees with the construction
  from \cite{CJS-gauge-floerhomotopy}.
\end{prop}
\begin{proof}
  To start, we briefly review the Cohen-Jones-Segal formulation,
  from~\cite[pp.~309--312]{CJS-gauge-floerhomotopy}. For any
  topological space $X$, let $X^+$ denote its one-point
  compactification (with the understanding that $X^+=X\amalg\pt$ if
  $X$ is already compact). Once again, choose integers $A,B$, such
  that all objects of $\Cat$ have grading in $[B,A]$. Consider the
  topological category $\JCat_{B-1}^A$ with $\Ob_{\JCat}=\{n\in\ZZ\mid
  B-1\leq n\leq A\}$ and morphisms
  \[
  \Hom_\JCat(m,n)=
  \begin{cases}
    \emptyset &  m<n\\
    \{\Id\} & m=n\\
    (\RR^{m-n-1})^+ & m>n.
  \end{cases}
  \]
  Composition is given by the map 
  \begin{align*}
    \RR^{n-p-1}\times\RR^{m-n-1}&\to \RR^{m-p-1}\\
    (t_{p+1},\dots,t_{n-1})\times(t_{n+1},\dots,t_{m-1})&\mapsto (t_{p+1},\dots,t_{n-1},0,t_{n+1},\dots,t_{m-1}).
  \end{align*}

  Given a continuous, basepoint-preserving functor $Z$ from $\JCat_{B-1}^A$ to the
  category of based topological spaces, one can form the geometric
  realization $\Realize{Z}$ of $Z$, which is obtained from
  \begin{equation}\label{eq:CJS-realize-functor}
  \coprod_{B-1\leq m\leq A} \Hom_{\JCat}(m,B-1)\wedge Z(m)/\sim,
  \end{equation}
  where
  $
  (s\circ t,x)\sim (s,Z(t)(x))
  $
  for any $s\in \Hom_\JCat(n,B-1)$, $t\in\Hom_{\JCat}(m,n)$ and $x\in
  Z(m)$.

  So, to produce a space, it suffices to produce a functor from $\JCat$
  to based topological spaces. Cohen-Jones-Segal do this as
  follows. Define
  \[
  Z(m)=\bigl(\Ob_{\Cat}(m)\times\RR^{d_B}\times\dots\times\RR^{d_{m-1}}\times(-\epsilon,\epsilon)^{d_m}\times\dots\times(-\epsilon,\epsilon)^{d_{A-1}}\bigr)^+.
  \]
  The map $\Hom_{\JCat}(m,n)\times Z(m)\to Z(n)$ is given as
  follows. There is a proper projection
  \begin{align*}
    \RR^{d_B}&\times\dots\times\RR^{d_{n-1}}\times
    \Moduli(m,n)\times(-\epsilon,\epsilon)^{d_n}\times\dots
    \times(-\epsilon,\epsilon)^{d_{A-1}}\\
    &\to \Ob(n)\times\RR^{d_B}\times\dots\times\RR^{d_{n-1}}
    \times(-\epsilon,\epsilon)^{d_n} \times\dots\times
    (-\epsilon,\epsilon)^{d_{A-1}}.\\
    \shortintertext{The framing identifies the normal bundle to
      $\Moduli_{\Cat}(m,n)$ in
      $\R^{d_n}\times\R_+\times\dots\times\R_+\times\R^{d_{m-1}}\cong
      \R_+^{m-n-1}\times\R^{d_n}\times\dots\times\R^{d_{m-1}}$ with
      $\Moduli_{\Cat}(m,n)\times(-\epsilon,\epsilon)^{d_n}\times\dots\times(-\epsilon,\epsilon)^{d_{m-1}}$.     
      So, there is an open inclusion}
    \RR^{d_B}&\times\dots\times\RR^{d_{n-1}}\times
    \Moduli(m,n)\times(-\epsilon,\epsilon)^{d_n}\times\dots\times
    \times(-\epsilon,\epsilon)^{d_{A-1}}\\
    & \into \RR_+^{m-n-1}\times\Ob(m)\times\RR^{d_B}
    \times\dots\times\RR^{d_m-1} \times(-\epsilon,\epsilon)^{d_m}
    \times\dots\times (-\epsilon,\epsilon)^{d_{A-1}}.
  \end{align*}
  Collapsing everything outside the image of the second map
  gives an umkehr map the other direction of one-point
  compactifications. Composing this with the first map gives a map
  $\Hom_{\JCat}(m,n)\wedge  Z(m)\to Z(n)$, as desired.

  It remains to identify the realization $\Realize{Z}$ of $Z$ with the
  space $\Realize{\Cat}$ of \Definition{flow-gives-space}. Consider
  the subset
  \begin{multline*}
    \CellPrime{x}= [0,R)\times (-R,R)^{\Tup{d}_B}\times[0,R)\times(-R,R)^{d_{B+1}}
    \times\dots\times[0,R)\times (-R,R)^{\Tup{d}_{m-1}}\\
    \times \{0\}\times(-\ep,\ep)^{\Tup{d}_m} \times\dots\times\{0\}\times
    (-\ep,\ep)^{\Tup{d}_{A-1}} \subset \Cell{x}.
  \end{multline*}
  Then $\bigl(\amalg_{\gr(x)=m}\CellPrime{x}\bigr)^+$ is exactly the
  space $\Hom_\JCat(m,B)\wedge Z(m)$ appearing in
  \Equation{CJS-realize-functor} (except that we have used $R$ instead
  of $\infty$).

  Under this identification, the umkehr map is the identification of
  $\CellPrime[y]{x}\subset \bdy\CellPrime{x}$ with $\Moduli(x,y)\times
  \CellPrime{y}$ from \Definition{flow-gives-space}. The proper projection
  corresponds to the projection $\Moduli(x,y)\times \CellPrime{y}\to
  \CellPrime{y}$. 
  Thus, the map $\Hom_\JCat(n,B)\wedge
  \Hom_\JCat(m,n)\wedge Z(m)\to \Hom_\JCat(n,B)\wedge Z(n)$ is the
  projection of $\CellPrime[y]{x}$ to $\CellPrime{y}$, while the map
  $\Hom_\JCat(n,B)\wedge \Hom_\JCat(m,n)\wedge Z(m)\to
  \Hom_{\JCat}(m,B)\wedge Z(m)$ is the inclusion of $\CellPrime[y]{x}$ in
  $\bdy\CellPrime{x}$. So, the identification $\sim$ in
  \Equation{CJS-realize-functor} agrees with the identification in the
  CW complex from \Definition{flow-gives-space}.
\end{proof}

\subsection{Covers, subcomplexes and quotient complexes}\label{sec:cover-sub-quot-flow}
We will construct the Khovanov flow category as
a kind of cover of the flow category for a cube. In the proof of
invariance of the Khovanov homotopy type we will consider certain
subcategories of the flow category, corresponding to subcomplexes and
quotient complexes of the Khovanov chain complex. We introduce these
three notions, of a cover, a downward closed subcategory and an
upward closed subcategory of a flow category here.

\subsubsection{Covers}\label{sec:covers}
\begin{defn}\label{def:flow-cat-cover}
  A grading-preserving functor $\Funky$ from a flow category
  $\hat\Cat$ to a flow category $\Cat$ is said to be a \emph{cover} if
  for all $x,y$ in $\Ob_{\hat\Cat}$, either
  $\Moduli_{\hat\Cat}(x,y)=\emptyset$, or $\Funky\from
  \Moduli_{\hat\Cat}(x,y)\to \Moduli_{\Cat}(\Funky(x),\Funky(y))$ is a
  $(\gr(x)-\gr(y)-1)$-map, a local diffeomorphism, and a covering map.

  We say the cover is \emph{trivial} if for all $x,y$ in
  $\Ob_{\hat\Cat}$, $\Funky\from \Moduli_{\hat\Cat}(x,y)\to
  \Moduli_{\Cat}(\Funky(x),\Funky(y))$, restricted to each component
  of $\Moduli_{\hat\Cat}(x,y)$, is a homeomorphism onto its image.
\end{defn}

If $\Funky\from\hat\Cat\to\Cat$ is a cover, and if $\iota$
is a neat immersion of $\Cat$ relative some $\TupV{d}$, the
`composition' $\iota\circ\Funky$ is a neat immersion of
$\hat\Cat$.
Furthermore, a coherent framing for $\iota$ induces a
coherent framing for $\iota\circ\Funky$; in such a situation, one can
frame $\hat\Cat$ by a perturbation (\Definition{coherent-framing-to-framed}).

\begin{remark}
  In a cover $\Funky$, the cardinality of the fiber of $\Funky\from
  \Moduli_{\hat\Cat}(x,y)\to \Moduli_{\Cat}(\Funky(x),\Funky(y))$ may depend on $x$ and $y$. This is even true for trivial covers, which perhaps should be called ``topologically trivial covers'', as they can have interesting combinatorics.
\end{remark}

\subsubsection{Subcomplexes and quotient complexes}\label{sec:sub-quot}

\begin{defn}\label{def:downward-closed}
  A full subcategory $\Cat'$ of a flow category $\Cat$ is said to be a
  \emph{downward closed subcategory} (\respectively an \emph{upward
    closed subcategory}) if for all $x,y\in\Ob_{\Cat}$ with
  $\Moduli_{\Cat}(x,y)\neq\emptyset$, $x\in\Ob_{\Cat'}$ implies
  $y\in\Ob_{\Cat'}$ (\respectively $y\in\Ob_{\Cat'}$ implies
  $x\in\Ob_{\Cat'}$). It is clear that downward closed subcategories
  and upward closed subcategories of a flow category are flow
  categories.
\end{defn}

\begin{lem}\label{lem:down-sub-up-quot}
  Let $\Cat'$ be a downward closed subcategory (\respectively an upward
  closed subcategory) of a flow category $\Cat$. Let $\Cat$ be framed,
  with $(\iota,\Frame)$ being the neat embedding and the coherent
  framing. Let $(\iota',\Frame')$ be the induced neat embedding and
  coherent framing for $\Cat'$. Let $\Realize[(\iota,\Frame)]{\Cat}'$
  be the subcomplex (\respectively the quotient complex) of
  $\Realize[(\iota,\Frame)]{\Cat}$ containing exactly the cells that correspond
  to the objects that are in $\Cat'$. Then
  \[
   \Realize[(\iota',\Frame')]{\Cat'}= \Realize[(\iota,\Frame)]{\Cat}'.
  \]
  (Here, we assume the integers $A,B$ of \Definition{flow-gives-space}
  are chosen to be the same for the two sides.)
\end{lem}

\begin{proof}
  This follows immediately from \Definition{flow-gives-space}.
\end{proof}

\begin{lem}\label{lem:wedge-cat-implies-wedge-space}
  Let $\Cat'$ and $\Cat''$ be downward closed subcategories of a
  framed flow category $\Cat$, such that any object of $\Cat$ is in
  exactly one of $\Cat'$ or $\Cat''$. Assume that $\Cat'$ and $\Cat''$
  are framed by the induced framing. Then
  $\Realize{\Cat}=\Realize{\Cat'}\vee \Realize{\Cat''}$.
\end{lem}

\begin{proof}
  Let $\Realize{\Cat}'$ (\respectively $\Realize{\Cat}''$) be the subcomplex
  of $\Realize{\Cat}$ containing precisely the cells that correspond
  to objects that are in $\Cat'$ (\respectively $\Cat''$). Since every object
  of $\Cat$ is in exactly one of $\Cat'$ or $\Cat''$,
  \begin{align*}
    \Realize{\Cat}'\cup\Realize{\Cat}''&=\Realize{\Cat}\text{ and}\\
    \Realize{\Cat}'\cap\Realize{\Cat}''&=\pt.
  \end{align*}
  Therefore,
  $\Realize{\Cat}=\Realize{\Cat}'\vee\Realize{\Cat}''$. However, from
  \Lemma{down-sub-up-quot}, $\Realize{\Cat'}=\Realize{\Cat}'$
  and $\Realize{\Cat''}=\Realize{\Cat}''$.
\end{proof}

\begin{lem}\label{lem:exact-sequence-space}
  Let $\Cat'$ be a downward closed subcategory of a framed flow
  category $\Cat$, and let $\Cat''$ be the complementary upward closed
  subcategory. Assume $\Cat'$ and $\Cat''$ are framed by the induced
  framings. Let $C^*$ denote the chain complex
  associated to a framed flow category (\Definition{framed-flow-cat}). Then the
  following hold:
  \begin{enumerate}
  \item\label{case:exact-sequence-1} If $C^*(\Cat'')$ is acyclic then the inclusion
    $\Realize{\Cat'}\into \Realize{\Cat}$ induces a
    stable homotopy equivalence.
  \item If $C^*(\Cat')$ is acyclic then the quotient map
    $\Realize{\Cat}\to\Realize{\Cat''}$ induces a
    stable homotopy equivalence.
  \item If $C^*(\Cat)$ is acyclic then the Puppe
    map $\Realize{\Cat''}\to \Sigma\Realize{\Cat'}$
    induces a stable homotopy equivalence.
  \end{enumerate}
\end{lem}
  
\begin{proof}
  By \Lemma{down-sub-up-quot}, $\Realize{\Cat'}$ is a
  subcomplex of $\Realize{\Cat}$ and $\Realize{\Cat''}$ is the
  corresponding quotient complex. Therefore, there is a short exact
  sequence of reduced cellular cochain complexes
  \[
  0\to \wt{C}^*(\Realize{\Cat''})\to \wt{C}^*(\Realize{\Cat})\to
  \wt{C}^*(\Realize{\Cat'})\to 0,
  \]
  leading to the long exact sequence of reduced cohomology groups
  \[
  \dots\to \wt{H}^{i-1}(\Realize{\Cat'})\to
  \wt{H}^i(\Realize{\Cat''})\to \wt{H}^i(\Realize{\Cat})\to
  \wt{H}^i(\Realize{\Cat'})\to \wt{H}^{i+1}(\Realize{\Cat''})\to\cdots.
  \]

  In the long exact sequence, the maps $\wt{H}^*(\Realize{\Cat})\to
  \wt{H}^*(\Realize{\Cat'})$, $\wt{H}^*(\Realize{\Cat''})\to
  \wt{H}^*(\Realize{\Cat})$ and $\wt{H}^{*-1}(\Realize{\Cat'})\to
  \wt{H}^*(\Realize{\Cat''})$ are induced by the inclusion
  $\Realize{\Cat'}\into \Realize{\Cat}$, the quotient map
  $\Realize{\Cat}\to\Realize{\Cat''}$ and the Puppe map
  $\Realize{\Cat''}\to\Sigma \Realize{\Cat'}$, respectively.

  Therefore, in \Case{exact-sequence-1} if $C^*(\Cat'')$ is acyclic,
  \Lemma{realization-well-defined-refines-chain-complex} implies that
  $\wt{H}^*(\Realize{\Cat''})=0$; thus the inclusion
  $\Realize{\Cat'}\into \Realize{\Cat}$ induces an isomorphism at the
  level of homology, and hence by Whitehead's theorem is a stable
  homotopy equivalence. The other two cases can be dealt with
  similarly.
\end{proof}

\section{A framed flow category for the cube}\label{sec:cube-flow}
In \Section{Kh-flow}, we will construct a framed flow category
refining the Khovanov chain complex. Before doing so, however, we
construct a simpler framed flow category which will be used in the
construction of the category of interest.

\subsection{The cube flow category}

\begin{definition}\label{def:cube-flow-category}
  Let $f_1\from\R\to\R$ be a Morse function with a single index zero
  critical point at $0$ and a single index one critical point at
  $1$. (For concreteness, we may choose $f_1(x)=3x^2-2x^3$.)  Let
  $f_n\from\R^n\to \R$ be the Morse function
  \[
  f_n(x_1,\dots,x_n)=f_1(x_1)+\dots+f_1(x_n).
  \]
  The \emph{$n$-dimensional cube flow category} $\CubeFlowCat(n)$ is
  the Morse flow category $\MorseFlowCat(f_n)$ of $f_n$.
\end{definition}

Let $\Cube(n)=[0,1]^n$ be the cube with the obvious CW complex
structure. The set of vertices of $\Cube(n)$ is $\{0,1\}^n$. We will
re-use the notations from \Definition{n-diagrams}. There is a grading
function $\gr$ on $\{0,1\}^n$, defined as $\gr(u)=|u|=\sum_i u_i$.
Write $v\leq_i u$ if $v\leq u$ and $\gr(u)-\gr(v)=i$.

We will use the following observations about $\CubeFlowCat(n)$.
\begin{lemma}\label{lem:cube-cat-emptyness}
  There is a grading preserving correspondence between the
  objects of $\CubeFlowCat(n)$ and the vertices of $\Cube(n)$. Furthermore:
  \begin{enumerate}
  \item For $u,v\in\{0,1\}^n$, $\Moduli_{\CubeFlowCat(n)}(u,v)$ is empty
    unless $v< u$.
  \item If $v<u$, then $\Moduli_{\CubeFlowCat(n)}(u,v)$ is
    canonically diffeomorphic to
    $\Moduli_{\CubeFlowCat(\gr(u)-\gr(v))}(\vect{1},\vect{0})$.
  \end{enumerate}
\end{lemma}

\begin{proof}
  This is clear from the form of the Morse function $f_n$.
\end{proof}

\begin{figure}
  \centering
  \includegraphics[height=2in]{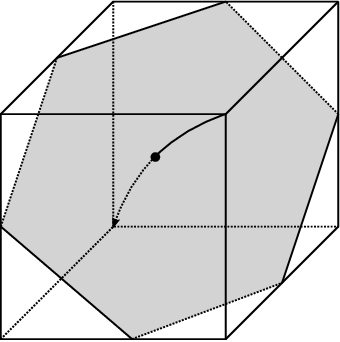}
  \caption[The cube moduli space
  $\Moduli_{\CubeFlowCat(3)}(\vect{1},\vect{0})$.]{\textbf{The cube
      moduli space $\Moduli_{\CubeFlowCat(3)}(\vect{1},\vect{0})$.}}
  \label{fig:low-d-cube}
\end{figure}

\begin{lemma}\label{lem:cube-cat-struct}
  The moduli space $\Moduli_{\CubeFlowCat(n)}(\vect{1},\vect{0})$ is
  diffeomorphic to a single point, an interval and a hexagon, for
  $n=1,2,3$, respectively.

  In general, the moduli space
  $\Moduli_{\CubeFlowCat(n+1)}(\vect{1},\vect{0})$ is homeomorphic to
  $\DD^n$ (and $\del\Moduli_{\CubeFlowCat(n+1)}(\vect{1},\vect{0})$ is
  homeomorphic to $S^{n-1}$).
\end{lemma}

\begin{proof}
  The structure of $\Moduli_{\CubeFlowCat(n)}(\vect{1},\vect{0})$ for
  small values of $n$ can be deduced from simple model computations; see
  \Figure{low-d-cube} for the case $n=3$.

  For the general case, recall from \Definition{flow-cat-of-Morse}
  that $\Moduli_{\CubeFlowCat(n+1)}(\vect{1},\vect{0})$ is a certain
  compactification of the quotient of the parametrized moduli space
  $\ParaModuli_{\CubeFlowCat(n+1)}(\vect{1},\vect{0})$ by a natural
  $\R$-action. From the form of the Morse function $f_{n+1}$, it is clear
  that $\ParaModuli_{\CubeFlowCat(n+1)}(\vect{1},\vect{0})=
  \prod_{i=1}^{n+1}\ParaModuli_{\CubeFlowCat(1)}((1),(0))$; and
  $\ParaModuli_{\CubeFlowCat(1)}((1),(0))=\R$.

  Therefore, $\Moduli_{\CubeFlowCat(n+1)}(\vect{1},\vect{0})$ is the
  compactification of the configuration space of $(n+1)$ points
  $a_1,\dots,a_{n+1}$ (not necessarily distinct) in $\R$, up to an
  overall translation in $\R$. This moduli space naturally breaks up
  into cubical chambers indexed by permutations of $\{1,\dots,n+1\}$
  as follows: For $\si$ is a permutation of $\{1,\dots,n+1\}$,
  consider the subspace of
  $\Moduli_{\CubeFlowCat(n+1)}(\vect{1},\vect{0})$ corresponding to
  configurations satisfying $a_{\si(1)}\leq\dots\leq
  a_{\si(n+1)}$. Since everything is up to an overall translation,
  such configurations are uniquely specified by the consecutive
  distances
  $(a_{\si(2)}-a_{\si(1)},a_{\si(3)}-a_{\si(2)},\dots,a_{\si(n+1)}-a_{\si(n)})$,
  and hence the compactification of such a chamber is the cube
  $[0,\infty]^n$.

  This identifies $\Moduli_{\CubeFlowCat(n+1)}(\vect{1},\vect{0})$
  with a cubical complex. However, this cubical complex can also be
  identified with the cubical subdivision of the permutohedron
  $\Permu{n+1}$---see \cite[Section 2 and Figure
  4]{Blo-gauge-spectralsequence}---thus completing the proof. (Although
  we do not need it here, but it is not hard to see that
  $\Moduli_{\CubeFlowCat(n+1)}(\vect{1},\vect{0})$ is diffeomorphic to
  $\Permu{n+1}$ as $\Codim{n}$-manifolds.)
\end{proof}

Given a pair of vertices $v\leq_m u$ of $\Cube(n)$, let
$\Cube_{u,v}=\set{x\in[0,1]^n}{\forall i\co v_i\leq x_i\leq u_i}$
be the corresponding $m$-cell of $\Cube(n)$.  It will be convenient to describe
the `faces' of $\CubeFlowCat(n)$ as follows:
\begin{definition}\label{def:CInt}
  Assume $\Cube_{u,v}$ is an $m$-cell of $\Cube(n)$. There is a
  corresponding inclusion functor
  \[
  \CInc_{u,v}\from\CubeFlowCat(m)\into \CubeFlowCat(n)
  \]
  defined as follows. Suppose $j_1<j_2<\dots<j_m$ are the $m$ indices
  where $u$ and $v$ differ.  Given an object $\ob{x}\in\{0,1\}^{m}$ in
  $\CubeFlowCat(m)$, define a new object $\ob{x}'\in \{0,1\}^n$ by
  setting the $i\th$ coordinate $x'_i$ of $\ob{x}'$ to be $1$ if
  $v_i=1$ and $0$ if $u_i=0$. For the rest, define
  $x'_{j_i}=x_i$.

  Then the full subcategory of $\CubeFlowCat(n)$ with objects
  $\{\ob{x}'\mid\ob{x}\in\Ob_{\CubeFlowCat(m)}\}$ is canonically
  isomorphic to $\CubeFlowCat(m)$ (cf.\ \Lemma{cube-cat-emptyness}),
  and $\CInc_{u,v}$ is defined to be this isomorphism.
\end{definition}

\subsection{Framing the cube flow category}\label{sec:framing}
The main goal of this section is to produce a framing of the cube flow category
$\CubeFlowCat(n)$. This is accomplished in
\Proposition{cube-can-be-framed}, by an obstruction theory
argument. (An alternate approach would be to use the fact that
any Morse flow category is framed.)

Consider $C^*(\Cube(n),\Z)$, the (cellular) cochain complex of
$\Cube(n)$.  Let $\Cube^{u,v}$ denote the cochain that sends the cell
$\Cube_{u,v}$ to $1$, and the rest of the cells to $0$. The
differentials in $C^*(\Cube(n),\F_2)$ are fairly easy to write down,
viz.\
\[
\diff\Cube^{u,v}=\sum_{v'\in\set{v'}{v'\imprec v}}\Cube^{u,v'}+\sum_{u'\in\set{u'}{u\imprec
    u'}}\Cube^{u',v}.
\]
To define $C^*(\Cube(n),\Z)$, we need to sign-refine the above cochain
complex. Towards this end, we will need the standard sign assignment.

\begin{defn}\label{def:sign-assignment}
  Let $1_i\in C^i(\Cube(n),\F_2)$ be the $i$-cocycle which assigns $1$
  to each $i$-cell.  A \emph{sign assignment} is a $1$-cochain $s\in
  C^1(\Cube(n),\F_2)$, such that $\diff s=1_2$. 

  Since $H^2(\Cube(n),\F_2)=0$, such a sign assignment exists. Since
  $H^1(\Cube(n),\F_2)=0$, any two such sign assignments $s,s'$ are
  related by a coboundary, i.e.\ $s-s'=\diff t$ for some $t\in
  C^0(\Cube(n),\F_2)$, or in other words, any two sign assignments are
  \emph{gauge equivalent}.

  One often works with the sign assignment $s_0$, called the
  \emph{standard sign assignment}, which assigns to the edge joining
  $(\ep_1,\dots,\ep_{i-1},0,\ep_{i+1},\dots,\ep_n)$ and
  $(\ep_1,\dots,\ep_{i-1},1,\ep_{i+1},\dots,\ep_n)$, the number
  $(\ep_1+\dots+\ep_{i-1})\pmod 2\in\F_2$; see also
  \Definition{knot-diag-to-kh-chain-complex}.
\end{defn}

\begin{definition}\label{def:cube-cochain-cx}
  Given a sign assignment $s$ for $\Cube(n)$, the
  \emph{$n$-dimensional cube complex} $C^*_s(n)$ is defined as
  follows. The group $C^*_s(n)$ is freely generated by $\{0,1\}^n$,
  the vertices of $\Cube(n)$. The differential is given by
  \begin{equation}\label{eq:d-on-cube}
  \diff v=\sum_{u\in\set{u}{v\imprec u}}(-1)^{s(\Cube_{u,v})}u.
  \end{equation}
\end{definition}

We will frame the cube flow category so that it refines the
cube complex $C^*_s(n)$. To get started, we will need to orient the
cube moduli spaces. (The orientations we choose here may not be the
ones induced from the framing we will produce.) To do so, first orient
each cell of $\Cube(n)$ by the product orientation. It is easy to see
that the differential in $C^*(\Cube(n),\Z)$ can be written as
\begin{equation}\label{eq:cubical-co-d}
\diff\Cube^{u,v}=\sum_{v'\in\set{v'}{v'\imprec
    v}}(-1)^{s_0(\Cube_{u-v',u-v})}\Cube^{u,v'}-\sum_{u'\in\set{u'}{u\imprec
    u'}}(-1)^{s_0(\Cube_{u'-v,u-v})}\Cube^{u',v}
\end{equation}
where $s_0$ is the standard sign assignment.

\begin{definition}
  A \emph{coherent orientation} for $\CubeFlowCat(n)$ is a choice of
  orientations for all the moduli space $\Moduli(u,v)$, satisfying the
  following conditions.
  \begin{enumerate}
  \item If $u\imprec u'$, then the orientation of $\Moduli(u,v)$,
    viewed as the subspace
    $\Moduli(u,v)\times\Moduli(u',u)\sbs\del\Moduli(u',v)$, is
    $-(-1)^{s_0(\Cube_{u'-v,u-v})}$ times the boundary orientation.
  \item If $v'\imprec v$, then the orientation of $\Moduli(u,v)$,
    viewed as the subspace
    $\Moduli(v,v')\times\Moduli(u,v)\sbs\del\Moduli(u,v')$, is
    $(-1)^{s_0(\Cube_{u-v',u-v})}$ times the boundary orientation.
  \end{enumerate}
\end{definition}

\begin{lem}\label{lem:coherent-orient-moduli-cube}
  There exists a coherent orientation for $\CubeFlowCat(n)$.
\end{lem}

\begin{proof}
  Orient all the cells of $\Cube(n)$ so that the
  differential in $C^*(\Cube(n),\Z)$ satisfies
  \Equation{cubical-co-d}. Recall from
  \Definition{flow-cat-of-Morse} that the interior of the cell
  $\Cube_{u,v}$ is diffeomorphic to $\interior{\Moduli}(u,v)\times\R$,
  where $\R$ is oriented along the flowlines, i.e.\ the value of the
  Morse function decreases along $\R$. 
  Now orient
  $\Moduli(u,v)$ so that the product orientation on
  $\interior{\Moduli}(u,v)\times\R$ agrees with the orientation of
  $\interior{\Cube}_{u,v}$.
\end{proof}

We digress briefly to talk about orthogonal groups. Treat the
orthogonal group $O(n)$ as a subgroup of $O(n+1)$, by identifying
$\R^n$ with the subspace $\R^n\times\{0\}$ in $\R^{n+1}$. Let
$O=\colim O(n)$. Since $O(n)$ is a topological group, the action of
$\pi_1(O(n),\Id)$ on $\pi_i(O(n),\Id)$ is trivial;
or in
other words, we can omit the basepoint while talking about the
homotopy groups of $O(n)$. For $n>i+1$, the map
$\pi_i(O(n))\to\pi_i(O)$ is an isomorphism.

Given an embedding
$(\DD^k,\del\DD^k)\into(\R^m\times\R_+,\R^m\times\{0\})$, a framing of
the normal bundle of $\del\DD^k$ is said to be \emph{null-concordant}
if it extends to a framing of the normal bundle of $\DD^k$.

\begin{defn}\label{def:obstruction-class}
  Fix a coherent orientation for $\CubeFlowCat(n)$ and a neat
  embedding $\iota$ of $\CubeFlowCat(n)$ relative some $\TupV{d}$. A
  coherent framing of all moduli spaces of dimension less than $k$
  produces an \emph{obstruction class} $\mf{o}\in
  C^{k+1}(\Cube(n),\pi_{k-1}(O))$ as follows.

  Consider some $(k+1)$-cell $\Cube_{u,v}$. The boundary
  $\del(\iota_{u,v}\Moduli(u,v))$ is an oriented $(k-1)$-sphere $S^{k-1}$ embedded
  in $\del\ESpace[\TupV{d}]{\gr(v)}{\gr(u)}$. Furthermore, since 
  $\del(\iota_{u,v}\Moduli(u,v))$ is identified with products of
  lower-dimensional moduli spaces, which are already framed, 
  $\del(\iota_{u,v}\Moduli(u,v))$ is framed by the product framing,
  which we denote $\Frame_{u,v}$.  Choose some
  null-concordant framing $\Frame^0_{u,v}$ of $S^{k-1}$. Thus we get
  the following map from $S^{k-1}$ to
  $O(d_{\gr(v)}+\dots+d_{\gr(u)-1})$: send a point $x\in S^{k-1}$ to
  the element of the orthogonal group that maps $\Frame^0_{u,v}(x)$ to
  $\Frame_{u,v}(x)$. Different choices for $\Frame^0_{u,v}$ produce
  homotopic maps;\footnote{Two different null-concordant frames
    produce a map from $S^{k-1}$ to the orthogonal group. Their
    null-concordances produce the null-homotopy of this map.}
  therefore, we get a well-defined element in
  $\pi_{k-1}(O(d_{\gr(v)}+\dots+d_{\gr(u)-1}))$, and hence an element
  $\mf{o}(\Cube_{v,w})\in\pi_{k-1}(O)$. Collecting all the terms, we
  get a cochain $\mf{o}\in C^{k+1}(\Cube(n),\pi_{k-1}(O))$. 
\end{defn}

By definition, $\mf{o}$ vanishes if and only if the framing of the
moduli spaces of dimension less than $k$ extends to a framing of the
$k$-dimensional moduli spaces.

Again, note that the framings of the moduli spaces need not be related to
the orientations of the moduli spaces. The framings are needed to
construct maps from $(k-1)$-spheres to $O$, while the orientations are
needed to orient the spheres, and thereby interpret those maps as
elements of $\pi_{k-1}(O)$.

\begin{lem}\label{lem:obstruction-is-cocycle}
  The obstruction class $\mf{o}$ from \Definition{obstruction-class}
  is a cocycle.
\end{lem}

\begin{proof}
  It is a standard fact from algebraic topology that obstruction
  classes are cocycles. The proof adapts easily to our setting; we
  include it here for completeness.

  Fix any $(k+2)$-cell $\Cube_{u,v}$. We want to show
  $\langle\diff\mf{o},\Cube_{u,v}\rangle=0$. 

  Let $S^k$ be the oriented $k$-sphere
  $\del(\iota_{u,v}\Moduli(u,v))$. Let $W=\set{w}{v\imprec w<
    u\text{ or }v< w\imprec u}$. For $w\in W$,
  let $\DD^k_w\sbs S^k$ be the $k$-cell
  $\iota_{u,v}(\Moduli(w,v)\times\Moduli(u,w))=\iota_{w,v}\Moduli(w,v)\times
  \iota_{u,w}\Moduli(u,w)$, oriented as a subspace of $S^k$; let
  $S^{k-1}_w=\del\DD^k_w$ be the oriented boundary. Let
  $B^k=S^k\sm(\bigcup_{w\in W} \interior{\DD}^k_w)$. Given any map
  $\Phi\from B^k\to O$, let $\Phi_w\in\pi_{k-1}(O)$ be
  the element induced by $\restrict{\Phi}{S^{k-1}_w}$. 

  We claim that $\sum_{w\in W} \Phi_w=0$ in $\pi_{k-1}(O)$. To see
  this, let $\DD'_w$ denote a slightly smaller ball concentric with
  $\DD^k_w$ and let $S'_w=\bdy \DD'_w$. Let $B'=S^k\sm(\bigcup_{w\in
    W} \interior{\DD}'_w)$, so $B'$ is the complement of a collection
  of balls in $S^k$ with disjoint closures.  The map $\Phi$ extends to
  a map $\Phi'\co B'\to O$. Let $p$ be a point in the interior of $B'$
  and let $\gamma_w$ be an embedded path from $p$ to $S'_w$ in $B'$
  (so that $\gamma_w$ is disjoint from $\gamma_{w'}$ if $w\neq w'$). Then
  $B'\setminus \bigcup_w \gamma_w$ is homeomorphic to a ball. It
  follows that $\sum_{w\in W}\restrict{\Phi}{S'_w}=0$ in
  $\pi_{k-1}(O)$. But $\sum_{w\in W} \Phi_w=\sum_{w\in
    W}\restrict{\Phi}{S'_w}$.

  Next, since all moduli spaces of dimension less than $k$ are framed, we
  get the product framing $\Frame$ of the bundle $\nu_{\iota_{u,v}}$
  restricted to $B^k$.  Choose some framing $\Frame^0$ of the
  whole normal bundle $\nu_{\iota_{u,v}}$. (We do not impose any
  coherence conditions; since $\Moduli(u,v)$ is a disk, such a framing
  exists.)  We construct our map $\Phi\from B^k\to O$ as follows: a
  point $x\in B^k$ is mapped to the element of the orthogonal group
  that sends $\Frame^0(x)$ to $\Frame(x)$. From the earlier
  observation, $\sum_{w\in W}\Phi_w=0$.

  Now fix some $w\in W$; we will analyze $\Phi_w$. Let $\Frame^0_w$
  denote the restriction of $\Frame^0$ to $S^{k-1}_w$.   We have the following two cases:
  \begin{enumerate}
  \item $v\imprec w< u$. In this case, $\iota_{w,v}\Moduli(w,v)=\pt$,
    and the restriction of $\Frame$ to
    $S^{k-1}_w=\pt\times\del(\iota_{u,w}\Moduli(u,w))\sbs\R^{\Tup{d}_{\gr(v)}}\times
    \del\ESpace[\TupV{d}]{\gr(w)}{\gr(v)}$ is a product framing, say
    $\Frame_{w,v}\wedge\Frame_{u,w}$. 

    Since the moduli spaces are coherently oriented, the orientation
    of $\del(\iota_{u,w}\Moduli(u,w))$ is
    $(-1)^{s_0(\Cube_{u-v,u-w})}$ times the orientation
    of $S^{k-1}_w$. Choose some null-concordant framing
    $\Frame^0_{u,w}$ of $\del\Moduli(u,w)$; $\mf{o}(\Cube_{u,w})$
    relates $\Frame^0_{u,w}$ to $\Frame_{u,w}$ on
    $\del(\iota_{u,w}\Moduli(u,w))$. Therefore,
    $(-1)^{s_0(\Cube_{u-v,u-w})}\mf{o}(\Cube_{u,w})$
    relates $\Frame_{w,v}\wedge\Frame^0_{u,w}$ to
    $\Frame_{w,v}\wedge\Frame_{u,w}$ on $S^{k-1}_w$.

    However, $\Phi_w$ relates $\Frame^0_w$ to
    $\Frame_{w,v}\wedge\Frame_{u,w}$ on $S^{k-1}_w$. Since both
    $\Frame^0_w$ and $\Frame_{w,v}\wedge\Frame^0_{u,w}$ are
    null-concordant, i.e.\ they extend to $\DD^k_w$, we have
    $\Phi_w=(-1)^{s_0(\Cube_{u-v,u-w})}\mf{o}(\Cube_{u,w})$.
  \item $v< w\imprec u$. An argument similar to the one in the
    previous case shows
    $\Phi_w=-(-1)^{s_0(\Cube_{u-v,w-v})}\mf{o}(\Cube_{w,v})$.
  \end{enumerate}
  
  Since $\sum_{w\in W}\Phi_w=0$, we get
  \[
  \sum_{v\imprec
    w<u}(-1)^{s_0(\Cube_{u-v,u-w})}\mf{o}(\Cube_{u,w})-
  \sum_{v<w\imprec u}(-1)^{s_0(\Cube_{u-v,w-v})}\mf{o}(\Cube_{w,v})=0.
  \]
  Using \Equation{cubical-co-d} we conclude $\langle\diff\mf{o},\Cube_{u,v}\rangle=0$. 
\end{proof}

\begin{lem}\label{lem:modify-frame-by-cocycle}
  Fix a coherent orientation for $\CubeFlowCat(n)$, a neat embedding
  $\iota$ of $\CubeFlowCat(n)$ relative some large enough $\TupV{d}$,
  a coherent framing of all moduli spaces of dimension less than $k$,
  and a cochain $\mf{p}\in C^k(\Cube(n),\pi_{k-1}(O))$.

  For each $k$-cell $\Cube_{u,v}$, choose a small oriented disk
  $\DD^{k-1}$ in the interior of the oriented manifold
  $\iota_{u,v}\Moduli(u,v)$, and a map
  $f_{u,v}\from(\DD^{k-1},\del\DD^{k-1})\allowbreak\to
  (O(d_{\gr(v)}+\dots+d_{\gr(u)-1}),\Id)$ that induces
  $\mf{p}(\Cube_{u,v})\in\pi_{k-1}(O)$; then change the framing on the
  disk $\DD^{k-1}$ of the framed manifold
  $\iota_{u,v}\Moduli(u,v)\sbs\ESpace[\TupV{d}]{\gr(v)}{\gr(u)}$ by
  $f_{u,v}$.

  Let $\mf{o}$ (\respectively $\mf{o}'$) be the obstruction class from
  \Definition{obstruction-class} before (\respectively after) this
  modification. Then, $\mf{o}'=\mf{o}+\diff\mf{p}$
\end{lem}

\begin{proof}
  This is clear. The signs work out the same way as in the proof
  of \Lemma{obstruction-is-cocycle}.
\end{proof}

\begin{prop}\label{prop:cube-can-be-framed}
  Given a sign assignment $s$ for $\Cube(n)$, the cube flow category
  $\CubeFlowCat(n)$ can be framed in a way so that it refines the cube
  complex $C^*_s(n)$.
\end{prop}

\begin{proof}
  Fix a coherent orientation of $\CubeFlowCat(n)$ and a neat embedding
  $\iota$ of $\CubeFlowCat(n)$ relative to some large enough
  $\TupV{d}$. We will produce a coherent framing $\Frame$ for $\iota$
  so that the framed flow category $(\Cat,\iota,\Frame)$ refines the
  cube complex $C^*_s(n)$.

  First, we frame the $0$-dimensional moduli spaces. Given $v\imprec
  u$, frame the point
  $\iota_{u,v}\Moduli(u,v)\in\ESpace[\TupV{d}]{\gr(v)}{\gr(u)}=\R^{d_{\gr(v)}}$
  so that the frame is positive in the tangent space $T\R^{d_{\gr(v)}}$ if and only
  if $s(\Cube_{u,v})=0$.

  Next, we frame the $1$-dimensional moduli spaces. Fix $v\leq_2
  u$. We are trying to frame the normal bundle of
  $\iota_{u,v}\Moduli(u,v)\sbs\ESpace[\TupV{d}]{\gr(v)}{\gr(u)}$. Observe
  that its boundary sphere $S^0$ is already framed by the product
  framing. Furthermore, since $s$ is a sign assignment, the framing on
  $S^0$ is null-concordant. Therefore, we can extend it to a framing of
  $\Moduli(u,v)$. Choose some extension (which we might need to change
  later).

  We do the rest inductively. For some $k\geq 2$, assume that we have
  framed all moduli spaces of dimension less than $k$. Further assume
  that the framings of the zero dimensional moduli spaces come from
  the sign assignment $s$. After changing the framings in the interior
  of the $(k-1)$-dimensional moduli spaces if necessary, we will
  extend the framings to all $k$-dimensional moduli spaces.

  \Definition{obstruction-class} furnishes us with an obstruction
  class $\mf{o}\in
  C^{k+1}(\Cube(n),\pi_{k-1}(O))$ to extending the framing. We know from \Lemma{obstruction-is-cocycle}
  that $\mf{o}$ is a cocycle. Since $\Cube(n)$ is acyclic,
  $\mf{o}$ is a coboundary as well. Choose a cocycle $\mf{p}\in
  C^k(\Cube(n),\pi_{k-1}(O))$ such that $\diff \mf{p}=-
  \mf{o}$. Modify the framings in the interior of the
  $(k-1)$-dimensional moduli spaces using the cocycle $\mf{p}$, as in
  \Lemma{modify-frame-by-cocycle}. After the modification, the new
  obstruction class vanishes, and hence we can extend the framings to
  the $k$-dimensional moduli spaces.
\end{proof}

\begin{lem}\label{lem:change-framing-cube}
  For any fixed sign assignment $s$ for $\Cube(n)$, let
  $(\iota_0,\Frame_0)$ and $(\iota_1,\Frame_1)$ be two framings of
  $\CubeFlowCat(n)$ that refine the cube complex $C^*_s(n)$. Then for
  some $\TupV{d}_0,\TupV{d}_1$, $(\iota_0[\TupV{d}_0],\Frame_0)$ can be
  connected to $(\iota_1[\TupV{d}_1],\Frame_1)$ by a smooth $1$-parameter
  family of framings.
\end{lem}

\begin{proof}
  Choose $\TupV{d}_0$ and $\TupV{d}_1$ large enough; and then by using
  \Lemma{connect-neat-immersions}, for $t\in[0,1]$, choose a
  $1$-parameter family $\iota(t)$ of neat embeddings of
  $\CubeFlowCat(n)$ with $\iota(0)=\iota_0[\TupV{d}_0]$ and
  $\iota(1)=\iota_1[\TupV{d}_1]$. We would like to extend $\Frame_0$
  and $\Frame_1$ to a
  $1$-parameter family of coherent framings $\Frame_t$ for $\iota(t)$.

  We will frame the moduli spaces $\iota_{u,v}(t)(\Moduli(u,v))$
  inductively on $|u-v|$. When $v\imprec u$, both the points
  $\iota_{u,v}(0)(\Moduli(u,v))$ and $\iota_{u,v}(1)(\Moduli(u,v))$
  are framed with the same sign: positively if $s(\Cube_{u,v})=0$ and
  negatively if $s(\Cube_{u,v})=1$. Therefore, there exists a
  $1$-parameter families of framings of $\iota_{u,v}(t)$ connecting the
  given framings at the endpoints. This takes care of the base case.

  Now assume that for some $k\geq 1$, we have produced $1$-parameter
  family of framings $\Frame_t$ of $\iota_{u,v}(t)$ for all $u,v$ with
  $|u-v|\leq k$. After changing $\Frame_t$ relative endpoints for the
  $(k-1)$-dimensional moduli spaces if necessary, we will extend
  $\Frame_t$ to all $k$-dimensional moduli spaces.

  The proof is similar to the proof of
  \Proposition{cube-can-be-framed}. Fix a coherent orientation of
  $\CubeFlowCat(n)$. For any $(k+1)$-cell $\Cube_{u,v}$, consider the
  $k$-sphere $S^k=\del(\Moduli(u,v)\times[0,1])$, oriented as the
  boundary of a product. We already have a framing on $S^k$: the
  framing on $\Moduli(u,v)\times\{0\}$ is the pullback of the framing
  $\Frame_0$ on $\iota_{u,v}$; the framing on
  $\Moduli(u,v)\times\{1\}$ is the pullback of the framing $\Frame_1$
  on $\iota_{u,v}$; and the framing on $\del\Moduli(u,v)\times[0,1]$
  is the pullback of the product of framings of lower dimensional
  moduli spaces. Comparing this framing with any other null-concordant
  framing of $S^k$ produces an element of $\pi_k(O)$. Collecting all
  these elements, we get a cochain $\mf{o}\in
  C^{k+1}(\Cube(n),\pi_k(O))$.

  Arguments similar to the ones in the proof of
  \Lemma{obstruction-is-cocycle} show that this obstruction class
  $\mf{o}$ is a cocycle, and hence a coboundary. Choose $\mf{p}\in
  C^k(\Cube(n),\pi_k(O))$ such that $\diff\mf{p}=-\mf{o}$. For
  $k$-cells $\Cube_{u,v}$, modify the framings for
  $\Moduli(u,v)\times[0,1]$ in the interior by
  $\mf{p}(\Cube_{u,v})$. After the modification, the new obstruction
  class vanishes, thus completing the proof.
\end{proof}

\section{A flow category for Khovanov homology}\label{sec:Kh-flow}
We return to the notations from \Section{res-config}, where $D$
is a link diagram. The goal of this section is to associate a framed
flow category to $D$; feeding this framed flow category into the
Cohen-Jones-Segal construction, as described
in \Section{flow-to-space}, then gives a suspension spectrum.

\subsection{Moduli spaces for decorated resolution configurations}
\label{sec:mod-dec-res-config}
The main work in this section will be to associate to each index $n$
basic decorated resolution configuration $(D,\gen{x},\gen{y})$ an
$(n-1)$-dimensional $\Codim{n-1}$-manifold
$\Moduli(D,\gen{x},\gen{y})$ together with an $(n-1)$-map
\[
\Forget\co \Moduli(D,\gen{x},\gen{y})\to \Moduli_{\CubeFlowCat(n)}(\vect{1},\vect{0}).
\]
These spaces and maps will be built inductively, satisfying the following
properties:
\begin{enumerate}[label=(RM-\arabic*),ref=RM-\arabic*]
\item\label{condition:res-moduli-compose} Let $(D,\gen{x},\gen{y})$ be
  a basic decorated resolution configuration, and $(E,\gen{z})\in
  P(D,\gen{x},\gen{y})$. Let $\gen{x}|$ and $\gen{y}|$ denote the
  induced labelings on $s(E\sm s(D))=s(D)\sm E$\footnote{$s(D)=s(E)$;
    therefore, $s(E\sm s(D))=s(D)\sm E$ using
    \Lemma{res-config-surgery}.} and $D\sm E$, respectively; by an
  abuse of notation, let $\gen{z}|$ denote the induced labelings on
  both $s(D\sm E)=E\sm D$ and $E\sm s(D)$. Then there is a
  \emph{composition map}
  \[
  \circ \co \Moduli(D\sm E,\gen{z}|,\gen{y}|)\times
  \Moduli(E\sm s(D),\gen{x}|,\gen{z}|)\to \Moduli(D,\gen{x},\gen{y}),
  \]
  respecting the map $\Forget$ in the sense that the following diagram
  commutes:
  \begin{equation}\label{eq:Forget-res}
    \begin{split}
  \xymatrix{
    \Moduli(D\sm E,\gen{z}|,\gen{y}|)\times
  \Moduli(E\sm s(D),\gen{x}|,\gen{z}|)
  \ar[r]^-\circ\ar[d]_{\Forget\times\Forget} &
  \Moduli(D,\gen{x},\gen{y})\ar[dd]^{\Forget}\\
  \Moduli_{\CubeFlowCat(n-m)}(\vect{1},\vect{0})\times
  \Moduli_{\CubeFlowCat(m)}(\vect{1},\vect{0})
  \ar[d]_{\CInc_{{v},\vect{0}}\times \CInc_{\vect{1},{v}}}&
  \\
  \Moduli_{\CubeFlowCat(n)}({v},\vect{0}) \times
  \Moduli_{\CubeFlowCat(n)}(\vect{1},{v})\ar[r]^-\circ & 
  \Moduli_{\CubeFlowCat(n)}(\vect{1},\vect{0}).
  }
  \end{split}
  \end{equation}
  Here, $n$ is the index of $D$; $m$ is the index of $E$; and the
  vector ${v}=(v_1,\dots,v_n)$ has $v_i=0$ if the $i\th$ arc in
  the totally ordered set $A(D)$ is in $A(E)$, and $v_i=1$ otherwise.
  (The maps $\CInc$ are given in \Definition{CInt}.)
\item The faces of $\Moduli(D,\gen{x},\gen{y})$ are given by
  \begin{equation}\label{eq:res-expected-bdy}
    \begin{split}
    \bdy_i \Moduli(D,\gen{x},\gen{y})&=
    \bdy_{\expect,i}\Moduli(D,\gen{x},\gen{y})\coloneqq\\
    &\coprod_{\substack{(E,\gen{z})\in
        P(D,\gen{x},\gen{y})\\ \ind(D\sm E)=i}} \circ\bigl(\Moduli(D\sm E,\gen{z}|,\gen{y}|)\times
    \Moduli(E\sm
    s(D),\gen{x}|,\gen{z}|)\bigr).
    \end{split}
\end{equation}
(Here, $\expect$ stands for ``expected.'')
\item\label{condition:res-moduli-covers} The map $\Forget$ is a covering
  map, and a local diffeomorphism. In fact:
\item\label{condition:res-moduli-cover-trivial} The covering map
  $\Forget$ is trivial on each component of $\Moduli(D,\gen{x},\gen{y})$.
\end{enumerate}

\Condition[s]{res-moduli-compose}--\ConditionCont{res-moduli-covers} together imply analogues of
\Condition[s]{maps-to-bdy}
and~\ConditionCont{boundary-is} of
\Definition{flow-cat}.
In particular, the boundary of $\Moduli(D,\gen{x},\gen{y})$ is given by
a pushout, similarly to \Condition{boundary-is}. Let $\mf{D}$ denote the diagram whose
vertices are the spaces
\[
\Moduli(D\sm E_k,\gen{z}_k|,\gen{y}|)\times \Moduli(E_k\sm
E_{k-1},\gen{z}_{k-1}|,\gen{z}_k|)\times\dots \times \Moduli(E_1\sm
s(D),\gen{x}|,\gen{z}_1|)
\]
where $k\geq 1$ and $(D,\gen{y})\prec (E_k,\gen{z}_k)\prec\dots\prec
(E_1,\gen{z}_1)\prec (s(D),\gen{y})$ is a chain in
$P(D,\gen{x},\gen{y})$; and whose edges correspond to applying the
composition map $\circ$ to an adjacent pair of spaces. Let
\[
\bdy_{\expect}\Moduli(D,\gen{x},\gen{y})=\colim \mf{D}=\bigcup_i\bdy_{\expect,i}\Moduli(D,\gen{x},\gen{y}).
\]
Note that $\bdy_\expect\Moduli(D,\gen{x},\gen{y})$ is naturally an
$\Codim{n-1}$-boundary (\Definition{n-boundary}), whose
top-dimensional faces are exactly the
$\bdy_{\expect,i}\Moduli(D,\gen{x},\gen{y})$.

\begin{lemma}\label{lem:bdy-is-bdy-expect}
  $\del_{\expect}\Moduli(D,\gen{x},\gen{y})=\del\Moduli(D,\gen{x},\gen{y})$.
\end{lemma}
\begin{proof}
  This is immediate from \Equation{res-expected-bdy} and the
  fact that $\Moduli(D,\gen{x},\gen{y})$ is an $\Codim{n}$-manifold.
\end{proof}

In the inductive construction of the $\Moduli(D,\gen{x},\gen{y})$ we will
start by constructing
$\Moduli(D,\gen{x},\gen{y})$ when $\ind(D)=1$; this moduli space will
always be a single point. Next, suppose that
$\Moduli(D,\gen{x},\gen{y})$ has been defined whenever $\ind(D)<k$. We
will build the moduli spaces $\Moduli(D,\gen{x},\gen{y})$ for
$\ind(D)=k$ by filling in $\bdy_{\expect}\Moduli(D,\gen{x},\gen{y})$
(which has already been defined) subject to
\Condition[s]{res-moduli-compose}--\ConditionCont{res-moduli-cover-trivial}. The
main work will be in checking that the restriction of $\Forget$ to
$\bdy_{\expect}\Moduli(D,\gen{x},\gen{y})$ is a trivial covering.

More formally:
\begin{prop}\label{prop:extend-moduli}
  Suppose that the moduli spaces $\Moduli(D,\gen{x},\gen{y})$ and maps
  $\Forget\from \Moduli(D,\gen{x},\gen{y})\to\allowbreak
  \Moduli_{\CubeFlowCat(\ind(D))}(\vect{1},\vect{0})$ have already been
  constructed for each basic decorated resolution configuration
  $(D,\gen{x},\gen{y})$ with $\ind(D)\leq n$, satisfying
  \Condition[s]{res-moduli-compose}--\ConditionCont{res-moduli-cover-trivial}.
  Then, given a decorated resolution configuration
  $(D,\gen{x},\gen{y})$ with $\ind(D)=n+1$:
  \begin{enumerate}[label=(E-\arabic*),ref=E-\arabic*]
  \item\label{item:extend-assemble} The maps $\Forget$ already defined
    assemble to give a continuous map
    \[
    \Forget|_{\bdy}\co \bdy_{\expect}\Moduli(D,\gen{x},\gen{y})\to \bdy\Moduli_{\CubeFlowCat(n+1)}(\vect{1},\vect{0}).
    \]
    Furthermore, this map respects the $\Codim{n}$-boundary structure
    of the two sides, i.e.\
    $\Forget|_{\del}^{-1}(\del_i\Moduli(\vect{1},\vect{0}))=\del_{\expect,i}\Moduli(D,\gen{x},\gen{y})$
    and $\Forget|_{\bdy_{\expect,i}\Moduli(D,\gen{x},\gen{y})}$ is an
    $n$-map.
  \item\label{item:extend-cover} The map $\Forget|_\bdy$ is a covering
    map.
  \item\label{item:extend-exist} There is an $n$-dimensional
    $\Codim{n}$-manifold $\Moduli(D,\gen{x},\gen{y})$ and an $n$-map
    $\Forget\co \Moduli(D,\gen{x},\gen{y})\to
    \Moduli_{\CubeFlowCat(n+1)}(\vect{1},\vect{0})$ satisfying
    \Condition[s]{res-moduli-compose}--\ConditionCont{res-moduli-cover-trivial}
    if and only if the map $\Forget|_{\bdy}$ is a trivial covering
    map.

    In particular, if $n\geq 3$ then the space
    $\Moduli(D,\gen{x},\gen{y})$ and map $\Forget$ necessarily exist.
  \item\label{item:extend-unique} If the space
    $\Moduli(D,\gen{x},\gen{y})$ and map $\Forget\co
    \Moduli(D,\gen{x},\gen{y})\to
    \Moduli_{\CubeFlowCat(n+1)}(\vect{1},\vect{0})$ exist, and if $n\geq 2$,
    then $\Moduli(D,\gen{x},\gen{y})$ and $\Forget$ are unique up to
    diffeomorphism (fixing the boundary).
  \end{enumerate}
\end{prop}
\begin{proof}
  For notational convenience, let $\CubeFlowCat=\CubeFlowCat(n+1)$.

  We start with \Part{extend-assemble}. The fact that
  $\Forget|_\bdy$ is well-defined follows from commutativity of
  \Diagram{Forget-res}. In more detail, let $\mf{E}$ denote
  the diagram whose vertices are the spaces
  \[
  \Moduli_{\CubeFlowCat}({v}_k,\vect{0})\times \Moduli_{\CubeFlowCat}({v}_{k-1},{v}_k)\times\dots\times\Moduli_{\CubeFlowCat}(\vect{1},{v}_1),
  \]
  where ${v}_i\in\{0,1\}^{n+1}$, $\vect{0}<{v}_k<\dots<{v}_1<\vect{1}$,
  and $k\geq 1$; and whose edges correspond to the composition map
  $\circ$ for $\CubeFlowCat$, applied to an adjacent pair of
  spaces. Since $\CubeFlowCat$ is a flow category, it follows from
  \Condition{boundary-is} of
  \Definition{flow-cat} that composition induces a
  homeomorphism
  \[
  \colim \mf{E} \cong \bdy\Moduli_{\CubeFlowCat}(\vect{1},\vect{0}).
  \]
  
  The maps $\Forget$ already defined induce a map of diagrams
  $\mathcal{G}\co \mf{D}\to \mf{E}$, where $\mathcal{G}$ is given on $\Moduli(D\sm E_k,\gen{z}_k|,\gen{y}|)\times \Moduli(E_k\sm
  E_{k-1},\gen{z}_{k-1}|,\gen{z}_k|)\times\dots \times \Moduli(E_1\sm
  s(D),\gen{x}|,\gen{z}_1|)
  $ by 
  \[
  (\CInc_{v_k,\vect{0}}\times \CInc_{v_{k-1},v_k}\times\dots\times\CInc_{\vect{1},v_1})\circ (\Forget\times\dots\times\Forget)
  \]
  where $v_m$ has $i\th$ entry $0$ if the $i\th$ arc of $A(D)$ is in
  $A(E_m)$ and $1$ otherwise. It follows from
  \Diagram{Forget-res} that $\mathcal{G}$ defines a map of
  diagrams.  Let $\Forget|_\bdy=\colim\mathcal{G}$.  Since
  $\bdy_i\Moduli(\vect{1},\vect{0})=\coprod_{|v|=i}\Moduli(v,\vect{0})\times\Moduli(\vect{1},v)$,
  we have
  $\bdy_{\expect,i}\Moduli(D,\gen{x},\gen{y})=\Forget|_\bdy^{-1}(\bdy_i\Moduli(\vect{1},\vect{0}))$;
  and the map $\Forget|_{\bdy_{\expect,i}\Moduli(D,\gen{x},\gen{y})}$
  is an $n$-map by the inductive hypothesis.

  \Part{extend-cover} follows from the previous part,
  \Proposition{n-bdy-cover} and the inductive hypothesis that
  $\Forget\co \Moduli(D',\gen{x}',\gen{y}')\to
  \Moduli_{\CubeFlowCat(\ind(D'))}(\vect{1},\vect{0})$ is a covering
  map whenever $\ind(D')\leq n$.

  For \Part{extend-exist}, recall from
  \Lemma{cube-cat-struct} that
  $\Moduli_{\CubeFlowCat}(\vect{1},\vect{0})$ is topologically an
  $n$-ball, so $\bdy \Moduli_{\CubeFlowCat}(\vect{1},\vect{0})$
  is topologically an $(n-1)$-sphere. Hence, by the previous part,
  $\bdy_\expect\Moduli(D,x,y)$ is a disjoint union of
  $(n-1)$-spheres. If the covering map $\Forget|_\bdy\co \amalg_{i=1}^m
  S^{n-1}\to S^{n-1}$ is a homeomorphism on each component then it
  extends in an obvious way to a trivial covering map $\Forget\co
  \amalg_{i=1}^m \DD^{n}\to \DD^{n}$. Conversely, if
  $\Forget|_\bdy$ extends then, since any covering space of a disk is
  trivial, the covering space $\Forget|_\bdy$ must have been trivial.

  If $\Forget|_\bdy$ extends to a covering map $\Forget$ then
  requiring $\Forget$ to be an $n$-map and a local diffeomorphism
  defines an $\Codim{n}$-manifold structure on
  $\Moduli(D,\gen{x},\gen{y})$. \Condition[s]{res-moduli-covers}
  and~\ConditionCont{res-moduli-cover-trivial} are vacuously satisfied
  by the extension $\Forget$. The composition map $\circ$ from
  \Condition{res-moduli-compose} is induced by the
  canonical maps
  \[
  \Moduli(D\sm E,\gen{z}|,\gen{y}|)\times \Moduli(E\sm
  s(D),\gen{x}|,\gen{z}|)\to \colim \mf{D}=\bdy_\expect\Moduli(D,\gen{x},\gen{y})=\bdy\Moduli(D,\gen{x},\gen{y})
  \]
  (since the left hand side is a vertex in $\mf{D}$). Since the
  restriction of $\Forget$ to $\bdy\Moduli(D,\gen{x},\gen{y})$ is by
  definition $\Forget|_\bdy$, \Diagram{Forget-res} commutes.
  
  Note that if $n\geq 3$ then any cover of $S^{n-1}$ is trivial, so
  the hypothesis is vacuously satisfied.

  For \Part{extend-unique}, it follows from
  \Lemma{bdy-is-bdy-expect} that for any other
  $\Moduli'(D,\gen{x},\gen{y})$ satisfying
  \Condition[s]{res-moduli-compose}--\ConditionCont{res-moduli-cover-trivial},
  $\bdy\Moduli'(D,\gen{x},\gen{y})$ is identified with
  $\bdy\Moduli(D,\gen{x},\gen{y})$. So, by the argument in
  \Part{extend-exist}, $\Moduli(D,\gen{x},\gen{y})$ and
  $\Moduli'(D,\gen{x},\gen{y})$ are disjoint unions of disks with the
  same boundary, and hence diffeomorphic (via a diffeomorphism
  commuting with $\Forget$ and fixing the boundary).
\end{proof}

\subsection{From resolution moduli spaces to the Khovanov flow category}
Suppose we have spaces $\Moduli(D,\gen{x},\gen{y})$ and maps $\Forget$
as above, satisfying
\Condition[s]{res-moduli-compose}--\ConditionCont{res-moduli-covers}.
Fix an oriented link diagram $L$ with $n$ crossings, and an ordering
of the crossings of $L$.  We will use this data to construct the
Khovanov flow category $\KhFlowCat(L)$ as a cover of the cube flow
category $\CubeFlowCat(n)$.

\begin{defn}\label{def:khovanov-flow-category}
  The \emph{Khovanov flow category $\KhFlowCat(L)$} has one object for
  each of the standard generators of Khovanov homology, cf.\
  \Definition{knot-diag-to-kh-chain-complex}. That is, an object of
  $\KhFlowCat(L)$ is a labeled resolution configuration of the form
  $\ob{x}=(\AssRes{L}{u},\gen{x})$ with $u\in\{0,1\}^n$. The grading
  on the objects is the homological grading $\homgr$ from
  \Definition{knot-diag-to-kh-chain-complex}; the quantum grading
  $\intgr$ is an additional grading on the objects. We need the
  orientation of $L$ in order to define these gradings, but the rest of
  the construction of $\KhFlowCat(L)$ is independent of the
  orientation.

  Consider objects $\ob{x}=(\AssRes{L}{u},\gen{x})$ and
  $\ob{y}=(\AssRes{L}{v},\gen{y})$ of $\KhFlowCat(L)$.  The space
  $\Moduli_{\KhFlowCat(L)}(\ob{x},\ob{y})$ is defined to be empty unless
  $\ob{y}\prec \ob{x}$ with respect to the partial order from
  \Definition{prec}. So, assume that $\ob{y}\prec\ob{x}$.
  Let $\gen{x}|$ denote the restriction of $\gen{x}$ to
  $s(\AssRes{L}{v}\sm \AssRes{L}{u})=\AssRes{L}{u}\setminus
  \AssRes{L}{v}$ and let $\gen{y}|$ denote the restriction of $\gen{y}$
  to $\AssRes{L}{v}\setminus \AssRes{L}{u}$.  Therefore,
  $(\AssRes{L}{v}\sm\AssRes{L}{u}, \gen{x}|,\gen{y}|)$ is a basic
  decorated resolution configuration. Define
  \[
  \Moduli_{\KhFlowCat(L)}(\ob{x},\ob{y})=\Moduli(\AssRes{L}{v}\setminus \AssRes{L}{u},\gen{x}|,\gen{y}|)
  \]
  as smooth manifolds with corners.  The composition maps for the
  resolution configuration moduli spaces (see
  \Condition{res-moduli-compose}, above) induce
  composition maps
  \[
  \Moduli_{\KhFlowCat(L)}(\ob{z},\ob{y})\times
  \Moduli_{\KhFlowCat(L)}(\ob{x},\ob{z})\to \Moduli_{\KhFlowCat(L)}(\ob{x},\ob{y}).
  \]

  The Khovanov flow category $\KhFlowCat(L)$ is equipped with a
  functor $\Funky$ to $\CubeFlowCat(n)[-n_-]$,\footnote{Henceforth, we
    will usually suppress the grading shifts from the notation.}  
  which is a cover in the sense of
  \Definition{flow-cat-cover}. On the objects, $\Funky\co
  \Ob_{\KhFlowCat(L)}\to \Ob_{\CubeFlowCat(n)}$ is defined as
  \[
  \Funky(\AssRes{L}{u},\gen{x})=u.
  \]
  For the morphisms,
  \[
  \Funky\from
  \Moduli_{\KhFlowCat(L)}((\AssRes{L}{u},x),(\AssRes{L}{v},y))\to
  \Moduli_{\CubeFlowCat(n)}(u,v)
  \]
  is defined to be composition 
  \[
  \xymatrix{
    \Moduli(\AssRes{L}{v}\setminus \AssRes{L}{u},\gen{x}|,\gen{y}|)\ar[r]^-\Forget&\Moduli_{\CubeFlowCat(|u|-|v|)}(\vect{1},\vect{0})\ar[r]^-{\CInc_{u,v}}&\Moduli_{\CubeFlowCat(n)}(u,v).}
  \]
\end{defn}

\begin{lemma}\label{lem:mod-q-gr}
  The space $\Moduli(\ob{x},\ob{y})$ is empty unless
  $\intgr(\ob{x})=\intgr(\ob{y})$ and $\homgr(\ob{y})<\homgr(\ob{x})$.
\end{lemma}
\begin{proof}
  The moduli space $\Moduli(\ob{x},\ob{y})$ is empty unless
  $\ob{y}\prec\ob{x}$; however, it is quite clear from \Definition{prec} and
  \Definition{knot-diag-to-kh-chain-complex} that the latter happens
  only if $\intgr(\ob{x})=\intgr(\ob{y})$ and
  $\homgr(\ob{y})<\homgr(\ob{x})$.
\end{proof}

\begin{defn}\label{def:Kh-space}
  Fix an oriented link diagram $L$ with $n$ crossings and an ordering
  of the crossings in $L$. The Khovanov flow category $\KhFlowCat(L)$
  (\Definition{khovanov-flow-category}) is a cover
  (\Definition{flow-cat-cover}) of the cube flow category
  $\CubeFlowCat(n)$ (\Definition{cube-flow-category}). Choose a sign
  assignment $s$ for $\CubeFlowCat(n)$ (\Definition{sign-assignment})
  and framing of $\CubeFlowCat(n)$ relative $s$
  (\Proposition{cube-can-be-framed}).  The cover $\Funky$ and the
  framing of $\CubeFlowCat(n)$ together produce a neat immersion of
  $\KhFlowCat(L)$ along with a coherent framing for the immersion
  (\Section{covers}).  Therefore, after choosing a perturbation
  (\Definition{coherent-framing-to-framed}), we get a framing of
  $\KhFlowCat(L)$. The \emph{Khovanov space} is the CW complex
  $\Realize{\KhFlowCat(L)}$ (\Definition{flow-gives-space}).

  The Khovanov homology (\Definition{knot-diag-to-kh-chain-complex})
  is the reduced cohomology of the Khovanov space shifted by $(-C)$
  for some positive integer $C$
  (\Lemma{realization-well-defined-refines-chain-complex}). The
  \emph{Khovanov spectrum} $\KhSpace(L)$ is the suspension spectrum of
  the Khovanov space, de-suspended $C$ times.

  By \Lemma{mod-q-gr} and \Lemma{wedge-cat-implies-wedge-space}, the
  Khovanov spectrum $\KhSpace(L)$ decomposes as a wedge sum over the
  quantum gradings,
  \[
  \KhSpace(L)=\bigvee_j \KhSpace^j(L).
  \]
\end{defn}

So, all that remains is to define the moduli spaces associated to
resolution configurations; we turn to this next.

\subsection{\texorpdfstring{$0$}{0}-dimensional moduli spaces}\label{sec:0D}
In this subsection we define $\Moduli(D,x,y)$ where $(D,x,y)$ is an index
$1$ basic decorated resolution configuration. Unsurprisingly, we
always define $\Moduli(D,x,y)$ to consist of a single point.
Since $\Moduli_{\CubeFlowCat(1)}(\vect{1},\vect{0})$ also consists of a single
point, there is an obvious map
\[
\Forget\co \Moduli(D,x,y) \to \Moduli_{\CubeFlowCat(1)}(\vect{1},\vect{0}).
\]

So far,
\Condition[s]{res-moduli-compose}--\ConditionCont{res-moduli-cover-trivial}
are vacuously satisfied.

\subsection{\texorpdfstring{$1$}{1}-dimensional moduli spaces}\label{sec:1D}
In this subsection we define $\Moduli(D,\gen{x},\gen{y})$ where
$(D,\gen{x},\gen{y})$ is an index $2$ basic decorated resolution
configuration. As we will see, this depends on a (global) choice.

\begin{defn}
  An index $2$ basic resolution configuration $D$ is said to be a
  \emph{ladybug configuration} if the following conditions are
  satisfied.
  \begin{itemize}
  \item $Z(D)$ consists of a single circle, which we will abbreviate as $Z$;
  \item The endpoints of the two arcs in $A(D)$, say $A_1$ and $A_2$,
    alternate around $Z$ (i.e., $\bdy A_1$ and $\bdy A_2$ are linked
    in $Z$).
  \end{itemize}
  See \Figure{ladybug-define}. (Ladybug configurations are the same as
  ``configurations of types $X$ and $Y$'' from
  \cite{OSzR-kh-oddkhovanov}.) 

  An index $2$ basic decorated resolution configuration $(D,x,y)$ is
  said to be a ladybug configuration if $D$ is a ladybug
  configuration. For grading reasons, in such configurations,
  $\gen{y}(Z)=x_+$, and the label in $\gen{x}$ of the (unique) circle
  in $Z(s(D))$ is $x_-$.
\end{defn}

Now, fix an index $2$ basic decorated resolution configuration
$(D,x,y)$.  Write $A(D)=\{A_1,A_2\}$, and assume that $A_1$ precedes
$A_2$ in the total order on $A(D)$.  Let
$(D_1,\gen{z}_1),\dots,\allowbreak(D_m,\gen{z}_m)$ be the labeled
resolution configurations such that $(D,\gen{y})\prec
(D_i,\gen{z_i})\allowbreak\prec\allowbreak (s(D),\gen{x})$. It is
clear that $\bdy_\expect\Moduli(D,\gen{x},\gen{y})$ consists of $m$
points.

\begin{lemma}\label{lem:ind-2-res-config}
  For any index $2$ basic decorated resolution configuration
  $(D,x,y)$, the number $m$ of labeled resolution configurations
  between $(D,\gen{y})$ and $(s(D),\gen{x})$ is either $2$ or $4$.
  Moreover, $m$ equals $4$ if and only if $(D,x,y)$ is a ladybug
  configuration. 
\end{lemma}
\begin{proof}
  If $D$ has a leaf or a co-leaf then
  \Lemma{technical-res-config-leaf}, in conjunction with
  \Lemma{res-config-dual}, implies that $P(D,x,y)$ is isomorphic to
  $\{0,1\}^2$, and hence $m=2$.

  However, it follows from a straightforward case analysis that the
  only index $2$ basic decorated resolution configuration without
  leaves or co-leaves is a ladybug configuration, in which case $m=4$;
  compare~\cite[Figure 2]{OSzR-kh-oddkhovanov}.
\end{proof}

\subsubsection{\texorpdfstring{$m=2$}{m=2}}The map 
\[
\Forget|_{\bdy}\co \bdy_{\expect}\Moduli(D,\gen{x},\gen{y})\to \bdy\Moduli_{\CubeFlowCat(2)}(\vect{1},\vect{0})
\]
is a trivial covering space, so by \Part{extend-exist} of
\Proposition{extend-moduli}, we can define
$\Moduli(D,\gen{x},\gen{y})$ and $\Forget\co
\Moduli(D,\gen{x},\gen{y})\to
\Moduli_{\CubeFlowCat(2)}(\vect{1},\vect{0})$ satisfying
\Condition[s]{res-moduli-compose}--\ConditionCont{res-moduli-cover-trivial}. More
concretely, 
$\Moduli(D,\gen{x},\gen{y})$ is defined to be a
single interval with boundary
$\bdy_{\expect}\Moduli(D,\gen{x},\gen{y})$. 
Assume without loss of generality that $D_i$ contains the arc $A_i$. 
Then the diffeomorphism $\Forget\co
\Moduli(D,\gen{x},\gen{y})\to
\Moduli_{\CubeFlowCat(2)}(\vect{1},\vect{0})$ is characterized (up to
isotopy) by the condition that
\begin{align*}
\Forget(\Moduli(D\sm D_1,\gen{z}_1,\gen{y})\times
\Moduli(D_1,\gen{x},\gen{z}_1))&=
\Moduli_{\CubeFlowCat(2)}((0,1),(0,0)) \circ
\Moduli_{\CubeFlowCat(2)}((1,1),(0,1))\\
\shortintertext{and}
\Forget(\Moduli(D\sm D_2,\gen{z}_2,\gen{y})\times
\Moduli(D_2,\gen{x},\gen{z}_2))&=
\Moduli_{\CubeFlowCat(2)}((1,0),(0,0)) \circ
\Moduli_{\CubeFlowCat(2)}((1,1),(1,0)).
\end{align*}
(This is equivalent to \Diagram{Forget-res}.) Note that, in
this case, $\Moduli(D,\gen{x},\gen{y})$ and the map $\Forget$ were
uniquely determined, even though \Part{extend-unique} of
\Proposition{extend-moduli} did not apply.

\subsubsection{\texorpdfstring{$m=4$}{m=4}} We need to decide which pairs of points in
$\bdy_{\expect}\Moduli(D,\gen{x},\gen{y})$ bound intervals, and for
this we need some more notation. 

Let $Z$ denote the unique circle in $Z(D)$.  The surgery $s_{A_1}(D)$
(\respectively $s_{A_2}(D)$) consists of two circles; denote these $Z_{1,1}$
and $Z_{1,2}$ (\respectively $Z_{2,1}$ and $Z_{2,2}$); i.e.,
$Z(s_{A_i}(D))=\{Z_{i,1},Z_{i,2}\}.$ Our main goal is to find a
bijection between $\{Z_{1,1},Z_{1,2}\}$ and $\{Z_{2,1},Z_{2,2}\}$;
this bijection will then tell us which points in
$\bdy_\expect\Moduli(\ob{x},\ob{y})$ to identify.

As an intermediate step, we distinguish two of the four arcs in
$Z\setminus(\bdy A_1\cup\bdy A_2)$.  Assume that the point $\infty\in
S^2$ is not in $D$, and view $D$ as lying in the plane
$S^2\setminus\{\infty\}\cong\RR^2$.  Then one of $A_1$ or $A_2$ lies
outside $Z$ (in the plane) while the other lies inside $Z$. Let $A_i$
be the inside arc and $A_o$ the outside arc. The circle $Z$ inherits
an orientation from the disk it bounds in $\RR^2$. With respect to
this orientation, each component of $Z\setminus(\bdy A_1\cup\bdy A_2)$
either runs from the outside arc $A_o$ to an inside arc $A_i$ or
vice-versa. The \emph{\OutIn{} pair}
is the pair of components of $Z\setminus(\bdy
A_1\cup\bdy A_2)$ which run from the outside arc $A_o$ to the inside
arc $A_i$. The other pair of components is the \emph{\InOut{}
  pair}. See \Figure{ladybug}.

\captionsetup[subfloat]{width={0.4\textwidth}}
\begin{figure}
  \centering 
  \subfloat[A ladybug configuration. The \OutIn{} pair is
  marked with small, numbered squares. (The numbering is
  arbitrary.)]{\label{fig:ladybug-define}\makebox[0.4\textwidth][c]{\includegraphics{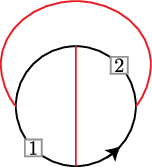}}} \hspace{0.1\textwidth}
  \subfloat[The pairs of circles $Z_{1,i}$ and $Z_{2,i}$; the
  bijection induced by the
  \OutIn{} pair is indicated.]{\includegraphics{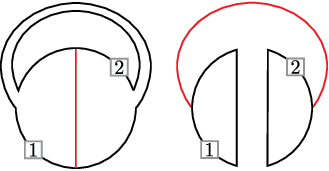}}\\
  \captionsetup[subfloat]{width={0.6\textwidth}}
  \subfloat[The image of $Z$ under three
  isotopies in $S^2$; note that the old \OutIn{} pair maps to the new
  \OutIn{} pair.]{\label{fig:ladybug-isotopy-in-s2}\includegraphics{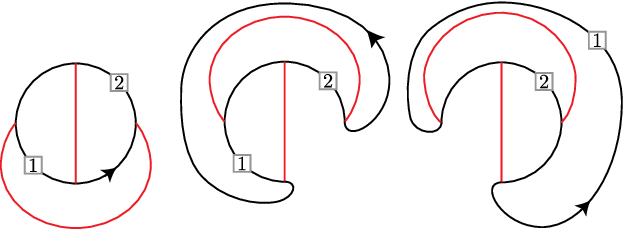}}
  \caption[A ladybug configuration.]{\textbf{A ladybug configuration.}}
  \label{fig:ladybug}
\end{figure}

\begin{lemma}\label{lem:in-out-preserved}
  Any isotopy of $Z\cup A_1\cup A_2$ in $S^2$ takes the \OutIn{} pair to
  the \OutIn{} pair.
\end{lemma}
\begin{proof}
  See \Figure{ladybug-isotopy-in-s2}. This also follows from the
  following local description of the \OutIn{} pair: The two arcs in the
  \OutIn{} pair are precisely the arcs that can be reached by traveling
  along $A_1$ or $A_2$ until $Z$, and then taking a right turn.
\end{proof}

The rest of the construction of the flow category is based on choosing
either the \OutIn{} pair or the \InOut{} pair.  Both are equally
natural choices and we will choose the \OutIn{} one.  To keep notation
simple, the moduli spaces that we will construct while working with
the \OutIn{} pair will be denoted
$\Moduli(D,\gen{x},\gen{y})$. Replacing the \OutIn{} pair by the
\InOut{} pair throughout gives a different set of moduli spaces, which
we denote $\Moduli_{\inout}(D,\gen{x},\gen{y})$.

So, let $\{P,Q\}$ be the \OutIn{} pair. After
renumbering, we can assume that 
\[
P\subset Z_{1,1},\quad Q\subset Z_{1,2},\quad P\subset Z_{2,1},\quad
Q\subset Z_{2,2}.
\]
Thus, $P,Q$ determine a bijection between $\{Z_{1,1},Z_{1,2}\}$ and
$\{Z_{2,1},Z_{2,2}\}$ by
\[
Z_{1,1}\longleftrightarrow Z_{2,1}\qquad\qquad Z_{1,2}\longleftrightarrow Z_{2,2}.
\]
We call this bijection the \emph{ladybug matching}.

The four points in $\Moduli(D,\gen{x},\gen{y})$ are given by the four
maximal chains in $P(D,x,y)$:
\begin{align*}
  a&=\bigl[(D,x_+)\prec (s_{A_1}(D),x_-x_+)\prec (s(D),x_-)\bigr]\\
  b&=\bigl[(D,x_+)\prec (s_{A_1}(D),x_+x_-)\prec (s(D),x_-)\bigr]\\
  c&=\bigl[(D,x_+)\prec (s_{A_2}(D),x_-x_+)\prec (s(D),x_-)\bigr]\\
  d&=\bigl[(D,x_+)\prec (s_{A_2}(D),x_+x_-)\prec (s(D),x_-)\bigr]
\end{align*}
Here, the expression $x_-x_+$ in $a$, say, means that $Z_{1,1}$ is
labeled by $x_-$ and $Z_{1,2}$ is labeled by $x_+$.  So, define
$\Moduli(D,\gen{x},\gen{y})$ to consist of two intervals, one with
boundary $ a\amalg c$ and the other with boundary $b\amalg d$. Note
that this depends on the bijection between $\{Z_{1,1},Z_{1,2}\}$ and
$\{Z_{2,1},Z_{2,2}\}$, but not the ordering between $Z_{1,1}$ and
$Z_{1,2}$.

Next we verify that we can define the map $\Forget\co
\Moduli(D,\gen{x},\gen{y})\to
\Moduli_{\CubeFlowCat(2)}(\vect{1},\vect{0})$ compatibly with our
choices, in an essentially unique way. The space
$\Moduli_{\CubeFlowCat(2)}(\vect{1},\vect{0})$ is a single interval.
\Diagram{Forget-res} implies that, under the map
$\Forget|_{\bdy}$, the points $a$ and $b$ in
$\bdy\Moduli(D,\gen{x},\gen{y})$ map to one endpoint of
$\Moduli_{\CubeFlowCat(2)}(\vect{1},\vect{0})$, while $c$ and $d$ map
to the other endpoint. Thus, $\Forget|_{\bdy}$ extends to a covering
map $\Forget\co \Moduli(D,\gen{x},\gen{y})\to
\Moduli_{\CubeFlowCat(2)}(\vect{1},\vect{0})$, uniquely up to isotopy.
(This is also why we could not simply match $a$ with $b$, and $c$ with
$d$.)

In summary:
\begin{prop}\label{prop:ind-2-exist}
  There are spaces $\Moduli(D,\gen{x},\gen{y})$ for $\ind(D)\leq 2$
  and maps $\Forget$ satisfying
  \Condition[s]{res-moduli-compose}--\ConditionCont{res-moduli-cover-trivial}.
\end{prop}
\begin{proof}
  This follows from \Proposition{extend-moduli} and the discussion
  in \Section{0D} and the rest of this subsection.
\end{proof}

\subsection{\texorpdfstring{$2$}{2}-dimensional moduli spaces}
\label{sec:ind-3-bdy}
In \Section{1D}, we constructed the spaces
$\Moduli(D,\gen{x},\gen{y})$ when $\ind(D)=2$ or, in other words,
$\Moduli(D,\gen{x},\gen{y})$ is $1$-dimensional.  The goal of this
subsection is to analyze
\[
\Forget|_\bdy\co \bdy_{\expect}\Moduli(D,\gen{x},\gen{y})\to \bdy\Moduli_{\CubeFlowCat(3)}(\vect{1},\vect{0})
\]
when $(D,x,y)$ is a basic index $3$ decorated resolution configuration. By
\Part{extend-cover} of \Proposition{extend-moduli}, we know
that $\Forget|_\bdy$ is a covering map. By
\Part{extend-exist} of \Proposition{extend-moduli}, in order
to construct $\Moduli(D,\gen{x},\gen{y})$, it suffices to show that
the map $\Forget|_{\bdy}$ is a trivial covering map on each
component. The spaces $\del_{\expect}\Moduli(D,x,y)$ and
$\bdy\Moduli_{\CubeFlowCat(3)}(\vect{1},\vect{0})$ are $2$-boundaries,
so we treat them as graphs.  From \Lemma{cube-cat-struct},
$\bdy\Moduli_{\CubeFlowCat(3)}(\vect{1},\vect{0})$ is a
$6$-cycle. Since the map $\Forget|_\bdy$ respects the graph structure,
it suffices to show that $\bdy_{\expect}\Moduli(D,\gen{x},\gen{y})$ is
a disjoint union of $6$-cycles.

The proof that $\bdy_{\expect}\Moduli(D,\gen{x},\gen{y})$ is a
disjoint union of $6$-cycles is a somewhat involved case analysis.  To
reduce the number of cases, we will use the notion of dual resolution
configurations (\Definition{dual-res-config}
and~\Definition{decorated-res-config-dual}) as well as the spaces
$\Moduli_{\inout}(D,\gen{x},\gen{y})$ constructed analogously to
$\Moduli(D,\gen{x},\gen{y})$ but using the \InOut{} pair in the ladybug
matching (see \Section{1D}).

\begin{lem}\label{lem:6-vertex-easy-case}
  The graph $\bdy_\expect\Moduli(D,x,y)$ is a disjoint union of
  cycles, and the number of vertices in each cycle is divisible by
  $6$. In particular, if $\del_\expect\Moduli(D,x,y)$ has exactly $6$
  vertices, then it is a $6$-cycle.
\end{lem}

\begin{proof}
  This is an immediate consequence of the fact that
  $\bdy_\expect\Moduli(D,x,y)$ is a cover of
  $\bdy\Moduli_{\CubeFlowCat(3)}(\vect{1},\vect{0})$, which is a
  $6$-cycle.
\end{proof}

\begin{lemma}\label{lem:ind3-dual-matching-dual-config}
  The graphs $\bdy_\expect\Moduli(D,\gen{x},\gen{y})$ and
  $\bdy_\expect\Moduli_{\inout}(D^*,\gen{y}^*,\gen{x}^*)$ are
  isomorphic.
\end{lemma}
\begin{proof}
  For any decorated resolution configuration $(E,u,v)$,
  \Lemma{res-config-dual} gives a natural isomorphism between
  $P(E,u,v)$ and the reverse of
  $P(E^*,v^*,u^*)$. Let $(f_{E}(F),w^*)\in
  P(E^*,v^*,u^*)$ be the labeled resolution configuration
  corresponding to the labeled resolution configuration $(F,w)\in
  P(E,u,v)$.

  It is easy to see that for any basic index $2$ decorated resolution
  configuration $(E,u,v)$, the vertices
  $\bigl[(E,v)\prec(F_1,w_1)\prec(s(E),u)\bigr]$ and
  $\bigl[(E,v)\prec(F_2,w_2)\prec(s(E),u)\bigr]$ are matched (i.e.\
  bound an interval) in $\Moduli(E,u,v)$ if and only if corresponding
  vertices
  $\bigl[(f_E(s(E)),u^*)\prec(f_E(F_1),w^*_1)\prec(f_E(E),v^*)\bigr]$ and
  $\bigl[(f_E(s(E)),u^*)\prec(f_E(F_2),w^*_2)\prec(f_E(E),v^*)\bigr]$ are
  matched in the dual $\Moduli_{\inout}(E^*,v^*,u^*)$.

  Now, in the graphs $\bdy_\expect\Moduli(D,\gen{x},\gen{y})$ and
  $\bdy_\expect\Moduli_{\inout}(D^*,\gen{y}^*,\gen{x}^*)$, the
  vertices correspond to the maximal chains of $P(D,x,y)$ and
  $P(D^*,y^*,x^*)$, respectively. Therefore, $f_D$ induces a bijection
  of the vertices.
  The edges correspond to the matchings of the maximal chains in index
  $2$ resolution configurations. By the above
  discussion $f_D$ induces a bijection of the edges as well.
\end{proof}

\begin{figure}
  \centering
  \subfloat{
    \xymatrix@C=0.1\textwidth{
      \vcenter{\hbox{\includegraphics[height=0.1\textheight]{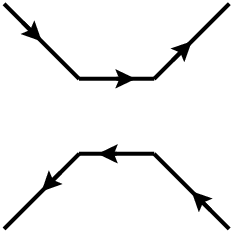}}}\ar[r]& \vcenter{\hbox{\includegraphics[height=0.1\textheight]{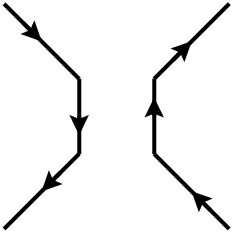}}}
    }
  }
  \caption[An oriented saddle move.]{\textbf{An oriented saddle move.}
  Changing the ladybug matching on exactly one face corresponds to a
  single oriented saddle move for the underlying graphs.}
  \label{fig:saddle-oriented-graphs}
\end{figure}

\begin{lemma}\label{lem:ind3-dual-matching}
  The graphs $\bdy_\expect\Moduli(D,\gen{x},\gen{y})$ and
  $\bdy_\expect\Moduli_{\inout}(D,\gen{x},\gen{y})$ have the same
  number of vertices and same parity of the number of cycles. In
  particular, if $\bdy_\expect\Moduli(D,x,y)$ has at most $12$
  vertices, then the two graphs are isomorphic.
\end{lemma}
\begin{proof}
  The vertices in either graph correspond to the maximal chains in
  $P(D,x,y)$. Therefore, the only thing that we have to prove is that
  the parity of the number of cycles in the two graphs is the same.

  By~\cite[Lemma 2.1]{OSzR-kh-oddkhovanov}, or a straightforward case
  analysis of configurations, it follows that an even number of faces
  in the cube of resolutions corresponding to $(D,x,y)$ come from
  ladybug configurations. (We in fact carry out this case analysis:
  the proofs of \Lemma{ind3-leaf} and \Lemma{ind3-hard-cases} imply
  that this even number is either $0$ or $2$.)

  Now orient the $6$-cycle
  $\bdy\Moduli_{\CubeFlowCat(3)}(\vect{1},\vect{0})$ arbitrarily. This
  induces orientations on the graphs
  $\bdy_\expect\Moduli(D,\gen{x},\gen{y})$ and
  $\bdy_\expect\Moduli_{\inout}(D,\gen{x},\gen{y})$, and hence we can
  think of them as oriented $1$-manifolds. If a face in the cube of
  resolutions of $(D,x,y)$ comes from a ladybug configuration, then
  changing the ladybug matching just for that face corresponds to a
  single oriented local saddle move for the underlying graphs; see
  \Figure{saddle-oriented-graphs}. Each such oriented saddle move
  changes the number of components of the graph by $1$. Since an even
  number of faces in the cube of resolutions of $(D,x,y)$ come from
  ladybug configurations, we are done.
\end{proof}

\begin{corollary}\label{cor:ind3-dual-config}
  The graph $\bdy_\expect\Moduli(D^*,\gen{y}^*,\gen{x}^*)$ is a
  $6$-cycle (\respectively a disjoint union of two $6$-cycles) if and only if 
  the graph $\bdy_\expect\Moduli(D,\gen{x},\gen{y})$ is a $6$-cycle (\respectively
  a disjoint union of two $6$-cycles).
\end{corollary}
\begin{proof}
  This is immediate from \Lemma{ind3-dual-matching-dual-config} and
  \Lemma{ind3-dual-matching}.
\end{proof}

\begin{lemma}\label{lem:ind3-leaf}
  Suppose that $D$ has a leaf. Then $\bdy_\expect\Moduli(D,x,y)$ is
  either a $6$-cycle or a disjoint union of two $6$-cycles.
\end{lemma}
\begin{proof}
  By \Lemma{technical-res-config-leaf}, there is a naturally
  associated index $2$ decorated resolution configuration $(D',x',y')$
  such that $P(D,x,y)$ is naturally isomorphic to
  $P(D',x',y')\times\{0,1\}$. If $(D',x',y')$ is not a ladybug
  configuration, then $P(D',x',y')$ is isomorphic to $\{0,1\}^2$; in
  that case, $P(D,x,y)$ has exactly $6$ maximal chains, and hence by
  \Lemma{6-vertex-easy-case}, $\bdy_\expect\Moduli(D,x,y)$ is a
  $6$-cycle.

  Now assume that $(D',x',y')$ is a ladybug
  configuration. An easy enumeration of the possible graphs $G(D)$
  (\Definition{graph-from-res-config}) tells us that, after
  reordering the arcs in $A(D)$ if necessary, $D$ is isotopic to one
  of the three resolution configurations
  \Object{fig:conf-a}--\Object{fig:conf-c} of
  \Figure{ind-3-res-config}. Furthermore, due to grading reasons, in
  configurations \Object{fig:conf-a} and \Object{fig:conf-b} (\respectively in
  configuration \Object{fig:conf-c}), the label in $y$ of the leaf
  (\respectively one of the leaves) is $x_+$.

  Assume the four maximal chains in $P(D',x',y')$ are
  $d_i=\bigl[b\prec c_i\prec a\bigr]$
  for $1\leq i\leq 4$.
  Assume further that the ladybug matching is $d_1\longleftrightarrow d_2$ and
  $d_3\longleftrightarrow d_4$.  The $12$ vertices in
  $P(D,x,y)=P(D',x',y')\times\{0,1\}$ can be represented by the
  following maximal chains:
  \begin{align*}
    u_i&=\bigl[(b,0)\prec(c_i,0)\prec(a,0)\prec(a,1)\bigr]\\    
    v_i&=\bigl[(b,0)\prec(c_i,0)\prec(c_i,1)\prec(a,1)\bigr]\\    
    w_i&=\bigl[(b,0)\prec(b,1)\prec(c_i,1)\prec(a,1)\bigr]
  \end{align*}
  for $1\leq i\leq 4$.  The following pairs are not parts of ladybug
  configurations, and so the following edges exist independently of
  our choice of ladybug matching:
  \begin{align*}
    v_1&\matching u_1 & v_1&\matching w_1 &
    v_2&\matching u_2 & v_2&\matching w_2 &
    v_3&\matching u_3 & v_3&\matching w_3 \\
    v_4&\matching u_4 & v_4&\matching w_4.     
    \shortintertext{Since $D$ contains a leaf where the label in $y$
      is $x_+$, it is clear that if the ladybug matching matches
      $d_i$ with $d_j$, then there are edges joining $u_i$ to $u_j$
      and $w_i$ to $w_j$. Therefore, the following edges come from the
      ladybug matching:}
    u_1&\matching u_2 & u_3 &\matching u_4 &
    w_1&\matching w_2 &  w_3&\matching w_4.
  \end{align*}
  So, the components of $\bdy_\expect\Moduli(D,\gen{x},\gen{y})$ are:
  \[
  v_1\matching u_1 \matching u_2 \matching
  v_2 \matching w_2 \matching w_1 \matching v_1
  \quad\text{ and }\quad 
  v_3\matching u_3 \matching u_4 \matching v_4
  \matching w_4 \matching w_3 \matching v_3.\qedhere
  \]
\end{proof}

\begin{corollary}\label{cor:ind3-coleaf}
  Suppose that $D$ has a co-leaf. Then
  $\bdy_\expect\Moduli(D,\gen{x},\gen{y})$ is a disjoint union of
  $6$-cycles.
\end{corollary}
\begin{proof}
  This is immediate from \Corollary{ind3-dual-config}
  and \Lemma{ind3-leaf}.
\end{proof}

\begin{lemma}\label{lem:no-leaves-or-coleaves}
  Up to isotopy in $S^2$ and reordering of the arcs, the four
  resolution configurations \Object{fig:conf-d}--\Object{fig:conf-g} of
  \Figure{ind-3-res-config} are the only ones with no leaves or
  co-leaves (\Definition{graph-from-res-config}).  Moreover,
  configurations \Object{fig:conf-d} and \Object{fig:conf-f} are dual, and
  configurations \Object{fig:conf-e} and \Object{fig:conf-g} are dual.
\end{lemma}
\begin{proof}
  This again follows from an easy enumeration of the possible graphs
  $G(D)$.  Compare~\cite[Figure
  4]{OSzR-kh-oddkhovanov}.
\end{proof}

\captionsetup[subfloat]{width=0.05\textwidth}
\begin{figure}
  \centering 
  \subfloat[]{\label{fig:conf-a}\includegraphics{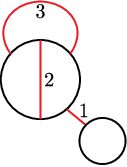}}\hspace{0.07\textwidth}
  \subfloat[]{\label{fig:conf-b}\includegraphics{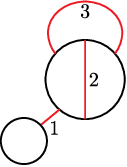}}\hspace{0.07\textwidth}
  \subfloat[]{\label{fig:conf-c}\includegraphics{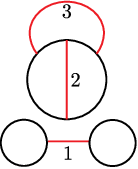}}\\
  \parbox{0.8\textwidth}{Up to isotopy in $S^2$ and reordering of the
    arcs, these are the only basic index $3$ resolution configurations
    with
    leaves and ladybugs.}\vspace{5pt}\\
  \subfloat[]{\label{fig:conf-d}\includegraphics{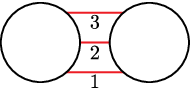}}\hspace{0.07\textwidth}
  \subfloat[]{\label{fig:conf-e}\includegraphics{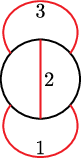}}\hspace{0.07\textwidth}
  \subfloat[]{\label{fig:conf-f}\includegraphics{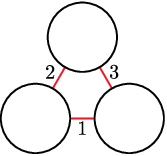}}\hspace{0.07\textwidth}
  \subfloat[]{\label{fig:conf-g}\includegraphics{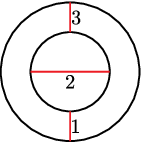}}\\
  \parbox{0.8\textwidth}{Up to isotopy in $S^2$ and reordering of the
    arcs, these are the only basic index $3$ resolution configurations
    without leaves or co-leaves. Configurations \Object{fig:conf-f} and
    \Object{fig:conf-g} are dual to configurations \Object{fig:conf-d} and
    \Object{fig:conf-e}, respectively.}
\caption[Some index $3$ resolution configurations.]{\textbf{Some index $3$ resolution configurations.}}
\label{fig:ind-3-res-config}    
\end{figure}

Next we check two cases directly:
\begin{lemma}\label{lem:ind3-hard-cases}
  If $D$ is of the type~\Object{fig:conf-d}, then
  $\bdy_\expect\Moduli(D,x,y)$ is a $6$-cycle; and if $D$ is of the
  type~\Object{fig:conf-e}, then $\bdy_\expect\Moduli(D,x,y)$ is a
  disjoint union of two $6$-cycles.
\end{lemma}
\begin{proof}
  For the configuration~\Object{fig:conf-d}, it is fairly easy to see
  that $P(D,x,y)$ has exactly $6$ maximal chains: In any maximal chain
  \[
  v=\bigl[(D,y)\prec(E_2,z_2)\prec(E_1,z_1)\prec(s(D),x)\bigr],
  \]
  $E_2$ is obtained from $D$ by a merge; therefore, the labeling $z_2$
  is determined entirely by $D$, $E_2$ and $y$ (cf.\
  \Definition{prec}), and hence, there are exactly three choices for
  $(E_2,z_2)$. However, it is easy to see that for each of those
  three choices, $E_2\sm s(D)$ is not a ladybug configurations, and hence
  there are exactly two maximal chains in $P(E_2\sm s(D),x,z_2)$.
  Therefore, $\bdy_\expect\Moduli(D,x,y)$ has exactly $6$
  vertices, and hence by \Lemma{6-vertex-easy-case}, we are done. 
  
  \begin{figure}
    \centering
    \includegraphics{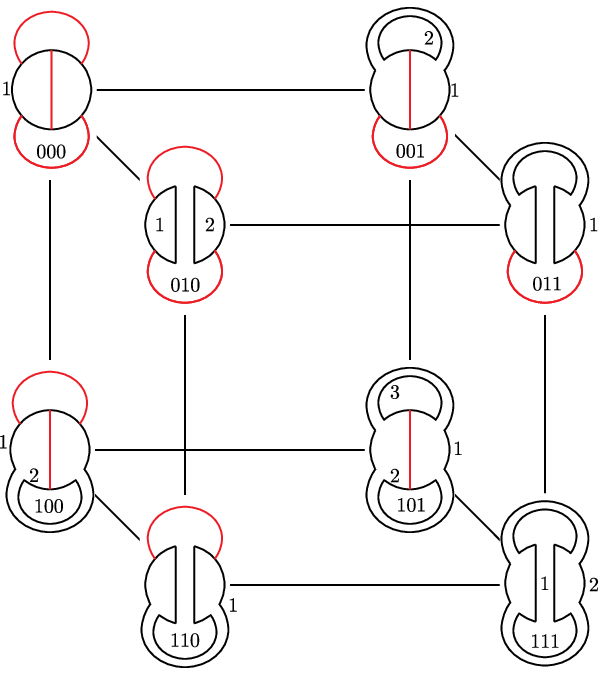}
    \caption{\textbf{The cube of resolutions for
        configuration \Object{fig:conf-e}.} The individual circles at
      each vertex are numbered.}
    \label{fig:cube-res-conf-a}
  \end{figure}

  For the configuration~\Object{fig:conf-e}, for grading reasons,
  $x=x_-x_-$ and $y=x_+$. The corresponding cube of resolutions is
  shown in \Figure{cube-res-conf-a}. The vertices of
  $\bdy_\expect\Moduli(D,\gen{x},\gen{y})$ are the maximal chains of
  $P(D,x,y)$:
  \begin{align*}
    v_1&=\bigl[ (D,x_+)\prec (s_{\{A_1\}}(D),x_+x_-) \prec (s_{\{A_1,A_2\}}(D), x_-)\prec (s(D),x_-x_-)  \bigr]\\
    v_2&=\bigl[ (D,x_+)\prec (s_{\{A_1\}}(D),x_+x_-) \prec (s_{\{A_1,A_3\}}(D),x_-x_-x_+)\prec (s(D),x_-x_-) \bigr]\\
    v_3&=\bigl[ (D,x_+)\prec (s_{\{A_1\}}(D),x_-x_+) \prec (s_{\{A_1,A_2\}}(D), x_-)\prec (s(D),x_-x_-)  \bigr]\displaybreak[0]\\
    v_4&=\bigl[ (D,x_+)\prec (s_{\{A_1\}}(D),x_-x_+) \prec (s_{\{A_1,A_3\}}(D), x_-x_+x_-)\prec (s(D),x_-x_-)  \bigr]\displaybreak[0]\\
    v_5&=\bigl[(D,x_+)\prec (s_{\{A_2\}}(D),x_+x_-)\prec (s_{\{A_1,A_2\}}(D), x_-)\prec (s(D),x_-x_-)  \bigr]\displaybreak[0]\\
    v_6&=\bigl[(D,x_+)\prec (s_{\{A_2\}}(D),x_+x_-)\prec (s_{\{A_2,A_3\}}(D), x_-)\prec (s(D),x_-x_-)   \bigr]\displaybreak[0]\\
    v_7&=\bigl[(D,x_+)\prec (s_{\{A_2\}}(D),x_-x_+)\prec (s_{\{A_1,A_2\}}(D), x_-)\prec (s(D),x_-x_-)  \bigr]\displaybreak[0]\\
    v_8&=\bigl[(D,x_+)\prec (s_{\{A_2\}}(D),x_-x_+)\prec (s_{\{A_2,A_3\}}(D), x_-)\prec  (s(D),x_-x_-)  \bigr]\displaybreak[0]\\
    v_9&=\bigl[ (D,x_+)\prec (s_{\{A_3\}}(D),x_+x_-) \prec (s_{\{A_2,A_3\}}(D), x_-)\prec (s(D),x_-x_-)  \bigr]\displaybreak[0]\\
    v_{10}&=\bigl[ (D,x_+)\prec (s_{\{A_3\}}(D),x_+x_-) \prec (s_{\{A_1,A_3\}}(D),x_-x_+x_-)\prec (s(D),x_-x_-) \bigr]\\
    v_{11}&=\bigl[ (D,x_+)\prec (s_{\{A_3\}}(D),x_-x_+) \prec (s_{\{A_2,A_3\}}(D), x_-)\prec (s(D),x_-x_-)  \bigr]\\
    v_{12}&=\bigl[ (D,x_+)\prec (s_{\{A_3\}}(D),x_-x_+) \prec (s_{\{A_1,A_3\}}(D), x_-x_-x_+)\prec (s(D),x_-x_-)  \bigr].
  \end{align*}
  The faces of the cube with vertices $(000,100,010,110)$ and
  $(000,001,010,011)$ come from ladybug configurations. The ladybug
  matchings for these arrangements are:
  \begin{center}
    \begin{tabular}{ccc}
      \toprule
      Face &\qquad  & Circle matching\\
      \midrule
      $(000,100,010,110)$ & & $1_{100}\longleftrightarrow 2_{010}$,
      $2_{100}\longleftrightarrow 1_{010}$\\ 
      $(000,010,001,011)$ & &$1_{010}\longleftrightarrow 1_{001}$, 
      $2_{010}\longleftrightarrow 2_{001}$\\
      \bottomrule
    \end{tabular}
  \end{center}
  The
  following pairs are not parts of ladybug configurations, and so
  the following edges exist independently of our choice of ladybug matching:
  \begin{align*}
    v_1&\matching v_2 & v_3&\matching v_4 &
    v_5&\matching v_6 & v_7&\matching v_8 &
    v_9&\matching v_{10} & v_{11}&\matching v_{12}\\
    v_2&\matching v_{12} & v_4&\matching v_{10}.
    \shortintertext{The following edges do depend on our choice of ladybug matching:}
    v_1&\matching v_7 & v_3 &\matching v_5 &
    v_6&\matching v_9 &  v_8&\matching v_{11}.
  \end{align*}
  So, the components of $\bdy_\expect\Moduli(D,\gen{x},\gen{y})$ are:
  \[
  v_1\matching v_2 \matching v_{12} \matching
  v_{11} \matching v_8 \matching v_7 \matching v_1
  \quad\text{ and }\quad 
  v_3\matching v_4 \matching v_{10} \matching v_9
  \matching v_6 \matching v_5 \matching v_3.\qedhere
  \]
\end{proof}

\begin{corollary}\label{cor:ind3-dual-hard-cases}
  In cases \Object{fig:conf-f} and \Object{fig:conf-g} of
  \Figure{ind-3-res-config}, the graph
  $\bdy_\expect\Moduli(D,\gen{x},\gen{y})$ is a disjoint union of
  $6$-cycles.
\end{corollary}
\begin{proof}
  This is immediate from \Lemma{ind3-hard-cases}
  and \Corollary{ind3-dual-config}.
\end{proof}

In summary:
\begin{prop}\label{prop:ind-3-exist}
  There are spaces $\Moduli(D,\gen{x},\gen{y})$ for $\ind(D)\leq 3$
  and maps $\Forget$ satisfying
  \Condition[s]{res-moduli-compose}--\ConditionCont{res-moduli-cover-trivial}.
\end{prop}
\begin{proof}
  This follows from
  \Proposition{extend-moduli}, \Proposition{ind-2-exist} and
  the discussion in the rest of this
  subsection. \Proposition{ind-2-exist} takes care of the case that
  $\ind(D)\leq 2$. For the case $\ind(D)=3$, by
  \Proposition{extend-moduli} it suffices to verify that the map
  $\bdy_\expect \Moduli(D,x,y)\to
  \bdy\Moduli_{\CubeFlowCat(3)}(\vect{1},\vect{0})$ is a trivial
  covering map. As discussed in the beginning of the section, this is
  equivalent to verifying that each component of $\bdy_\expect
  \Moduli(D,x,y)$ is a $6$-cycle. \Lemma{ind3-leaf} and
  \Corollary{ind3-coleaf} take care of the cases where $D$ has a leaf
  or co-leaf. By \Lemma{no-leaves-or-coleaves}, there are four
  remaining cases. \Lemma{ind3-hard-cases}
  takes care of two of these four cases and
  \Corollary{ind3-dual-hard-cases} polishes off the last two cases.
\end{proof}

\subsection{\texorpdfstring{$n$}{n}-dimensional moduli spaces, \texorpdfstring{$n\geq 3$}{n>=3}}\label{sec:nd-moduli-spaces}
\begin{prop}
  There are spaces $\Moduli(D,\gen{x},\gen{y})$ and maps $\Forget$
  satisfying
  \Condition[s]{res-moduli-compose}--\ConditionCont{res-moduli-cover-trivial}.
\end{prop}
\begin{proof}
  \Proposition{ind-3-exist} takes care of the case that $\ind(D)\leq
  3$. The case $\ind(D)\geq 4$ follows from
  \Part{extend-exist} of \Proposition{extend-moduli}.
\end{proof}

\section{Invariance of the Khovanov spectrum}
\label{sec:invariance}
We devote this section to proving our main theorem,
\Theorem{kh-space}. The main work is in proving \Part{thm-invariant}. In the
construction of $\KhSpace^j(L)$ we made the following choices:
\begin{enumerate}
\item A choice of ladybug matching: left or right.
\item An oriented link diagram $L$ with $n$ crossings.
\item\label{item:choice:crossings} An ordering of the crossings of $L$.
\item A sign assignment $s$ of the cube $\Cube(n)$.
\item A neat embedding $\iota$ and a framing $\Frame$ for the cube
  flow category $\CubeFlowCat(n)$ relative $s$.
\item A framed neat embedding of the Khovanov flow category $\KhFlowCat(L)$
  relative to some $\TupV{d}$. This framed neat embedding is a perturbation of $(\iota,\Frame)$.
\item\label{item:choice:ABepR} Integers $A,B$ and real
  numbers $\ep,R$ used in the construction of the CW complex (\Definition{flow-gives-space}).
\end{enumerate}

We will prove independence of these choices in reverse order.

\begin{prop}\label{prop:choice-independent}
  Modulo the choice of a ladybug matching, the stable homotopy type
  of $\KhSpace^j(L)$ is an invariant of the link diagram.
\end{prop}
\begin{proof}
  We must verify independence of
  choices~\Item{choice:crossings}--\Item{choice:ABepR}. Throughout the
  proof we will suppress the decomposition along quantum gradings, but
  it will be clear that this decomposition is preserved.

  Independence of $\epsilon$ and $R$ is \Lemma{CW-indep-ep-R-framing},
  and independence of $A$, $B$ and $\TupV{d}$ is
  \Lemma{CW-indep-A-B}. (Note that this does introduce some
  suspensions and canceling de-suspensions.) Next, for a given neat
  embedding $\iota$ and
  framing $\Frame$ for the cube flow category $\CubeFlowCat(n)$, if we
  consider two framed neat embeddings $(\iota_0,\Frame_0)$ and
  $(\iota_1,\Frame_1)$ of $\KhFlowCat(L)$ which are perturbations of
  $(\iota,\Frame)$, by \Lemma{perturb-connect-framings}, after
  increasing $\TupV{d}$ if necessary we can find a one-parameter
  family of neat embeddings and framings $(\iota_t,\Frame_t)$
  connecting $(\iota_0,\Frame_0)$ and $(\iota_1,\Frame_1)$. Then, it
  again follows from \Lemma{CW-indep-ep-R-framing} that the Khovanov
  spaces $\Realize[{\iota_0,\Frame_0}]{\KhFlowCat}$ and
  $\Realize[{\iota_1,\Frame_1}]{\KhFlowCat}$ are homotopy equivalent.
  By \Lemma{change-framing-cube}, after increasing $\TupV{d}$ if
  necessary, any two choices of framed neat embeddings of the cube
  flow category $\CubeFlowCat(n)$ relative to $s$ can be connected by
  a $1$-parameter family of framed neat embeddings, and so lead to homotopy
  equivalent space $\Realize{\KhFlowCat}$, as before.

  Next we turn to independence of the sign assignment. Given two
  different sign assignments $s_0$ and $s_1$ for $\Cube(n)$ there is a
  sign assignment $s$ for $\Cube(n+1)$ whose restriction to the
  codimension-$1$ face $\{i\}\times\Cube(n)$ is $s_i$ for
  $i\in\{0,1\}$. Consider the link diagram $L'$ obtained by taking the
  disjoint union of $L$ with a $1$-crossing unknot $U$, drawn so that
  the $0$-resolution of $U$ consists of two circles (and hence the
  $1$-resolution of $U$ consists of a single circle). Order the
  crossings of $L'$ so that the crossing in $U$ is the first crossing
  and the other crossings are ordered as in $L$.  Frame $\Cube(n+1)$
  relative to the sign assignment $s$, and choose a framed neat embedding of
  $\KhFlowCat(L')$ which is a perturbation.

  Consider the full subcategory $\Cat$ of $\KhFlowCat(L')$ generated
  by objects where the (one or two) circles coming from $U$ are
  labeled $x_+$. In $\KhFlowCat(L')$, there are no arrows into or out
  of $\Cat$, so $\Realize{\Cat}$ is a summand of
  $\Realize{\KhFlowCat(L')}$. Moreover, $\Realize{\KhFlowCat(L,s_0)}$
  is a subcomplex of $\Realize{\Cat}$, and the suspension
  $\Sigma\Realize{\KhFlowCat(L,s_1)}$ is the quotient
  $\Realize{\Cat}/\Realize{\KhFlowCat(L,s_0)}$. The reduced cohomology
  of $\Realize{\Cat}$ is clearly trivial (compare \Proposition{RI}). Therefore, using \Lemma{exact-sequence-space},
  $\Sigma\Realize{\KhFlowCat(L,s_0)}\simeq
  \Sigma\Realize{\KhFlowCat(L,s_1)}$, as desired.

  Finally, for independence of the ordering of crossings, observe that
  changing the ordering of the crossings has the effect of composing
  the functor $\Forget\co \KhFlowCat(L)\to \CubeFlowCat(n)$ with an
  automorphism $\phi$ of $\CubeFlowCat(n)$. This is equivalent to
  replacing the sign assignment $s$ and framed neat embedding $\iota$
  of $\CubeFlowCat(n)$ with $\phi^*(s)=s\circ\phi$ and
  $\phi^*(\iota)=\iota\circ \phi$. So, independence of the ordering of
  the crossings follows from independence of the sign assignment and
  neat embedding.
\end{proof}

\captionsetup[subfloat]{width=0.2\textwidth}
\begin{figure}
  \centering
  \subfloat[RI.]{
    \xymatrix@C=0.1\textwidth{
      \vcenter{\hbox{\includegraphics[height=0.3\textheight]{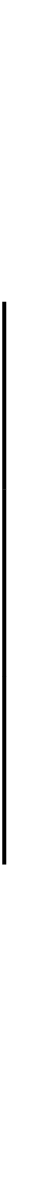}}}\ar@{->}[r]& \vcenter{\hbox{\includegraphics[height=0.3\textheight]{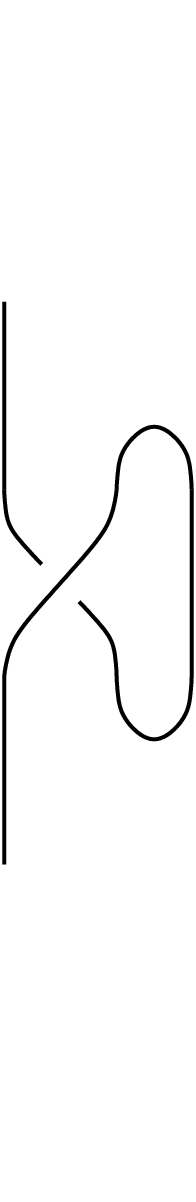}}}
    }
  }
  \hspace{0.03\textwidth}
  \subfloat[RII.]{
    \xymatrix@C=0.1\textwidth{
      \vcenter{\hbox{\includegraphics[height=0.3\textheight]{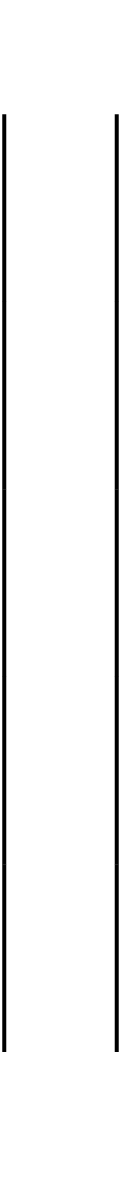}}}\ar@{->}[r]& \vcenter{\hbox{\includegraphics[height=0.3\textheight]{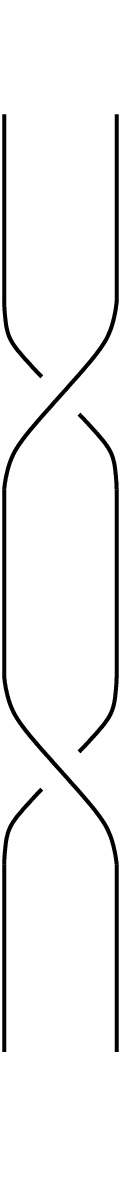}}}
    }
  }
  \hspace{0.03\textwidth}
  \subfloat[RIII.]{
    \xymatrix@C=0.1\textwidth{
      \vcenter{\hbox{\includegraphics[height=0.3\textheight]{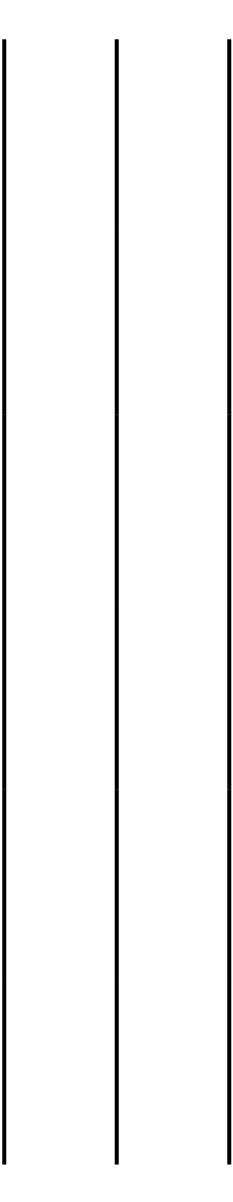}}}\ar@{->}[r]& \vcenter{\hbox{\includegraphics[height=0.3\textheight]{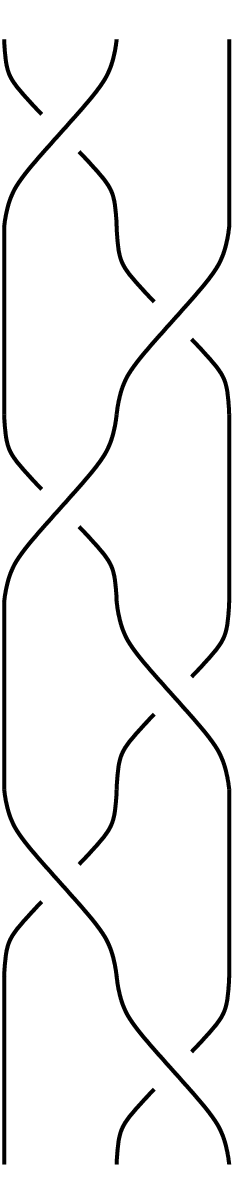}}}
    }
  }
  \caption{\textbf{The three Reidemeister moves.} The orientations of
    the strands are arbitrary.}\label{fig:Reid-moves}
\end{figure}

We now consider the three Reidemeister moves RI, RII and RIII of
\Figure{Reid-moves}. For unoriented link diagrams, it is enough to
consider just these three moves (and their inverses): The other type
of Reidemeister I move can be generated by RI and RII; and the usual
Reidemeister III move can be generated by this braid-like RIII and RII
(see \cite[Section 7.3]{Baldwin-hf-s-seq}). 

For each of RI, RII and RIII, the strands are oriented in the same
way before and after the move, but are otherwise oriented arbitrarily. The orientations of the
strands are only needed in order to count the number of positive and
negative crossings, which in turn are only needed in the definition of
the gradings. For all consistent orientations, in RI, RII and RIII,
$(n_+,n_-)$ increases by $(1,0)$, $(1,1)$ and $(3,3)$, respectively.

\begin{figure}
  \centering
  \captionsetup[subfloat]{width=0.4\textwidth}
  \subfloat[The subcomplex $C_1$.]{\label{fig:RI-sub}
    \makebox[0.4\textwidth][c]{
      \xymatrix@C=0.1\textwidth{
        \vcenter{\hbox{\includegraphics[height=0.15\textheight]{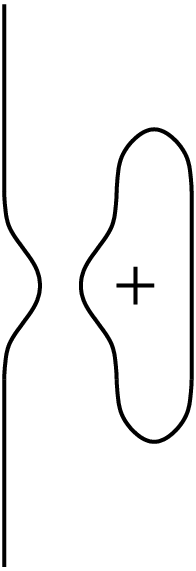}}}\ar@{->}[r]&
        \vcenter{\hbox{\includegraphics[height=0.15\textheight]{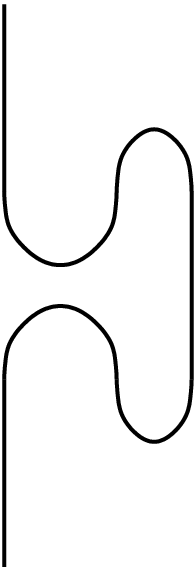}}}
      }
    }
  } 
  \hspace{0.03\textwidth}
  \subfloat[The quotient complex $C_2$.]{\label{fig:RI-quotient}
    \makebox[0.4\textwidth][c]{
      \xymatrix@C=0.1\textwidth{
        \vcenter{\hbox{\includegraphics[height=0.15\textheight]{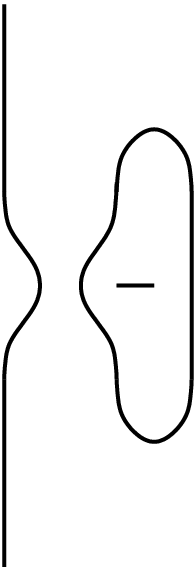}}}
      }
    }
  }
  \caption{\textbf{Invariance under the Reidemeister moves RI.} The
    relevant subcomplexes and quotient complexes are indicated.}
  \label{fig:Reidemeister-I}
\end{figure}

\begin{prop}\label{prop:RI}
  Let $L$ be a link diagram and suppose that $L'$ is obtained from $L$ by
  performing the Reidemeister move RI. Then $\KhSpace^j(L)$ is
  stably homotopy equivalent to $\KhSpace^j(L')$.
\end{prop}
\begin{proof}
  We will frame the two flow categories appropriately so that
  $\Realize{\KhFlowCat(L)}$ will be homotopy equivalent to
  $\Realize{\KhFlowCat(L')}$. It will be clear that the bi-grading
  shifts will work out so that the homotopy equivalence will induce a
  homotopy equivalence between $\KhSpace(L)$ and $\KhSpace(L')$, and
  the latter will respect the decomposition along the quantum
  gradings.
  
  Following~\cite[Section 3.5.1]{Bar-kh-khovanovs}, there is a
  contractible subcomplex $C_1$ of $\KhCx(L')$ so that the
  corresponding quotient complex $C_2=\KhCx(L')/C_1$ is isomorphic to
  $\KhCx(L)$; to develop some notation, we spell this out. The
  subcomplex $C_1$ and the quotient complex $C_2$ are shown in
  \Figure{RI-sub} and \Figure{RI-quotient}, respectively. That is, the
  $0$-resolution of the new crossing has a small circle; denote it
  $U$. Then $C_1$ consists of all $0$-resolutions at the new crossing
  where $U$ is labeled by $x_+$, as well as all $1$-resolutions at the
  new crossing (with any labeling of the components); while
  $C_2$ is generated by all $0$-resolutions at the new
  crossing where $U$ is labeled by $x_-$.

  It is immediate from the definitions that $C_1$ corresponds to a
  subcomplex of $\KhCx(L')$, and that $C_2$ is isomorphic to
  $\KhCx(L)$. In fact, $C_1$ corresponds to an upward closed
  subcategory $\Cat_1$ of $\KhFlowCat(L')$. The complementary downward
  closed subcategory $\Cat_2$ (corresponding to $C_2$) is obviously
  isomorphic to $\KhFlowCat(L)$. Moreover, the chain complex $C_1$ is
  acyclic (the horizontal arrow in \Figure{RI-sub} is an
  isomorphism). So, by \Lemma{exact-sequence-space}, for compatible
  choices of framings, the inclusion map $\Realize{\KhFlowCat(L)}\into
  \Realize{\KhFlowCat(L')}$ is a homotopy equivalence.
\end{proof}

\begin{figure}
  \centering
  \captionsetup[subfloat]{width=0.2\textwidth}
  \subfloat[RII: $C_1$.]{
    \makebox[0.2\textwidth][c]{
      \xymatrix{
        \vcenter{\hbox{\includegraphics[height=0.15\textheight]{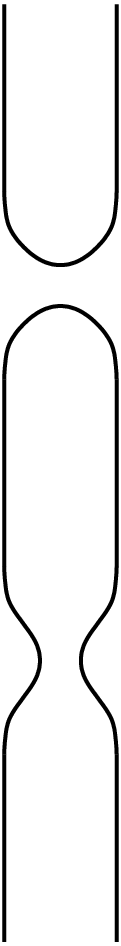}}}\\
        \vcenter{\hbox{\includegraphics[height=0.15\textheight]{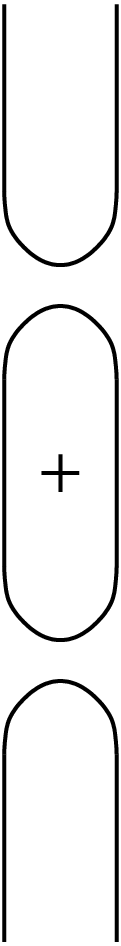}}}\ar[u]
      }
    }
  }
  \hspace{0.02\textwidth}
  \subfloat[RII: $C_2$.]{
    \makebox[0.2\textwidth][c]{
      \xymatrix{
        \vcenter{\hbox{\includegraphics[height=0.15\textheight]{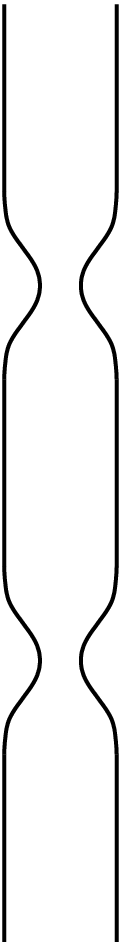}}}
        &  \vcenter{\hbox{\includegraphics[height=0.15\textheight]{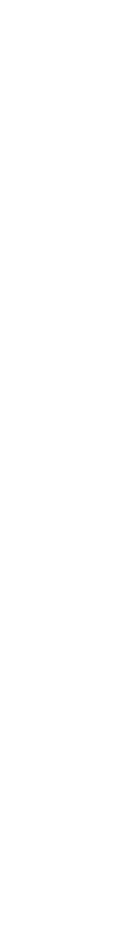}}}\\
        \vcenter{\hbox{\includegraphics[height=0.15\textheight]{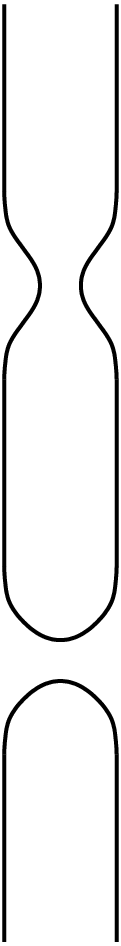}}}\ar[r]\ar[u]
        & \vcenter{\hbox{\includegraphics[height=0.15\textheight]{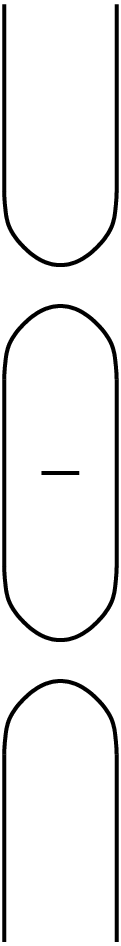}}}
      }
    }
  }
  \hspace{0.02\textwidth}
  \subfloat[RII: $C_3$.]{
    \makebox[0.2\textwidth][c]{
      \xymatrix{
        \vcenter{\hbox{\includegraphics[height=0.15\textheight]{RII-placeholder}}}
        &  \vcenter{\hbox{\includegraphics[height=0.15\textheight]{RII-placeholder}}}\\
        \vcenter{\hbox{\includegraphics[height=0.15\textheight]{RII-C00}}}\ar[r]
        & \vcenter{\hbox{\includegraphics[height=0.15\textheight]{RII-C01m}}}
      }
    }
  }
  \hspace{0.02\textwidth}
  \subfloat[RII: $C_4$.]{
    \makebox[0.2\textwidth][c]{
      \xymatrix{
        \vcenter{\hbox{\includegraphics[height=0.15\textheight]{RII-C10}}}\\
        \vcenter{\hbox{\includegraphics[height=0.15\textheight]{RII-placeholder}}}
      }
    }
  }
  \caption{\textbf{Invariance under the Reidemeister moves RII.} The
    relevant subcomplexes and quotient complexes are indicated.}
  \label{fig:Reidemeister-II}
\end{figure}

\begin{prop}\label{prop:RII}
  Let $L$ be a link diagram and suppose $L'$ is obtained from $L$ by
  performing the Reidemeister move RII. Then $\KhSpace^j(L)$ is
  stably homotopy equivalent to $\KhSpace^j(L')$.
\end{prop}
\begin{proof}
  As in the proof of RI invariance, we will suppress the bi-grading
  information. We will once again collapse acyclic subcomplexes and
  quotient complexes of $\KhCx(L')$ until we are left with
  $\KhCx(L)$. One sequence of such subcomplexes and quotient complexes
  is given in~\cite[Section 3.5.2]{Bar-kh-khovanovs}, which we repeat
  in \Figure{Reidemeister-II} for the reader's
  convenience. Specifically, the $(01)$-resolution (of the two new
  crossings) has a closed circle $U$. The $(01)$-resolutions in which
  $U$ is labeled by $x_+$ together with all $(11)$-resolutions
  generate an acyclic subcomplex $C_1$ of $\KhCx(L')$, and a
  corresponding upward closed subcategory $\Cat_1$ of
  $\KhFlowCat(L')$.  The objects not in $C_1$ (\respectively $\Cat_1$) form a
  quotient complex $C_2$ (\respectively downward-closed subcategory
  $\Cat_2$), whose objects correspond to the $(01)$-resolutions in which $U$ is
  labeled by $x_-$, the $(00)$-resolutions, and the
  $(10)$-resolutions. The complex $C_2$ has an acyclic quotient complex
  $C_3$ (corresponding to a further downward closed subcategory
  $\Cat_3$) consisting of the $(00)$-resolutions and the remaining
  $(01)$-resolutions. Moreover, the complement, $C_4$, corresponds
  exactly to the $(10)$-resolutions. The corresponding subcategory
  $\Cat_4$ is obviously isomorphic to $\KhFlowCat(L)$. So, two
  applications of \Lemma{exact-sequence-space} give the result.
\end{proof}

There are two general principles at work in the proof of
\Proposition{RII}, which we abstract for use in the RIII case. Let $v$
and $u$ be vertices in a (partial) cube of resolutions, such that
there is an arrow from $v$ to $u$, and one of the following holds.
\begin{itemize}
\item The arrow from $v$ to $u$ merges a circle $U$ of the (partial)
  resolution at $v$. Let $S$ be the set of all generators that
  correspond to $u$; and let $T$ be the set of all generators
  corresponding to $v$ with the circle $U$ labeled by $x_+$.
\item The arrow from $v$ to $u$ splits off a circle $U$ in the (partial)
  resolution at $u$. Let $S$ be the set of all generators that
  correspond to $u$ with the circle $U$ labeled by $x_-$; and let $T$
  be the set of all generators that correspond to $v$. 
\end{itemize}
Let $C$ be the chain complex generated by $S$ and $T$; it is an
acyclic complex, and therefore, we can delete it without changing the
homology. If in addition, $C$ is a subcomplex or a quotient complex of
the original chain complex, then in deleting $C$, we do not introduce
any new boundary maps.

\begin{figure}
  \centering
  \input{draws/RIIIcube}
  \caption{\textbf{Cube of resolutions for the braid-like RIII.}}
  \label{fig:RIII-cube}
\end{figure}

\begin{prop}\label{prop:RIII}
  Let $L$ be a link diagram and suppose $L'$ is obtained from $L$ by
  performing the braid-like Reidemeister move RIII. Then
  $\KhSpace^j(L)$ is stably homotopy equivalent to $\KhSpace^j(L')$.
\end{prop}
\begin{proof}
  Once again, we suppress the bi-grading information.  The relevant
  part of the cube of resolutions for $L'$ is shown in
  \Figure{RIII-cube}. The top half of the figure indicates a
  subcomplex of $\KhCx(L')$ (corresponding to an upward
  closed subcategory of $\KhFlowCat(L')$), the bottom half a
  quotient complex of $\KhCx(L')$ (corresponding to a
  downward closed subcategory of $\KhFlowCat(L')$); the homology is
  carried by the vertex $(000111)$ (corresponding to a subcategory
  isomorphic to $\KhFlowCat(L)$) on the middle-right of the
  figure. Note that the division essentially corresponds to the
  homological grading, with some of the vertices in the middle of the
  homological grading occurring in both the bottom and top halves
  (with some circle labeled differently). We will show that the top
  and bottom halves are acyclic. Assuming this, as in the proof of
  \Proposition{RII}, two applications of \Lemma{exact-sequence-space}
  show that $\Realize{\KhFlowCat(L')}$ is homotopy equivalent to
  $\Realize{\KhFlowCat(L)}$.

  In order to complete the proof, it only remains to see that the two
  halves are indeed acyclic chain complexes. We demonstrate this by
  doing the following two sequences of cancellations (the
  cancellations are similar to the ones in~\cite[Section
  7.3]{Baldwin-hf-s-seq}):
  \begin{center}
    \begin{tabular}{lp{0.7\textwidth}}
      \toprule
      Top half & 
      1*1111, 1*1110, 1*1101, 1*1011,
      110*11, 0111*1, *01111,
      1*1100, 1*1010, 1*1001,
      0111*0, 110*01, *01110, 01101*,
      *01011, 110*10, 010*11, 1*0011,
      01*101, *01101, 1001*1, 10*101.\\
      \midrule
      Bottom half &
      0000*0, 1000*0, 0100*0, 0010*0,
      00*100, 0*0001, 00001*,
      0110*0, 1100*0, 1*1000, 1010*0,
      01*100, 1*0001, *01001, 00*101,
      01*001, 01001*, 010*10, 0*0110, *01100,
      *01010, 100*10, 1*0100, 10*100.\\
      \bottomrule
    \end{tabular}
  \end{center}
  Here, for instance, 1*1111 means to cancel along the edge from the
  101111 resolution to the 111111 resolution (using the principles
  enunciated above). It is straightforward to verify that each of
  these cancellations corresponds to taking either a subcomplex or a
  quotient complex of the image under the previous cancellations;
  hence, during the cancellations we do not introduce additional
  maps. 
\end{proof}

Finally, we show that the space $\KhSpace^j(L)$ is independent of the
choice of ladybug matching. This argument is inspired by ideas of
O.~Viro.
\begin{proposition}\label{prop:ladybug-invariance}
  Up to stable homotopy equivalence, $\KhSpace^j(L)$ is independent of
  the ladybug matching.
\end{proposition}
\begin{proof}
  Fix a link diagram $L$, and let $L'$ be the result of reflecting $L$
  across the $y$-axis, say, and reversing all of the crossings. (In
  other words, the corresponding links differ by rotation by $\pi$
  around the $y$-axis.) The Khovanov homologies of $L$ and $L'$ are
  isomorphic for two different reasons:
  \begin{enumerate}
  \item The diagrams $L$ and $L'$ represent the same 
    link, and so are related by a sequence of Reidemeister moves.
  \item There is a bijection between states for $L$ and $L'$, obtained
    by reflecting across the $y$-axis.
  \end{enumerate}
  
  This observation allows us to prove independence of the ladybug
  matching, as follows. Let $\KhFlowCat$ denote the standard Khovanov
  flow category, constructed using the usual, right pair in the
  ladybug matching, and let $\KhFlowCat^*$ denote the Khovanov flow
  category constructed using the left pair in the ladybug
  matching. 

  We claim that there is an isomorphism of framed flow categories
  $\KhFlowCat(L)\cong \KhFlowCat^*(L')$. Indeed, reflecting across the
  $y$-axis gives a bijection between objects of $\KhFlowCat(L)$ and
  objects of $\KhFlowCat^*(L')$. This bijection respects the Khovanov
  differential, and hence induces an identification of the index $1$
  moduli spaces. The reflection exchanges the right and left pairs in
  each ladybug configuration, and thus induces an identification
  between index $2$ moduli spaces in $\KhFlowCat(L)$ and
  $\KhFlowCat(L')$.  Since the higher index moduli spaces are built
  inductively from these, it follows that the reflection induces the
  claimed isomorphism $\KhFlowCat(L)\cong \KhFlowCat^*(L')$.

  So, we have homotopy equivalences
  \[
  \Realize{\KhFlowCat(L)}\simeq \Realize{\KhFlowCat(L')}\cong
  \Realize{\KhFlowCat^*(L)}.
  \]
  Here, the first stable homotopy equivalence uses Reidemeister
  invariance of $\Realize{\KhFlowCat}$, and the second isomorphism
  uses the claim above. This proves the result.
\end{proof}

\begin{proof}[Proof of \Theorem{kh-space}]
  The proof of \Part{thm-refines} is straightforward. To be precise,
  choose the standard sign assignment $s_0$ in \Definition{Kh-space};
  then it follows from \Definition{knot-diag-to-kh-chain-complex} and
  \Definition{khovanov-flow-category} that the subcategory of the
  Khovanov flow category $\KhFlowCat(L)$ in quantum grading $j$
  refines $\KhCx^{*,j}(L)$. Hence,
  \Lemma{realization-well-defined-refines-chain-complex} implies that
  the reduced cohomology of $\KhSpace^j(L)$ is the Khovanov homology
  in quantum grading $j$.

  \Part{thm-invariant} follows from the previous
  propositions. \Proposition{choice-independent} implies that the
  stable homotopy type of $\KhSpace^j(L)$ depends only on the link diagram
  $L$, and not on the other choices made in the
  construction. Propositions~\ref{prop:RI},~\ref{prop:RII}
  and~\ref{prop:RIII} imply that, if $L$ and $L'$ are link diagrams
  for isotopic links then, for appropriate auxiliary choices,
  $\KhSpace^j(L)\simeq \KhSpace^j(L')$. Finally
  \Proposition{ladybug-invariance} tells us that this construction is
  independent of the choice of ladybug matching as well.
\end{proof}

\section{Skein sequence}\label{sec:skein}
Let $L$ be a link diagram with a distinguished crossing $c$, and let
$L_0$ and $L_1$ be the links obtained by performing the $0$- and
$1$-resolutions at $c$, respectively. Then for any choice of
orientations of $L$, $L_0$ and $L_1$ there is an exact triangle
on Khovanov homology
\[
\cdots\to\Kh^{i+a,j+b}(L_1)\to \Kh^{i,j}(L)\to \Kh^{i+c,j+d}(L_0)\to\Kh^{i+a+1,j+b}(L_1)\to\cdots.
\]
Here, $a$, $b$, $c$ and $d$ are integers which depend on how the
orientations were chosen.

\begin{theorem}
  With notation as above, there is a cofibration sequence
  \[
  \Sigma^{-c}\KhSpace^{j+d}(L_0)\into \KhSpace^{j}(L)\onto \Sigma^{-a}\KhSpace^{j+b}(L_1).
  \]
\end{theorem}
\begin{proof}
  This is immediate from \Lemma{down-sub-up-quot}: the flow category
  $\KhFlowCat(L_0)$ is downward closed subcategory of $\KhFlowCat(L)$
  with complementary upward closed
  subcategory $\KhFlowCat(L_1)$, so $\Realize{\KhFlowCat(L_0)}$ is a
  subcomplex of $\Realize{\KhFlowCat(L)}$ with quotient complex
  $\Realize{\KhFlowCat(L_1)}$.
\end{proof}

\section{Reduced version}\label{sec:reduced}
In~\cite[Section 3]{Kho-kh-patterns}, Khovanov introduced a reduced
version of Khovanov homology, whose Euler characteristic is the
ordinary Jones polynomial $V(L)$.  To define it, fix a
point $p\in L$, not at one of the crossings. For each resolution
$\AssRes{L}{v}$ of $L$, $p$ lies in some circle $Z_p(v)$. The
generators of $\KhCx(L)$ in which $Z_p(v)$ is labeled by $x_-$ span a
subcomplex $\rKhCx_-(L)$ of $\KhCx(L)$, and the generators of
$\KhCx(L)$ in which $Z_p(v)$ is labeled by $x_+$ generate the
corresponding quotient complex $\rKhCx_+(L)$. Moreover, it is easy to
see that the complexes $\rKhCx_-(L)$ and $\rKhCx_+(L)$ are isomorphic;
keeping track of gradings, the isomorphism identifies
$\rKhCx^{i,j-1}_-(L)$ and $\rKhCx^{i,j+1}_+(L)$. Let $\rKhCx(L)$
denote either of these complexes, graded so that
$\rKhCx^{i,j}(L)=\rKhCx^{i,j-1}_-(L)$. Then there are short exact
sequences of chain complexes
\[
0\to \rKhCx^{*,j+1}\to \KhCx^{*,j}\to \rKhCx^{*,j-1}\to 0,
\]
which induce long exact sequences
\begin{equation}
  \label{eq:rKh-and-Kh}
  \dots\to \rKh^{i,j+1}(L)\to \Kh^{i,j}(L)\to \rKh^{i,j-1}(L)\to \rKh^{i+1,j+1}(L)\to\cdots;
\end{equation}
compare~\cite[Proposition 4.3]{Ras-kh-survey}.

We can imitate this construction to produce a spectrification of the
reduced Khovanov homology. Again, fix $p\in L$, not at one of the crossings. Let $\rKhFlowCat_-(L)$
(\respectively $\rKhFlowCat_+(L)$) denote the full subcategory of
$\KhFlowCat(L)$ whose objects are labeled resolution configurations
$(\AssRes{L}{v},\gen{x})$ so that the labeling in $x$ of the circle
$Z_p(v)$ is $x_-$ (\respectively $x_+$); this is an upward closed
(\respectively downward closed) subcategory. As with the Khovanov chain
complex, it is straightforward to verify that there is an isomorphism
$\rKhFlowCat_-(L)\cong \rKhFlowCat_+(L)$, commuting with the functor
$\Forget\co \KhFlowCat(L)\to \CubeFlowCat(n)$. Let $\rKhFlowCat(L)$
denote either of these (isomorphic) categories. Define the gradings on
$\rKhFlowCat(L)$ to agree with the gradings on $\rKhCx(L)$; that is,
the quantum grading is shifted by $\mp1$ from the quantum grading on
$\rKhFlowCat_\pm(L)$, and the internal (homological) grading is
inherited from $\KhFlowCat(L)$. Let
\[
\rKhSpace(L)=\Sigma^{-C} \Realize{\rKhFlowCat(L)},
\]
where $C$ is chosen so that the cohomology of $\rKhSpace(L)$ agrees
with the Khovanov homology; compare \Definition{Kh-space}.

\begin{theorem} The spectrum $\rKhSpace(L)=\bigvee_j \rKhSpace^j(L)$
  satisfies the following properties:
  \begin{enumerate}
  \item\label{item:rKh-refines} The reduced cohomology of
    $\rKhSpace^j(L)$ is the reduced Khovanov homology $\rKh^{*,j}(L)$.
  \item\label{item:rKh-invariant} The stable homotopy type of
    $\rKhSpace^j(L)$ depends only on the isotopy class of $L$, and not
    on any of the auxiliary choices made in the construction.
  \item\label{item:rKh-sequence} There is a cofibration
    $\rKhSpace^{j-1}(L)\to \KhSpace^j(L)$ so that
    $\KhSpace^j(L)/\rKhSpace^{j-1}(L)\simeq \rKhSpace^{j+1}(L)$. These
    maps induce the long exact sequence in \Equation{rKh-and-Kh}.
  \end{enumerate}
\end{theorem}
\begin{proof}
  \Part[s]{rKh-refines} and~\PartCont{rKh-sequence} are immediate from
  the definitions; for \Part{rKh-sequence} see
  also \Section{sub-quot}. For \Part{rKh-invariant}, observe that for
  Reidemeister moves not crossing the marked point $p$, the stable homotopy
  equivalences in \Section{invariance} preserve the subcomplex
  $\Realize{\rKhFlowCat_+(L)}\subset\Realize{\rKhFlowCat(L)}$. Any two
  diagrams for isotopic pointed links can be connected by Reidemeister
  moves not crossing the basepoint $p$ and and isotopies of the
  diagrams in $S^2$ (but possibly passing over $\infty$): passing a
  strand over (or under) $p$ can be exchanged for taking the strand
  around $S^2$ the other way (compare~\cite[Section
  3]{Kho-kh-patterns}).
\end{proof}

\begin{remark}
  It is not the case that $\KhSpace^j(L)\cong
  \rKhSpace^{j-1}(L)\vee\rKhSpace^{j+1}(L)$; this can be seen already
  in the case of the trefoil (see \Example{kh-space-trefoil}), and
  follows from the fact that $\Kh^{i,j}(L)\not\cong
  \rKh^{i,j-1}(L)\oplus\rKh^{i,j+1}(L)$.
\end{remark}
\section{Examples}\label{sec:examples}
In this section, we start by computing the Khovanov spectra associated
with a one-crossing unknot and the Hopf link.  After the explicit
examples, we observe that the Khovanov spectrum
for any alternating link contains no more information than
the Khovanov homology of the link.

\subsection{An unknot}
Consider the $1$-crossing unknot $U$ shown in \Figure{unknot}. For
convenience, name the generators of the Khovanov complex (or
equivalently the labeled resolution configurations, or equivalently
the objects of the Khovanov flow category $\KhFlowCat(U)$, or
equivalently the cells of the Khovanov space $\Realize{\KhFlowCat(U)}$)
as follows:

\begin{figure}
  \centering
\captionsetup[subfloat]{width=0.4\textwidth}
\subfloat[A one-crossing unknot $U$.]{\label{fig:unknot-diagram}\makebox[0.4\textwidth][c]{\includegraphics{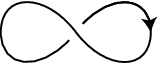}}}
\hspace{0.08\textwidth}
\subfloat[The $0$-resolution and the $1$-resolution.]{
  \includegraphics{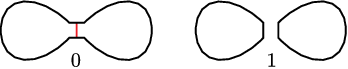}}
  \caption[A one-crossing unknot.]{\textbf{A $1$-crossing unknot $U$ and its
      resolutions.}}
  \label{fig:unknot}
\end{figure}

\begin{center}
  \begin{tabular}{clrrcclrr}
    \toprule
    Name & Generator & $\gr_h$ & $\gr_q$&\qquad\qquad & Name & Generator & $\gr_h$ & $\gr_q$ \\
    \midrule
    $\ob{a}$ & $(\AssRes{U}{1},x_+x_+)$ & $0$ & $1$ & &$\ob{d}$ & $(\AssRes{U}{1}, x_-x_-)$ & $0$ & $-3$\\
    $\ob{b}$ & $(\AssRes{U}{1},x_+x_-)$ & $0$ & $-1$ & & $\ob{x}$ & $(\AssRes{U}{0},x_+)$ & $-1$ & $-1$\\
    $\ob{c}$ & $(\AssRes{U}{1},x_-x_+)$ & $0$ & $-1$ & &$\ob{y}$ & $(\AssRes{U}{0},x_-)$ & $-1$ & $-3$\\
    \bottomrule
  \end{tabular}
\end{center}

The only $\prec$-comparable resolution configurations (\Definition{prec}) are:
\[
\ob{x}\prec \ob{b}\qquad \ob{x}\prec \ob{c} \qquad \ob{y}\prec \ob{d}.
\]
The moduli spaces $\Moduli(\ob{b},\ob{x})$, $\Moduli(\ob{c},\ob{x})$
and $\Moduli(\ob{d},\ob{y})$ each consist of a single,
positively-framed point. (The framing comes from the standard sign
assignment; see \Section{framing}).

To choose an embedding of this flow category as in
\Definition{neat-immersion}, choose $d_{-1}=1$ and $d_i=0$ for $i\neq
-1$. Then, embed the three points $\Moduli(\ob{b},\ob{x}) \amalg \Moduli(\ob{c},\ob{x})
\amalg \Moduli(\ob{d},\ob{y})$ in $\ESpace[\TupV{d}]{-1}{0}=\RR^1$; call the resulting
real numbers $b$, $c$ and $d$. These points inherit the
standard, positive framing. See \Figure{unknot-flow-cat}.

\captionsetup[subfloat]{width=0.35\textwidth}
\begin{figure}
  \centering
\subfloat[The Khovanov flow category $\KhFlowCat(U)$.]{\tiny
\psfrag{-1}{$-1$}
\psfrag{0}{$0$}
\psfrag{z1}{$\ob{y}$}
\psfrag{z2}{$\ob{x}$}
\psfrag{y1}{$\ob{b}$}
\psfrag{y2}{$\ob{c}$}
\psfrag{y3}{$\ob{d}$}
\psfrag{y4}{$\ob{a}$}
\psfrag{p}{$+$}
  \includegraphics[width=0.35\textwidth]{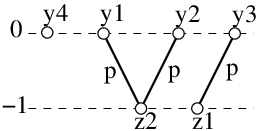}}
\hspace{0.1\textwidth}
\captionsetup[subfloat]{width=0.3\textwidth}
\subfloat[A framing of $\KhFlowCat(U)$.]{\tiny
\psfrag{b}{$b$}
\psfrag{c}{$c$}
\psfrag{d}{$d$}
\includegraphics[width=0.3\textwidth]{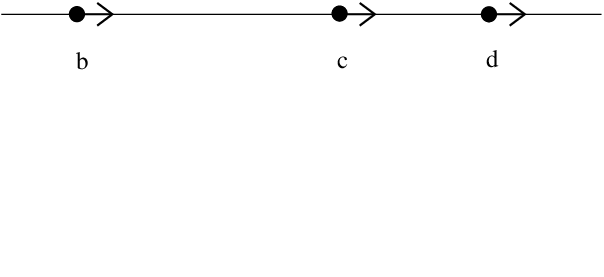}}\\
\subfloat[The cell $\Cell{\ob{x}}$ is the thick interval.]{\tiny
\psfrag{ep1}{$\{0\}\times[-\ep,\ep]$}
\includegraphics[width=0.3\textwidth]{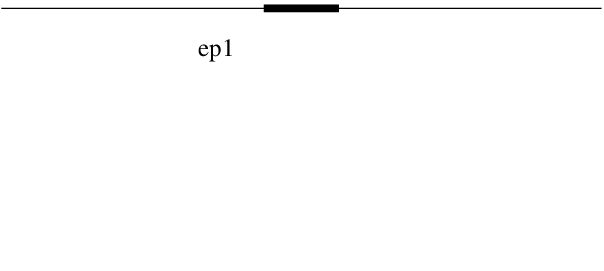}}
\hspace{0.025\textwidth}
\subfloat[The cell $\Cell{\ob{y}}$.]{\tiny
\psfrag{ep1}{$\{0\}\times[-\ep,\ep]$}
\includegraphics[width=0.3\textwidth]{unknot-x-space}}
\hspace{0.025\textwidth}
\subfloat[The cell $\Cell{\ob{a}}$ is the shaded rectangle.]{\tiny
\psfrag{R}{$[0,R]\times[-R,R]$}
\includegraphics[width=0.3\textwidth]{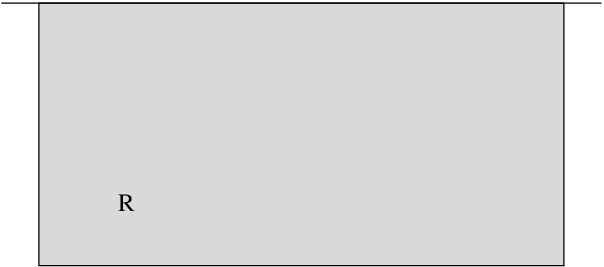}}\\
\subfloat[The cell $\Cell{\ob{b}}$. The thick interval on the boundary
is identified with $\Cell{\ob{x}}$.]{\tiny
  \psfrag{ep1}{$\{0\}\times[b-\ep,b+\ep]$}
  \psfrag{R}{$[0,R]\times[-R,R]$}
\includegraphics[width=0.3\textwidth]{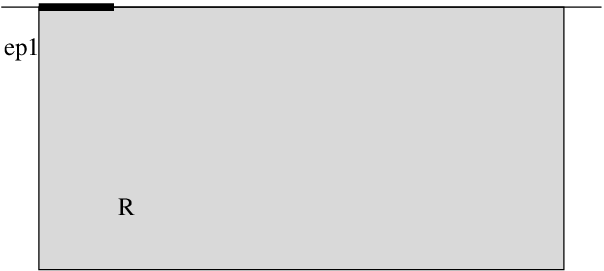}}
\hspace{0.025\textwidth}
\subfloat[The cell $\Cell{\ob{c}}$.]{\tiny
  \psfrag{ep1}{$\{0\}\times[c-\ep,c+\ep]$}
  \psfrag{R}{$[0,R]\times[-R,R]$}
\includegraphics[width=0.3\textwidth]{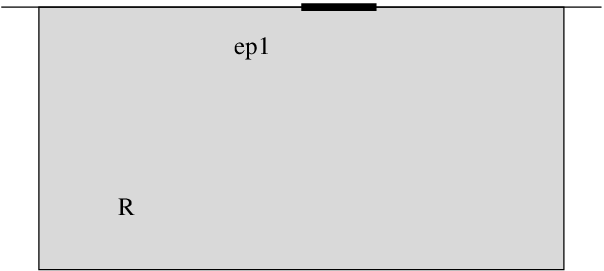}}
\hspace{0.025\textwidth}
\subfloat[The cell $\Cell{\ob{d}}$.]{\tiny
  \psfrag{ep1}{$\{0\}\times[d-\ep,d+\ep]$}
  \psfrag{R}{$[0,R]\times[-R,R]$}
\includegraphics[width=0.3\textwidth]{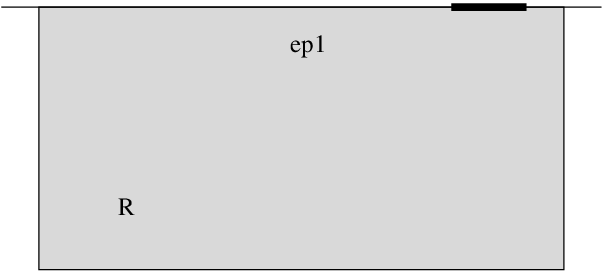}}
  \caption[The Khovanov space for the one-crossing
  unknot.]{\textbf{Constructing the Khovanov space for the
      $1$-crossing unknot $U$ of \Figure{unknot-diagram}.}}
  \label{fig:unknot-flow-cat}
\end{figure}

Next we turn to the construction of $\Realize{\KhFlowCat}$ from the
framed flow category (\Definition{flow-gives-space}). Take $B=-1$ and
$A=0$. The cells $\Cell{\ob{x}}$ and $\Cell{\ob{y}}$ are given by
\[
\bigl(\{0\}\times[-\epsilon,\epsilon]^{d_{-1}=1}\bigr)/\bdy\bigl(\{0\}\times[-\epsilon,\epsilon]^{d_{-1}=1}\bigr)\cong S^1.
\]
So, the $1$-skeleton of $\Realize{\KhFlowCat}$ is $S^1\vee S^1$. The
cell corresponding to $\ob{b}$ is
\[
\Cell{\ob{b}}= [0,R]\times [-R,R]^{d_{-1}=1} /\sim,
\]
where $\sim$ identifies $\{0\}\times [b-\epsilon,b+\epsilon]$ with
$\Cell{\ob{x}}$ and the rest of the boundary with the basepoint. So,
$\Cell{\ob{b}}$ consists of a $2$-cell whose boundary is attached by a
degree-one map to the first $S^1$-summand of $S^1\vee S^1$. The story
for $\Cell{\ob{c}}$ is exactly the same. The cell $\Cell{\ob{d}}$
consists of a $2$-cell whose boundary is attached by a degree-one
map to the second $S^1$-summand of $S^1\vee S^1$. Finally, all of the
boundary of $\Cell{\ob{a}}$ is collapsed to the basepoint. So,
\begin{align*}
  \Realize{\KhFlowCat}&=S^2\vee S^2\vee \DD^2,\\
  \shortintertext{where the first $S^2$ is
    $\Cell{\ob{b}}\displaystyle\fibercup_{\Cell{\ob{x}}}
    \Cell{\ob{c}}$; the second one is $\Cell{\ob{a}}$; and the $\DD^2$
    is $\Cell{\ob{d}}$. Hence,} \KhSpace&\simeq
  \Sigma^{-2}(S^2\vee S^2)\\
&= S^0_{-1}\vee S^0_1.
\end{align*}
The first $S^0$ summands lies in quantum grading $-1$ and the second in
quantum grading $1$. We have indicated this by the subscripts.

\subsection{The Hopf link}
We turn next to the Hopf link $H$, shown in \Figure{Hopf}. As before, for
convenience we name the generators:
\begin{center}
  \begin{tabular}{clrrcclrr}
    \toprule
    Name & Generator & $\gr_h$ & $\gr_q$ &\qquad\qquad& Name &
    Generator & $\gr_h$ & $\gr_q$\\
    \midrule
    $\ob{a}$ & $(\AssRes{H}{00},x_+x_+)$ & $0$ & $4$&& $\ob{c}$ & $(\AssRes{H}{01}, x_+)$ & $1$ &
    $4$\\  
    $\ob{v}$ & $(\AssRes{H}{00},x_+x_-)$ & $0$ & $2$&& $\ob{y}$ & $(\AssRes{H}{01}, x_-)$ & $1$ &
    $2$ \\ 
   $\ob{w}$ & $(\AssRes{H}{00},x_-x_+)$ & $0$ & $2$ &&$\ob{e}$ & $(\AssRes{H}{11}, x_-x_+)$ & $2$
   & $4$\\ 
    $\ob{m}$ & $(\AssRes{H}{00},x_-x_-)$ & $0$ & $0$ &&$\ob{z}$ & $(\AssRes{H}{11}, x_-x_-)$ &
    $2$ &$2$\\ 
    $\ob{b}$ & $(\AssRes{H}{10}, x_+)$ & $1$ & $4$ &&$\ob{l}$ & $(\AssRes{H}{11}, x_+x_+)$ & $2$ &$6$\\
    $\ob{x}$ & $(\AssRes{H}{10}, x_+)$ & $1$ & $2$ &&$\ob{d}$ & $(\AssRes{H}{11}, x_+x_-)$ & $2$ &$4$\\
    \bottomrule
  \end{tabular}
\end{center}

\begin{figure}
  \centering
  \captionsetup[subfloat]{width=0.35\textwidth}
  \subfloat[The Hopf link $H$.]{\makebox[0.35\textwidth][c]{\includegraphics{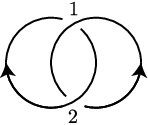}}}
  \hspace{0.03\textwidth}
  \captionsetup[subfloat]{width=0.45\textwidth}
  \subfloat[The resolutions of $H$.]{\includegraphics{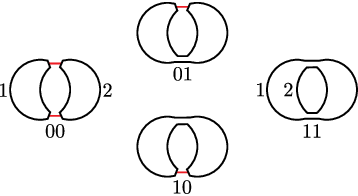}}
  \caption[The Hopf link.]{\textbf{The Hopf link $H$ and its resolutions.}}
  \label{fig:Hopf}
\end{figure}

To keep the notation simple, we will discuss each quantum grading
separately. In quantum grading $0$ (\respectively $6$) there is a
unique generator, and hence the Khovanov spectrum is a sphere in this
grading, of virtual dimension $0$ (\respectively $2$) after
de-suspending.

In quantum gradings $2$ and $4$ the Khovanov chain complex is given by
\[
\xymatrix{
\ob{v} \ar[r]\ar[ddr] & \ob{x}\ar[dr]^{-1} & & &  & \ob{b}\ar[r]^{-1}\ar[ddr]^(.65){-1} & \ob{d}\\
& & \ob{z} & \text{and} & \ob{a}\ar[ur]\ar[dr] & & \\
\ob{w}\ar[uur]\ar[r] & \ob{y}\ar[ur] & & & &  \ob{c}\ar[uur]\ar[r] & \ob{e},
}
\]
respectively. 

\captionsetup[subfloat]{width=0.45\textwidth}
\begin{figure}
  \centering
  \subfloat[Embedding of the subcategory in quantum grading $4$.]{\label{fig:Hopf-embedding-4}
    \tiny
    \psfrag{e1}{$\R^{d_0}$}
    \psfrag{e2}{$\R^{d_1}$}
    \psfrag{f}{}
    \psfrag{e}{$\Moduli(\ob{e},\ob{a})$}
    \psfrag{d}{$\Moduli(\ob{d},\ob{a})$}
    \psfrag{eb}{$\Moduli(\ob{e},\ob{b})$}
    \psfrag{ec}{$\Moduli(\ob{e},\ob{c})$}
    \psfrag{db}{$\Moduli(\ob{d},\ob{b})$}
    \psfrag{dc}{$\Moduli(\ob{d},\ob{c})$}
    \psfrag{b}{$\Moduli(\ob{b},\ob{a})$}
    \psfrag{c}{$\Moduli(\ob{c},\ob{a})$}
    \includegraphics[height=0.35\textwidth]{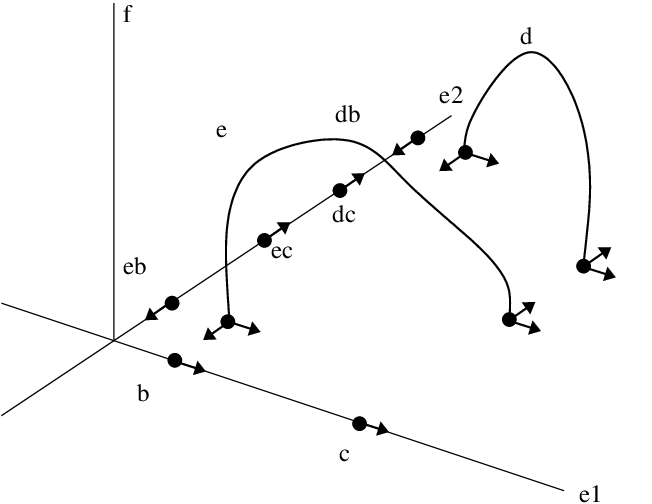}}
  \hspace{0.05\textwidth}
  \subfloat[\label{fig:emb-Hopf-2}Embedding of the subcategory in quantum grading $2$.]{
    \tiny
    \psfrag{e1}{$\R^{d_0}$}
    \psfrag{e2}{$\R^{d_1}$}
    \psfrag{f}{}
    \psfrag{v}{$\Moduli(\ob{z},\ob{v})$}
    \psfrag{w}{$\Moduli(\ob{z},\ob{w})$}
    \psfrag{x}{$\Moduli(\ob{z},\ob{x})$}
    \psfrag{y}{$\Moduli(\ob{z},\ob{y})$}
    \psfrag{vx}{$\Moduli(\ob{x},\ob{v})$}
    \psfrag{vy}{$\Moduli(\ob{y},\ob{v})$}
    \psfrag{wx}{$\Moduli(\ob{x},\ob{w})$}
    \psfrag{wy}{$\Moduli(\ob{y},\ob{w})$}
    \includegraphics[height=0.35\textwidth]{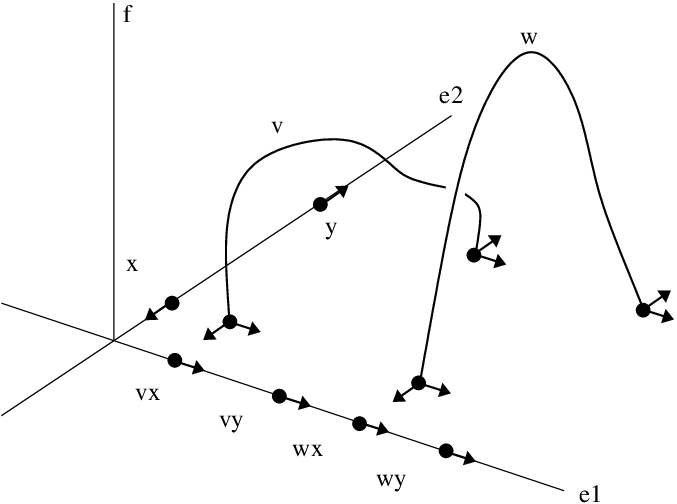}}
  \caption[Embedding of the flow category for the Hopf link.]{\textbf{Embedding of the flow category for the Hopf link.}}
  \label{fig:Hopf-embedding}
\end{figure}

Consider quantum grading $4$. Embed this flow category relative
$\TupV{d}$ where $d_0=1$, $d_1=1$, and $d_i=0$ for $i\notin\{0,1\}$, as
shown in \Figure{Hopf-embedding-4}. Choose $B=0$ and $A=2$. The cells of
this part of the flow category are given by 
\begin{align*}
\Cell{\ob{a}}&=\{0\}\times [-\epsilon,\epsilon]\times\{0\}\times[-\epsilon,\epsilon]/\sim_a\\
\Cell{\ob{b}}&=[0,R]\times [-R,R]\times\{0\}\times[-\epsilon,\epsilon]/\sim_b\\
\Cell{\ob{c}}&=[0,R]\times[-R,R]\times\{0\}\times[-\epsilon,\epsilon]/\sim_c\\
\Cell{\ob{d}}&=[0,R]\times [-R,R]\times[0,R]\times[-R,R]/\sim_d\\
\Cell{\ob{e}}&=[0,R]\times [-R,R]\times[0,R]\times[-R,R]/\sim_e.
\end{align*}
The identification $\sim_a$ identifies all of the boundary to the
basepoint. Each of the $3$-cells $\Cell{\ob{b}}$ and $\Cell{\ob{c}}$
has part of its boundary identified to $\Cell{\ob{a}}$, by a
degree-one map, and the rest collapsed to the basepoint; the picture
for this identification is exactly as in
\Figure{Cell-z-1}. The result of this attaching is two $3$-balls with
their boundaries glued together, i.e., a $3$-sphere.

Each of the $4$-cells $\Cell{\ob{d}}$ and $\Cell{\ob{e}}$ has part of
its boundary attached to $\Cell{\ob{b}}$ (\respectively $\Cell{\ob{c}}$) via
a degree $\pm 1$ map (according to the signs in the Khovanov complex),
part attached to $\Cell{\ob{a}}$, and the rest collapsed to the
basepoint; the picture for this identification is like
\Figure{Cell-x-bdy} but without the furthest back box (the one
corresponding to $\Cell{z_4}$). Up to homotopy, one can shrink the
thickened arc corresponding to $\Cell{\ob{a}}$ to become a
$2$-dimensional rectangle, leaving the boxes corresponding to
$\Cell{\ob{b}}$ and $\Cell{\ob{c}}$ glued together along a rectangle
on each of their boundaries. This rectangle is the $S^2$ along which
the $3$-balls $\Cell{\ob{b}}$ and $\Cell{\ob{c}}$ are glued. So,
gluing on $\Cell{\ob{d}}$ and $\Cell{\ob{e}}$ gives two $4$-balls
glued along their boundary. Thus, in quantum grading $4$, the Khovanov
space $\Realize{\KhFlowCat(H)}$ is homotopy equivalent to $S^4$. To obtain the
Khovanov spectrum, we de-suspend twice, so $\KhSpace^4\simeq S^2$.

Finally we turn to quantum grading $2$.  Each of $\Cell{\ob{v}}$ and
$\Cell{\ob{w}}$ are $2$-cells; $\Cell{\ob{x}}$ and $\Cell{\ob{y}}$ are
$3$-cells; and $\Cell{\ob{z}}$ is a $4$-cell. In this quantum grading,
the $2$-skeleton of the Khovanov space $\Realize{\KhFlowCat}$ is
$S^2\vee S^2$. The attaching map of $\Cell{\ob{x}}$ (\respectively
$\Cell{\ob{y}}$) is degree-one onto each $S^2$. (Each of
$\Cell{\ob{x}}$ and $\Cell{\ob{y}}$ is a box with two rectangles on
top, one identified with $\Cell{\ob{v}}$ and the other with
$\Cell{\ob{w}}$; compare \Figure{emb-Hopf-2}.) The result is the
suspension of the $2$-sphere with the north and south poles
identified, and is homotopy equivalent to $S^3\vee S^2$. The cell
$\Cell{\ob{z}}$ is attached by a map which is degree $-1$ to
$\Cell{\ob{x}}$ and degree $1$ to $\Cell{\ob{y}}$. To understand this
map better, consider the embedding of the moduli spaces from
\Figure{emb-Hopf-2}. This gives an embedding of $\Cell{\ob{x}}\cup
\Cell{\ob{y}}\cup([0,1]\times\Cell{\ob{v}})\cup([0,1]\times\Cell{\ob{w}})$
in $\del\Cell{\ob{z}}=S^3$. As in the previous paragraph, collapse the $[0,1]$ factors
to points. We are left with two $3$-balls, corresponding to
$\Cell{\ob{x}}$ and $\Cell{\ob{y}}$, glued together along two disks in
their boundaries. The complement of these two balls in
$S^3$ is an unknotted solid torus which is quotiented to the
basepoint. Therefore, the result of attaching $\Cell{\ob{z}}$
according to this embedding is $\DD^4$ with a solid torus on its
boundary quotiented to the basepoint. Thus, in quantum grading $4$,
the Khovanov space $\Realize{\KhFlowCat}$ is homotopy equivalent to
$S^2$. To obtain the Khovanov spectrum, we again de-suspend twice, so
$\KhSpace^2\simeq S^0$.

In sum,
\[
\KhSpace(H)=S^0_0\vee S^0_2\vee S^2_4\vee S^2_6.
\]

\subsection{The Khovanov spectrum of alternating links}
Generalizing the previous two examples, we show that for quasi-alternating links the Khovanov
spectrum is somewhat uninteresting. We start by collecting some facts
about Khovanov homology:
\begin{lem}\label{lem:quasi-alt-thin}
  If $L$ is a quasi-alternating link with signature\footnote{We use
    the convention in which a positive knot, such as
    the right-handed trefoil $T_{2,3}$, has negative signature.} $\sigma$ then:
  \begin{enumerate}
  \item\label{item:rKh-thin} The reduced Khovanov homology
    $\rKh^{i,j}(L)$ is supported on the single diagonal
    $2i-j=\sigma(K)$.
  \item\label{item:rKh-torsion} The reduced Khovanov homology is torsion-free.
  \item\label{item:Kh-thin} The un-reduced Khovanov homology is
    supported on the two diagonals $2i-j=\sigma\pm 1$.
  \item\label{item:Kh-torsion} All torsion in the un-reduced Khovanov
    homology lies on the diagonal $2i-j=\sigma+1$.
  \end{enumerate}
\end{lem}
\begin{proof}
  \Part{rKh-thin} is exactly~\cite[Theorem
  1.1]{MO-hf-quasialternating} (Note that their $j$ is one half of our $j$). See also~\cite{Lee-kh-endomorphism}
  in the alternating case.

  \Part{rKh-torsion} follows from the fact
  that~\cite[Theorem 1.1]{MO-hf-quasialternating} holds over the ring
  $\ZZ/p$ and the universal coefficient theorem. (This is not
  explicitly stated in~\cite{MO-hf-quasialternating}, but is clear
  from their proof.)
  
  To prove \Part{Kh-thin}, it follows from the exact sequence in
  \Equation{rKh-and-Kh} and \Part{rKh-thin} that the only cases in
  which $\Kh^{i,j}(L)$ is non-zero are when $j=2i+\sigma\pm
  1$. Further,
  \begin{align*}
    \Kh^{i,2i-\sigma-1}(L)&=\cokernel\bigl(\rKh^{i-1,2i-\sigma-2}(L)\to\rKh^{i,2i-\sigma}(L)
    \bigr)\\
    \Kh^{i,2i-\sigma+1}(L)&=\ker\bigl(\rKh^{i,2i-\sigma}(L)\to\rKh^{i+1,2i-\sigma+2}(L)
    \bigr).
  \end{align*}
  Combined with \Part{rKh-torsion}, this gives \Part{Kh-torsion}.
\end{proof}

\begin{prop}\label{prop:Moore}
  Let $L$ be a quasi-alternating link. Then $\KhSpace(L)$ and $\rKhSpace(L)$
  are wedge sums of Moore spaces.
\end{prop}
\begin{proof}
  We start with the reduced Khovanov homology.  Since $\rKhSpace(L)$
  is a wedge-sum over quantum gradings, it suffices to show that in
  each quantum grading $j$, $\rKhSpace^j(L)$ is a Moore space. But by
  \Lemma{quasi-alt-thin}, the reduced cohomology
  $\widetilde{H}^*(\rKhSpace^j(L))$ is nontrivial in a single grading,
  and so $\rKhSpace(L)$ is by definition a Moore space (and is, in
  fact, a wedge of spheres).

  Next, we turn to the unreduced story. Again, fix a quantum grading
  $j$. By \Lemma{quasi-alt-thin}, the reduced cohomology of
  $\KhSpace^j(L)$ is nontrivial in gradings $(j-1+\sigma)/2$ and
  $(j+1+\sigma)/2$, and the torsion lies in grading
  $(j+1+\sigma)/2$. For notational convenience, let
  $n=(j-1+\sigma)/2$.
  By the universal coefficient theorem, the
  reduced homology of $\KhSpace^j(L)$ is nontrivial in gradings
  $n$ and $n+1$, and the torsion lies in
  grading $n$. Let
  $G=H_{n}(\KhSpace^j(L))$ and
  $K=H_{n+1}(\KhSpace^j(L))$ (so $K$ is a free
  abelian group). Since the homology of $\KhSpace$ vanishes in degrees
  less than $n$,
  there is a map $f$ from the Moore space
  $M(G,n)$ to $\KhSpace$ inducing an isomorphism on
  $H_{n}$. Moreover, the map $\pi_{n+1}(\KhSpace^j(L))\to
  H_{n+1}(\KhSpace^j(L))$ is surjective (see, for
  instance,~\cite[Theorem 10.25]{Switzer-top-book-old} or~\cite[Exercise 4.2.23]{Hatcher-top-book}). So,
  there is a map $S^{n+1}\vee\dots\vee S^{n+1}=M(K,n+1)\to
  \KhSpace^j(L)$ inducing an isomorphism on $H_{n+1}$. Thus, we have
  obtained a map $M(G,n)\vee M(H,n+1)\to \KhSpace^j(L)$ inducing an
  isomorphism on homology; by Whitehead's theorem, this map is a
  homotopy equivalence.
\end{proof}

\begin{corollary}
  If $L$ is an alternating link then the Khovanov spectrum (reduced or
  un-reduced) is determined by its Khovanov homology. In particular,
  the reduced Khovanov spectrum of an alternating link $L$ is
  determined by the Jones polynomial and signature of $L$.
\end{corollary}

\begin{exam}\label{exam:kh-space-trefoil}
  The Khovanov homology of the left-handed trefoil $3_1$ is given by
  $\Kh^{-3,-9}=\Kh^{-2,-5}=\Kh^{0,-3}=\Kh^{0,-1}=\ZZ$ and
  $\Kh^{-2,-7}=\ZZ/2$; and the reduced Khovanov homology is given by
  $\rKh^{-3,-8}=\rKh^{-2,-6}=\rKh^{0,-2}=\Z$. Thus,
  \begin{align*}
  \KhSpace(3_1)&=\Sigma^{-3}S^0_{-9}\vee \Sigma^{-2}S^0_{-5}\vee
  S^0_{-3}\vee S^0_{-1}\vee \Sigma^{-4}\mathbb{R}\mathrm{P}^2_{-7}\\
  \shortintertext{and}
  \rKhSpace(3_1)&=\Sigma^{-3}S^0_{-8}\vee \Sigma^{-2}S^0_{-6}\vee S^0_{-2},
  \end{align*}
  where the subscripts denote the quantum gradings.
\end{exam}

\begin{rem}
  The proof of \Proposition{Moore} was completely independent
  of our construction of $\KhSpace$; the corresponding result holds
  for any Khovanov stable homotopy type of this form (whatever its
  provenance).
\end{rem}

\begin{rem}\label{rem:not-Moore}
  We show in~\cite{RS-steenrod} that for many non-alternating knots,
  such as $T_{3,4}$, the Khovanov space is not a wedge sum of Moore
  spaces.
\end{rem}

\section{Some structural conjectures}\label{sec:speculations}
We conclude with a few structural conjectures about the behavior of
the Khovanov homotopy type under some basic operations on knots.

\subsection{Mirrors}
Given a link $L$, let $\mirror{L}$ denote the mirror of $L$.

Given a spectrum $X$, let $\SWdual{X}$ denote the Spanier-Whitehead
dual (or $S$-dual) of $X$. That is, $\SWdual{X}$ is the dual of $X$ as
a module over the sphere spectrum. Equivalently, if $X$ is the
suspension spectrum of a finite CW complex (which we also denote $X$)
then $\SWdual{X}$ can be obtained by choosing an embedding $i\co X\hookrightarrow
S^N$ (for some large $N$) and formally desuspending the complement
$S^N\setminus i(X)$ $(N-1)$ times:
\[
\SWdual{X}=\Sigma^{1-N}(S^N\setminus i(X));
\]
see~\cite{Spanier-top-dual}.

\begin{conjecture}\label{conj:mirror}
  Let $L$ be a link. Then
  \[
  \KhSpace^j(m(L))\simeq\SWdual{\KhSpace^{-j}(L)}
  \]
\end{conjecture}

\begin{remark}
  The analogous result for Manolescu's Seiberg-Witten stable homotopy type,
  that $\mathrm{SWF}(-Y)$ is Spanier-Whitehead dual to
  $\mathrm{SWF}(Y)$, is observed in~\cite[Remark
  2]{Man-gauge-swspectrum}, using the corresponding result about
  Conley indices from~\cite{Cornea-top-dual}.
\end{remark}

\subsection{Disjoint unions and connected sums}
For Khovanov homology, disjoint unions correspond to tensor
products. The obvious analogue at the space level is the smash
product:
\begin{conjecture}\label{conj:disjoint-union}
  Let $L_1$ and $L_2$ be links, and $L_1\amalg L_2$ their disjoint union. Then
  \[
  \KhSpace^j(L_1\amalg L_2)\simeq \bigvee_{j_1+j_2=j}\KhSpace^{j_1}(L_1)\smas\KhSpace^{j_2}(L_2).
  \]
  Moreover, if we fix a basepoint $p$ in $L_1$, not at a crossing, and
  consider the corresponding basepoint for $L_1\amalg L_2$, then
  \[
  \rKhSpace^j(L_1\amalg L_2)\simeq \bigvee_{j_1+j_2=j}\rKhSpace^{j_1}(L_1)\smas\KhSpace^{j_2}(L_2).
  \]
\end{conjecture}

Similarly, reduced Khovanov homology should be well behaved for connected sums:
\begin{conjecture}\label{conj:connect-sum}
  Let $L_1$ and $L_2$ be based links and $L_1\# L_2$ the connected sum
  of $L_1$ and $L_2$, where we take the connected sum near the
  basepoints. Then
  \[
  \rKhSpace^j(L_1\# L_2)\simeq \bigvee_{j_1+j_2=j}\rKhSpace^{j_1}(L_1)\smas\rKhSpace^{j_2}(L_2).
  \]
\end{conjecture}

The unreduced Khovanov homology of a connected sum is slightly more
complicated. A choice of basepoint makes the Khovanov complex
$\KhCx(L)$ into a module over $\Kh(U)=\Z\langle x_+,x_-\rangle$. (The
multiplication on $\Kh(U)$ is defined as follows: $x_+$ acts as the
identity and $x_-^2=0$.)  At the level of complexes,
$\KhCx(L_1\#L_2)\cong
\KhCx(L_1)\otimes_{\Kh(U)}\KhCx(L_2)$~\cite[Proposition
3.3]{Kho-kh-patterns}. There is also an analogous
statement using comultiplication: $\Kh(U)$ is a coalgebra with
$\Delta(x_-)=x_-\otimes x_-$ and $\Delta(x_+)=x_+\otimes
x_-+x_-\otimes x_+$, and a choice of basepoint makes $\KhCx(L)$ into a
comodule over $\Kh(U)$.

\begin{lemma}\label{lem:conn-sum-comod}
  Let $L_1$ and $L_2$ be based links and $L_1\# L_2$ the connected sum
  of $L_1$ and $L_2$, where we take the connected sum near the
  basepoints. Then $\KhCx(L_1\# L_2)$ is the cotensor product over
  $\Kh(U)$ of $\KhCx(L_1)$ and $\KhCx(L_2)$:
  \[
  \KhCx(L_1\# L_2)=\KhCx(L_1)\Box_{\Kh(U)}\KhCx(L_2).
  \]
\end{lemma}
\begin{proof}
  Recall that the cotensor product $M\Box_R N$ is defined by the short
  exact sequence
  \[
  \xymatrix{
    0 \ar[r]& M\Box_R N\ar[r] & M\otimes
    N\ar[rrr]^-{\Delta_M\otimes\Id_N-\Id_M\otimes\Delta_N}&&& M\otimes R\otimes N.}
  \]
  see, for instance,~\cite[Section 1.10]{BW-other-corings}.
  In particular, $\KhCx(L_1)\Box_{\Kh(U)}\KhCx(L_2)$ is the subcomplex
  of $\KhCx(L_1)\otimes \KhCx(L_2)$ generated by:
  \begin{itemize}
  \item Elements $(\AssRes{L_1}{u_1},x_1)\otimes(\AssRes{L_2}{u_2},x_2)$ where
    $x_i$ labels the marked circle in $\AssRes{L_i}{u_i}$ by $x_-$, and
  \item Elements $(\AssRes{L_1}{u_1},x_1)\otimes(\AssRes{L_2}{u_2},y_2)+
    (\AssRes{L_1}{u_1},y_1)\otimes(\AssRes{L_2}{u_2},x_2)$ where $x_i$
    (respectively $y_i$) labels the marked circle in $D_{L_i}(u_i)$ by
    $x_-$ (respectively $x_+$).
  \end{itemize}
  The split map $\Delta\co \KhCx(L_1\#L_2)\to \KhCx(L_1)\otimes
  \KhCx(L_2)$ is injective, and its image is exactly
  $\KhCx(L_1)\Box_{\Kh(U)}\KhCx(L_2)$ as just described.
\end{proof}

We can lift this to the level of spaces as follows. The space
$\KhSpace(U)$ is $S^0\vee S^0$. Denote the basepoint in $\KhSpace(U)$
by $*$, and the other two points by $p_+$ and $p_-$. Then
\[
\KhSpace(U)\smas \KhSpace(U)=\{*\smas *, p_-\smas p_-, p_-\smas p_+, p_+\smas p_-, p_+\smas p_+\}.
\]
The comultiplication on $\Kh(U)$ dualizes to a basepoint-preserving
map
$\KhSpace(U)\smas \KhSpace(U)\to \KhSpace(U)$ given by
\begin{align*}
  p_-\smas p_-&\mapsto p_- & p_+\smas
  p_-&\mapsto p_+\\
  p_-\smas p_+&\mapsto p_+ & p_+\smas
  p_+&\mapsto *.
\end{align*}
This makes $\KhSpace(U)$ into a (discrete) ring spectrum.

Similarly, if $L$ is a based link then $\KhSpace(L)\smas
\KhSpace(U)\simeq \KhSpace(L)\vee \KhSpace(L)$, where the first
summand corresponds to $p_-$ (say) and the second to
$p_+$. There is a map 
\[
m\co \KhSpace(L)\smas \KhSpace(U)\simeq \KhSpace(L)\vee \KhSpace(L)\to \KhSpace(L)
\]
given as follows:
\begin{itemize}
\item On the first wedge summand $\KhSpace(L)$, $m$ is the identity map.
\item Recall from \Section{reduced} that $\Realize{\KhFlowCat(L)}$ has
  a subcomplex $\Realize{\rKhFlowCat_+(L)}$ and a quotient complex
  $\Realize{\rKhFlowCat_-(L)}=\Realize{\KhFlowCat(L)}/\Realize{\rKhFlowCat_+(L)}$, and there is a canonical
  identification
  $\Realize{\rKhFlowCat_+(L)}\cong\Realize{\rKhFlowCat_-(L)}$. On the
  second wedge summand of $\KhSpace(L)$, $m$ collapses
  $\Realize{\rKhFlowCat_+(L)}$ to the basepoint and takes
  $\Realize{\rKhFlowCat_-(L)}$ to $\Realize{\rKhFlowCat_+(L)}$ by the
  canonical identification.
\end{itemize}
This makes $\KhSpace(L)$ into a module spectrum over $\KhSpace(U)$.

We can now formulate a space-level version of \Lemma{conn-sum-comod}:
\begin{conjecture}\label{conj:unred-con-sum}
  Let $L_1$ and $L_2$ be based links. Then 
  \[
  \KhSpace(L_1\# L_2)\simeq \KhSpace(L_1)\otimes_{\KhSpace(U)}\KhSpace(L_2).
  \]
\end{conjecture}
We leave the gradings in \Conjecture{unred-con-sum} as an exercise to the reader.

\begin{remark}
  Note that a smash product of Moore spaces is typically not a Moore
  space. (For example, consider $\RR P^2\smas \RR P^2$.) So,
  \Conjecture{disjoint-union}, together with the computation of the
  Khovanov homology of the trefoil $T$, implies that $\KhSpace(T\amalg
  T)$ is not a wedge sum of Moore spaces. Similarly,
  \Conjecture{connect-sum} (respectively \Conjecture{unred-con-sum})
  implies that there are non-prime knots $K$ for with $\rKhSpace(K)$
  (respectively $\KhSpace(K)$) is not a wedge sum of Moore spaces.
\end{remark}

\bibliographystyle{myalpha}
\bibliography{newbibfile}

\end{document}